\documentclass[11pt]{amsart}
\usepackage{amsmath}
\usepackage{amssymb}
\usepackage{amsthm}
\usepackage{latexsym}
\usepackage{hyperref}
\usepackage{enumerate}
\usepackage[all]{xy}

\newtheorem{thm}{Theorem}[section]
\newtheorem{cor}[thm]{Corollary}
\newtheorem{lem}[thm]{Lemma}
\newtheorem{prop}[thm]{Proposition}

\newtheorem{thmintro}{Theorem}

\theoremstyle{definition}

\newtheorem{rem}[thm]{Remark}
\newtheorem{cond}[thm]{Condition}

\providecommand{\norm}[1]{\left\| #1 \right\|}
\newcommand{\enuma}[1]{\begin{enumerate}[\textup{(}a\textup{)}] {#1} \end{enumerate}}

\newcommand{\mh}{\mathbb}
\newcommand{\mr}{\mathrm}
\newcommand{\mc}{\mathcal}
\newcommand{\mf}{\mathfrak}

\newcommand{\ts}{\textstyle}

\newcommand{\isom}{\xrightarrow{\;\sim\;}}

\newcommand{\N}{\mathbb N}
\newcommand{\Z}{\mathbb Z}
\newcommand{\Q}{\mathbb Q}
\newcommand{\R}{\mathbb R}
\newcommand{\C}{\mathbb C}

\newcommand{\es}{\emptyset}
\newcommand{\af}{\mr{aff}}

\newcommand{\un}{\mr{un}}
\newcommand{\nr}{\mr{nr}}
\newcommand{\unr}{\mr{unr}}
\newcommand{\rs}{\mr{rs}}
\newcommand{\op}{\mr{op}}

\newcommand{\inp}[2]{\langle #1 \,,\, #2 \rangle}

\def\Hom{{\rm Hom}}
\def\End{{\rm End}}

\def\Irr{{\rm Irr}}
\def\Rep{{\rm Rep}}
\def\Mod{{\rm Mod}}
\def\Res{{\rm Res}}
  
\def\ind{{\rm ind}}

\def\fs{{\mathfrak s}}
\def\GL{{\rm GL}}
\def\SL{{\rm SL}}

\begin{document}

\title{On completions of Hecke algebras}
\date{\today}
\thanks{The author is supported by a NWO Vidi grant "A Hecke algebra approach to the 
local Langlands correspondence" (nr. 639.032.528).}
\maketitle

\begin{center}
{\Large Maarten Solleveld} \\[1mm]
IMAPP, Radboud Universiteit \\
Heyendaalseweg 135, 6525AJ Nijmegen, the Netherlands \\
email: m.solleveld@science.ru.nl \\
\end{center}

\vspace{4mm}

\begin{abstract}
Let $G$ be a reductive $p$-adic group and let $\mc H (G)^\fs$ be a Bernstein block of
the Hecke algebra of $G$. We consider two important topological completions of 
$\mc H (G)^\fs$: a direct summand $\mc S (G)^\fs$ of the Harish-Chandra--Schwartz
algebra of $G$ and a two-sided ideal $C_r^* (G)^\fs$ of the reduced $C^*$-algebra of $G$.
These are useful for the study of all tempered smooth $G$-representations.

We suppose that $\mc H (G)^\fs$ is Morita equivalent to an affine Hecke algebra
$\mc H (\mc R,q)$ -- as is known in many cases. The latter algebra also has a Schwartz
completion $\mc S (\mc R,q)$ and a $C^*$-completion $C_r^* (\mc R,q)$, both defined in
terms of the underlying root datum $\mc R$ and the parameters $q$. 

We prove that, under some mild conditions, a Morita equivalence $\mc H (G)^\fs \sim_M
\mc H (\mc R,q)$ extends to Morita equivalences $\mc S (G)^\fs \sim_M \mc S (\mc R,q)$
and $C_r^* (G)^\fs \sim_M C_r^* (\mc R,q)$. We also check that our conditions are 
fulfilled in all known cases of such Morita equivalences between Hecke algebras.
This is applied to compute the topological K-theory of the reduced $C^*$-algebra of
a classical $p$-adic group. \\[5mm]
\end{abstract}

\tableofcontents

\section*{Introduction}

Let $G$ be a connected reductive group over a non-archimedean local field.
Let $\Rep (G)$ be the category of smooth $G$-representations on complex vector spaces.
To study such representations, it is often useful to consider various group algebras
of $G$. Most fundamentally, there is the Hecke algebra $\mc H (G)$, the convolution 
algebra of locally constant, compactly supported functions $G \to \C$. The category
$\Rep (G)$ is equivalent to the category $\Mod (\mc H (G))$ of nondegenerate
$\mc H (G)$-modules. (Here $V$ nondegenerate means that $\mc H (G) \cdot V = V$.)

For purposes of harmonic analysis, and in particular for the study of tempered
smooth $G$-representations, the Harish-Chandra--Schwartz algebra $\mc S (G)$
\cite{HC} can be very convenient. This is a topological completion of $\mc H (G)$,
it consists of locally constant functions $G \to \C$ that decay rapidly in a specified
sense. By \cite[\S III.7]{Wal} an admissible smooth $G$-representation is tempered if
and only if it is naturally a $\mc S (G)$-module. For larger representations it is
best to define the category of tempered smooth $G$-representations as the category
$\Mod (\mc S (G))$ of nondegenerate $\mc S (G)$-modules \cite[Appendix]{SSZ}.

Further, from the point of view of operator algebras or noncommutative geometry,
the reduced $C^*$-algebra $C_r^* (G)$ may be the most suitable. The modules of
$C_r^* (G)$ are Banach spaces, so they are usually not smooth as $G$-representations.
But $\mc S (G) \subset C_r^* (G)$ and the smooth vectors in any $C_r^* (G)$-module
do form a nondegenerate $\mc S (G)$-module and hence a smooth $G$-representation. 
Moreover this operation provides a bijection between the irreducible representations of 
$C_r^* (G)$ and those of $\mc S (G)$. This feature distinguishes $C_r^* (G)$ from other 
Banach group algebras like $L^1 (G)$ or the maximal $C^*$-algebra of $G$.\\

Let $\Rep (G)^\fs$ be a Bernstein block of $\Rep (G)$ \cite{BeDe}. It is well-known
that in many cases (see Sections \ref{sec:types} and \ref{sec:progen}) $\Rep (G)^\fs$ 
is equivalent to the category of modules of an affine Hecke algebra $\mc H (\mc R,q)$. 
Here $\mc R$ is a root datum and $q$
is a parameter function for $\mc R$. In such cases it would be useful if one could
detect, in terms of $\mc H (\mc R,q)$, whether a $G$-representation in $\Rep (G)^\fs$
\begin{enumerate}\label{questions}
\item[(i)] is tempered;
\item[(ii)] is unitary;
\item[(iii)] admits a continuous extension to a $C_r^* (G)$-module.
\end{enumerate}
The structure needed to make sense of this is available for (extended) affine Hecke 
algebras with positive parameters. They have a natural *-operation, so unitarity is 
defined. Temperedness of finite-dimensional $\mc H (\mc R,q)$-modules can be defined 
conveniently either in terms of growth of matrix coefficient or by means of
weights for a large commutative subalgebra of $\mc H (\mc R,q)$ \cite[\S 2.7]{Opd-Sp}.

Furthermore there exists a Schwartz completion $\mc S (\mc R,q)$ of $\mc H (\mc R,q)$ 
\cite[\S 6.2]{Opd-Sp} with a similar structure as $\mc S (G)$ \cite{DeOp1}. 
By \cite[Corollary 6.7]{Opd-Sp} a finite dimensional $\mc H (\mc R,q)$-module is tempered
if and only if it extends continuously to a $\mc S (\mc R,q)$-module. Like for $G$, we
define the category of tempered $\mc H (\mc R,q)$-modules to be the module category
of $\mc S (\mc R,q)$.

The algebra $\mc H (\mc R,q)$ is a Hilbert algebra, so it acts by multiplication on
its own Hilbert space completion. Then one can define the reduced $C^*$-completion
$C_r^* (\mc R,q)$ as the closure of $\mc H (\mc R,q)$ in the algebra of bounded linear
operators on that Hilbert space. It is reasonable to expect that this algebra plays 
a role analogous to $C_r^* (G)$.

Let $\mc H (G)^\fs$ (resp. $\mc S (G)^\fs$ and $C_r^* (G)^\fs$) be the direct summand
of $\mc H (G)$ (resp. $\mc S (G)$ and $C_r^* (G)$) corresponding to $\Rep (G)^\fs$
via the Bernstein decomposition 
\[
\Mod (\mc H (G)) \cong \Rep (G) = \prod\nolimits_{\fs \in \mf B (G)} \Rep (G)^\fs . 
\]
In view of the above, it is natural to ask whether an equivalence of categories
\begin{equation}\label{eq:1}
\Rep (G)^\fs \cong \Mod (\mc H (G)^\fs ) \isom \Mod (\mc H (\mc R,q))
\end{equation}
extends to Morita equivalences
\begin{equation}\label{eq:2}
\mc S (G)^\fs \sim_M \mc S (\mc R,q) \quad \text{and} \quad
C_r^* (G)^\fs \sim_M C_r^* (\mc R,q) .
\end{equation}
That would solve the issues (i) and (iii) completely, and would provide a partial
answer to (ii). Namely, since irreducible tempered representations are unitary, it would 
imply that \eqref{eq:1} matches the unitary tempered
representations on both sides. (It is not clear what it could say about unitary
non-tempered representations.) Furthermore \eqref{eq:2} would make $\Mod (\mc S (G)^\fs)$
and $\Mod (C_r^* (G)^\fs)$ amenable to much more explicit calculations, in terms of
the generators and relations from $\mc H (\mc R,q)$. That is important for topological
K-theory, where one deals with finitely generated projective $C_r^* (G)$-modules.

While \eqref{eq:2} looks fairly plausible, it is not automatic. To prove it, we impose
some conditions on the Morita equivalence $\mc H (G)^\fs \sim_M \mc H (\mc R,q)$:
\begin{itemize}\label{cond:intro}
\item Condition \ref{cond:Morita} is about compatibility with parabolic induction
and restriction. 
\item Condition \ref{cond:Hecke} says roughly that under this Morita equivalence every
(suitable) parabolic subgroup of $G$ should give rise to a parabolic subalgebra of
$\mc H (\mc R,q)$, and this correspondence should respect positivity in the underlying
root systems.
\item Sometimes we obtain, instead of $\mc H (\mc R,q)$, an extended affine Hecke
algebra $\mc H (\mc R,q) \rtimes \Gamma$, where $\Gamma$ is a finite group. Then we
require Condition \ref{cond:Gamma}, which says that $\Gamma$ respects all the relevant
structure.
\end{itemize}
Here Condition \ref{cond:Morita} has little to do with affine Hecke algebras. If it holds,
then on based general principles. Condition \ref{cond:Hecke} is needed to get affine
Hecke algebras into play. If it does not hold, than our results simply cannot be
formulated in such terms. The positivity part is innocent, usually it can be achieved by
a good choice of a standard minimal parabolic subgroup of $G$. The Condition 
\ref{cond:Gamma} is of a more technical nature, it serves to rule out some phenomena that
could happen for arbitrary $\Gamma$ but not for reductive groups.

Notice that the above conditions do not say anything about *-homomorphisms between
$\mc H (\mc R,q)$ and $\mc H (G)^\fs$. Instead, our condition are chosen so that they will 
hold true for affine Hecke algebras arising from reductive $p$-adic groups via the two 
main methods: Bushnell--Kutzko types and Bernstein's progenerators.

\begin{thmintro}\label{thm:1} \textup{(see Theorem \ref{thm:4.2})} \\
Suppose that there is a Morita equivalence between $\mc H (G)^\fs$ and an extended affine 
Hecke algebra $\mc H (\mc R,q) \rtimes \Gamma$ with positive parameters, such that 
Conditions \ref{cond:Morita}, \ref{cond:Hecke} and \ref{cond:Gamma} hold. Then it induces 
Morita equivalences $\mc S (G)^\fs \sim_M \mc S (\mc R,q)$ and
$C_r^* (G)^\fs \sim_M C_r^* (G)$.
\end{thmintro}

Hitherto this was only proven for the Schwartz completions in the case of 
Iwahori--spherical representations of split groups \cite[Theorem 10.2]{DeOp1}. 
In all cases where a Morita equivalence on the Hecke algebra level is known (to the
author), we check that the conditions of Theorem \ref{thm:1} are fulfilled. This
includes principal series representations of $F$-split groups ($F$ is any 
non-archimedean local field), level zero representations, inner forms of $\GL_n (F)$, 
inner forms of $\SL_n (F)$, symplectic groups (not necessarily
split) and special orthogonal groups (also possibly non-split). 

In all these cases we obtain a pretty good picture of $C_r^* (G)^\fs$, up to Morita
equivalence. This can, for instance, be used to compute the topological K-theory of these
algebras. Indeed, in \cite{SolK} the author determined the K-theory of $C_r^* (\mc R,q)$
for many root data $\mc R$ (it does not depend on $q$). These calculations, together with
Theorem \ref{thm:1} for classical groups, lead to a result which is useful in relation
with the Baum--Connes conjecture.

\begin{thmintro}\label{thm:2} \textup{(see Theorem \ref{thm:6.3})} \\
Let $G$ be a special orthogonal or a symplectic group over $F$ (not necessarily split), 
or an inner form of $\GL_n (F)$. Then $K_* (C_r^* (G))$ is a free abelian group. 
For every Bernstein block $\Rep (G)^\fs$, the rank of $K_* (C_r (G)^\fs)$ 
is finite and can be computed explicitly.
\end{thmintro}

Let us also discuss other approaches to the topics (i), (ii) and (iii) from page
\pageref{questions}. In most cases where a Morita equivalence $\mc H (G)^\fs \sim_M 
\mc H (\mc R,q) \rtimes \Gamma$ is known, these issues are not discussed in the 
literature. When the Morita equivalence comes from a type $(K,\lambda)$ in the sense of 
Bushnell--Kutzko \cite{BuKu}, several relevant techniques are available. Let 
$e_\lambda \in \mc H (G)$ and $\mc H (G,J,\lambda) = \End_G (\ind_J^G \lambda)$ be 
the idempotent and the algebra associated to the type $(K,\lambda)$ in 
\cite[\S 2]{BuKu}. In this setting the Morita equivalence can be implemented
by injective algebra homomorphisms
\begin{equation}\label{eq:3}
\mc H (\mc R,q) \rtimes \Gamma \xrightarrow{\; \Upsilon_\lambda \;} \mc H (G,J,\lambda)
\to e_\lambda \mc H (G) e_\lambda \to \mc H (G)^\fs ,
\end{equation}
where we assume that $\Upsilon_\lambda$ is an isomorphism. It follows quickly from the 
definition of types that the last two maps in \eqref{eq:3} induce Morita equivalences
\cite[(2.12) and Theorem 4.3]{BuKu}. Both algebras $\mc H (G,J,\lambda)$ and 
$e_\lambda \mc H (G) e_\lambda$ are endowed with a natural trace and *, which
are preserved by the injections 
\begin{equation}\label{eq:4}
\mc H (G,J,\lambda) \to e_\lambda \mc H (G) e_\lambda \to \mc H (G)^\fs .
\end{equation}
Considerations with Hilbert algebras show that the maps \eqref{eq:4} induce equivalences
between the associated categories of finite length tempered representations \cite{BHK}.
Moreover, after the first version of this paper appeared, Ciubotaru showed that
\eqref{eq:4} also induces equivalences between the respective subcategories of
unitary modules \cite{Ciu}. Hence, whenever $\Upsilon_\lambda$ is a *-isomorphism, the
categories of unitary modules of all the algebras in \eqref{eq:3} are equivalent.
Earlier, this had been proven under certain additional conditions \cite{BaMo,BaCi}.

Notice that in \eqref{eq:3} all algebras are endowed with extra structure, which on the
left hand side comes from affine Hecke algebras and for the other terms from the 
embedding in $\mc H (G)$. In particular both $\mc H (G)$ and 
$\mc H (\mc R,q) \rtimes \Gamma$ are endowed with a canonical
trace, which stems from evaluation of functions at the unit element of $G$.
Usually \eqref{eq:3} will transfer the trace on $\mc H (G)^\fs$ to a positive
scalar multiple of the trace on $\mc H (\mc R,q) \rtimes \Gamma$. If that is the 
case and $\Upsilon_\lambda$ is a *-isomorphism, then by \cite[Theorem 10.1]{DeOp1} 
\eqref{eq:3} induces an equivalence between the category of finite length tempered 
$G$-representations in $\Rep (G)^\fs$ and the category of finite dimensional tempered
modules of $\mc H (\mc R,q) \rtimes \Gamma$. This relies on properties of the 
Plancherel measures of $G$ and of $\mc H (\mc R,q) \rtimes \Gamma$, and it uses that 
$\Upsilon_\lambda$ preserves these Plancherel measures, up to a scalar multiple.

When the Morita equivalence $\mc H (G)^\fs \sim_M \mc H (\mc R,q) \rtimes \Gamma$
does not arise from a type, fewer techniques for (i), (ii) and (iii) were known.
Heiermann established such Morita equivalences, for symplectic and special
orthogonal groups and for inner forms of $\GL_n (F)$ \cite{Hei}, by using Bernstein's
progenerators of $\Rep (G)^\fs$. In \cite{Hei2} he showed that these equivalences 
preserve temperedness of finite length modules. Here it is unknown whether the * 
and the trace are preserved by the Morita equivalence. Since maps like \eqref{eq:3} are
lacking, it is even unclear how such a statement could be formulated in this setting.

Summarizing, in the literature already several results about the behaviour of finite 
length modules under a Morita equivalence between a Bernstein block $\Rep (G)^\fs$ and 
the module category of an (extended) affine Hecke algebra can be found, but there is 
so far almost nothing about the Schwartz completions and the $C^*$-completions.\\

Let us briefly describe the contents of the paper. In the first section we recall the
definitions of affine Hecke algebras and their topological completions. We formulate
the Plancherel isomorphism for these completions, from \cite{DeOp1}, and we establish
suitable versions for affine Hecke algebras extended with finite groups. We also
ana\-lyse the space of irreducible representations and the subspace of irreducible 
tempered representations, mainly relying on \cite{SolAHA}. This is formulated in terms 
of the Langlands classification and induction from discrete series representations of
parabolic subalgebras.

After that we look at the aforementioned group algebras for a reductive $p$-adic group
$G$. We recall the Plancherel isomorphism for the Schwartz algebra of $G$ \cite{HC,Wal} 
and for the reduced $C^*$-algebra of $G$ \cite{Ply1}. Like for affine Hecke algebras,
we analyse the space of irreducible smooth $G$-representation in terms of the Langlands
classification and parabolic induction of square-integrable representations, 
following \cite{SolPadicHP}.

This forms the setup for the proof of our main result Theorem \ref{thm:1}, which 
occupies Section \ref{sec:comparison}. The crucial idea behind our argument is that in
the Plancherel isomorphisms for $\mc S (G)^\fs$ and $\mc S (\mc R,q)$ very similar
algebras appear. In both settings one encounters a bundle of matrix algebras over a 
compact torus, one takes $C^\infty$-sections of those, and then invariants with 
respect to a finite group acting via intertwining operators. We compare the resulting
algebras on both sides, analysing the data used to describe the Plancherel isomorphisms.
First we prove that a Morita equivalence between $\mc H (G)^\fs$ and $\mc H (\mc R,q)$, plus 
the mild extra conditions listed on page \pageref{cond:intro}, imply that the two necessary 
sets of data, for $\mc S (G)^\fs$ and for $\mc S (\mc R,q)$, become equivalent after some
manipulations. The most problematic part is to prove that the Morita equivalences preserve
temperedness and match square-integrable $G$-representations with discrete series 
$\mc H (\mc R,q)$-modules. For that we use very specific information about the spaces of
irreducible representations and their subspaces of tempered representations.
When we have compared all the data needed for the Plancherel isomorphisms on both sides,
we establish the desired Morita equivalences between topological algebras.

In Section \ref{sec:types} we check that the conditions from Section \ref{sec:comparison}
are fulfilled in (most) known cases of Morita equivalences coming from types. 
In the final section \ref{sec:progen} we do the same for Heiermann's Morita equivalences 
constructed with the use of projective generators, and we derive Theorem \ref{thm:2}.
\vspace{5mm}

\section{Affine Hecke algebras}
\label{sec:AHA}

Let $\mf a$ be a finite dimensional real vector space and let $\mf a^*$ be its dual. 
Let $Y \subset \mf a$ be a lattice and 
$X = \mr{Hom}_\Z (Y,\Z) \subset \mf a^*$ the dual lattice. Let
\begin{equation}\label{eq:1.37}
\mc R = (X, R, Y ,R^\vee ,\Delta) .
\end{equation}
be a based root datum. Thus $R$ is a reduced root system in $X ,\, R^\vee \subset Y$ 
is the dual root system, $\Delta$ is a basis of $R$ and the set of positive roots is denoted 
$R^+$. Furthermore a bijection $R \to R^\vee ,\: \alpha \mapsto \alpha^\vee$ is given, such 
that $\inp{\alpha}{\alpha^\vee} = 2$ and such that the corresponding reflections
$s_\alpha : X \to X$ (resp. $s^\vee_\alpha : Y \to Y$) stabilize $R$ (resp. $R^\vee$).
We do not assume that $R$ spans $\mf a^*$. The reflections $s_\alpha$ generate the Weyl group 
$W = W (R)$ of $R$, and 
$S_\Delta := \{ s_\alpha \mid \alpha \in \Delta \}$ is the collection of simple reflections. 

We have the affine Weyl group $W^\af = \mh Z R \rtimes W$ 
and the extended (affine) Weyl group $W^e = X \rtimes W$. Both can be considered as groups
of affine transformations of $\mf a^*$. We denote the translation corresponding to $x \in X$ by 
$t_x$. As is well-known, $W^\af$ is a Coxeter group, and the basis $\Delta$ of $R$ gives rise to 
a set $S^\af$ of simple (affine) reflections. More explicitly, let $\Delta_M^\vee$ be the set of 
maximal elements of $R^\vee$, with respect to the dominance ordering coming from $\Delta$. Then
\[
S^\af = S_\Delta \cup \{ t_\alpha s_\alpha \mid \alpha^\vee \in \Delta_M^\vee \} .
\]
The length function $\ell$ of the Coxeter system $(W^\af ,S^\af )$ extends naturally to $W^e$.
The elements of length zero form a subgroup $\Omega \subset W^e$ and $W^e = W^\af \rtimes \Omega$.

A complex parameter function for $\mc R$ is a map $q : S^\af \to \mh C^\times$ such that 
$q(s) = q(s')$ if $s$ and $s'$ are conjugate in $W^e$. This extends naturally to a map 
$q : W^e \to \C^\times$ which is 1 on $\Omega$ and satisfies 
\[
q(w w') = q(w) q(w') \quad \text{if} \quad \ell (w w') = \ell (w) + \ell (w').
\]
Equivalently (see \cite[\S 3.1]{Lus-Gr}) one can define $q$ as a $W$-invariant function
\begin{equation}\label{eq:1.21}
q : R \cup \{ 2 \alpha : \alpha^\vee \in 2 Y \} \to \C^\times . 
\end{equation}
We speak of equal parameters if $q(s) = q(s') \; \forall s,s' \in S^\af$ and of positive 
parameters if $q(s) \in \R_{>0} \; \forall s \in S^\af$. 
We fix a square root $q^{1/2} : S^\af \to \mh C^\times$.

The affine Hecke algebra $\mc H = \mc H (\mc R ,q)$ is the unique associative
complex algebra with basis $\{ N_w \mid w \in W^e \}$ and multiplication rules
\begin{equation}\label{eq:multrules}
\begin{array}{lll}
N_w \, N_{w'} = N_{w w'} & \mr{if} & \ell (w w') = \ell (w) + \ell (w') \,, \\
\big( N_s - q(s)^{1/2} \big) \big( N_s + q(s)^{-1/2} \big) = 0 & \mr{if} & s \in S^\af .
\end{array}
\end{equation}
In the literature one also finds this algebra defined in terms of the
elements $q(s)^{1/2} N_s$, in which case the multiplication can be described without
square roots. This explains why $q^{1/2}$ does not appear in the notation $\mc H (\mc R ,q)$.
For $q = 1$ \eqref{eq:multrules} just reflects the defining relations of $W^e$, so
$\mc H (\mc R,1) = \C [W^e]$.

The set of dominant elements in $X$ is
\[
X^+ = \{ x \in X : \inp{x}{\alpha^\vee} \geq 0 \; \forall \alpha \in \Delta \} .
\]
The subset $\{ N_{t_x}  : x \in X^+ \} \subset \mc H (\mc R,q)$ is closed under 
multiplication, and isomorphic to $X^+$ as a semigroup. For any $x \in X$ we put
\[
\theta_x = N_{t_{x_1}} N_{t_{x_2}}^{-1} \text{, where } 
x_1 ,x_2 \in X^+ \text{ and } x = x_1 - x_2 . 
\]
This does not depend on the choice of $x_1$ and $x_2$, so $\theta_x \in \mc H (\mc R,q)^\times$
is well-defined. The Bernstein presentation of $\mc H (\mc R,q)$ \cite[\S 3]{Lus-Gr} 
says that:
\begin{itemize}
\item $\{ \theta_x : x \in X \}$ forms a $\C$-basis of a subalgebra of $\mc H (\mc R,q)$
isomorphic to $\C[X] \cong \mc O (T)$, which we identify with $\mc O (T)$.
\item $\mc H (W,q) := \C \{ N_w : w \in W \}$ is a finite dimensional subalgebra of 
$\mc H (\mc R,q)$ (known as the Iwahori--Hecke algebra of $W$).
\item The multiplication map $\mc O (T) \otimes \mc H (W,q) \to \mc H (\mc R,q)$ 
is a $\C$-linear bijection.
\item There are explicit cross relations between $\mc H (W,q)$ and $\mc O (T)$, deformations
of the standard action of $W$ on $\mc O (T)$.
\end{itemize}
To define parabolic subalgebras of affine Hecke algebras, we associate some objects
to any set of simple roots $Q \subset \Delta$. Let $R_Q$ be the root system they generate,
$R_Q^\vee$ the root system generated by $Q^\vee$ and $W_Q$ their Weyl group. We also define
\[
\begin{array}{l@{\qquad}l}
X_Q = X \big/ \big( X \cap (Q^\vee )^\perp \big) &
X^Q = X / (X \cap \mh Q Q ) , \\
Y_Q = Y \cap \mh Q Q^\vee & Y^Q = Y \cap Q^\perp , \\
T_Q = \mr{Hom}_{\mh Z} (X_Q, \mh C^\times ) &
T^Q = \mr{Hom}_{\mh Z} (X^Q, \mh C^\times ) , \\
\mf a_Q = Y_Q \otimes_\Z \R & \mf a^Q = Y^Q \otimes_\Z R , \\
\mc R_Q = ( X_Q ,R_Q ,Y_Q ,R_Q^\vee ,Q) & \mc R^Q = (X,R_Q ,Y,R_Q^\vee ,Q) , \\
\mc H_Q = \mc H (\mc R_Q,q_Q) & \mc H^Q = \mc H (\mc R^Q ,q^Q) .
\end{array}
\]
Here $q_Q$ and $q^Q$ are derived from $q$ via \eqref{eq:1.21}. Both $\mc H_Q$ and $\mc H^Q$
are called parabolic subalgebras of $\mc H$. 
The quotient map $X \mapsto X_Q$ yields a natural projection
\begin{equation}\label{eq:1.1}
\mc H^Q \to \mc H_Q \;:\; \theta_x N_w \mapsto \theta_{x_Q} N_w . 
\end{equation}
In this way one can regard $\mc H_Q$ as a ``semisimple" quotient of $\mc H^Q$.
The algebra $\mc H^Q$ is embedded in $\mc H$ via the Bernstein presentation, as the image
of $\mc O (T) \otimes \mc H (W_Q,q) \to \mc H$.
Any $t \in T^Q$ and any $u \in T^Q \cap T_Q$ give rise to algebra automorphisms
\begin{equation}\label{eq:1.23}
\begin{array}{llcl}
\psi_u : \mc H_Q \to \mc H_Q , & \theta_{x_Q} N_w & \mapsto & u (x_Q) \theta_{x_Q} N_w , \\
\psi_t : \mc H^Q \to \mc H^Q , & \theta_x N_w & \mapsto & t(x) \theta_x N_w  .
\end{array}
\end{equation}
Let $\Gamma$ be a finite group acting on $\mc R$, i.e. it acts $\Z$-linearly on $X$ and
preserves $R$ and $\Delta$. We also assume that $\Gamma$ acts on $T$ by affine 
transformations, whose linear part comes from the action on $X$. Thus $\Gamma$ acts on
$\mc O (T) \cong \C [X]$ by
\begin{equation}\label{eq:1.38}
\gamma (\theta_x) = z_\gamma (x) \theta_{\gamma x} ,  
\end{equation}
for some $z_\gamma \in T$. 
We suppose throughout that $q^{1/2}$ is $\Gamma$-invariant, and
that $\Gamma$ acts on $\mc H (\mc R,q)$ by the algebra automorphisms
\begin{equation}\label{eq:1.39}
\text{Ad}(\gamma) : \sum_{w \in W, x \in X} c_{x,w} \theta_x N_w \; \mapsto \; 
\sum_{w \in W, x \in X} c_{x,w} z_\gamma (x) \theta_{\gamma (x)} N_{\gamma w \gamma^{-1}} . 
\end{equation}
This being a group action, the multiplication relations in $\mc H (\mc R,q)$ imply that 
we must have $z_\gamma \in T^W$. We build the crossed product algebra 
\begin{equation}\label{eq:1.22}
\mc H (\mc R,q) \rtimes \Gamma. 
\end{equation}
In \cite{SolAHA} we considered a slightly less general action of $\Gamma$ on $\mc H (\mc R,q)$,
where the elements $z_\gamma \in T^W$ from \eqref{eq:1.38} were all equal to 1. But 
the relevant results from \cite{SolAHA}  do not rely on $\Gamma$ fixing the unit element of $T$,
so they are also valid for the actions as in \eqref{eq:1.39}. In this paper we will tacitly
use some results from \cite{SolAHA} in the generality of \eqref{eq:1.39}. We note that
nontrivial $z_\gamma \in T^W$ are sometimes needed to describe Hecke algebras coming from
$p$-adic groups, for example in \cite[\S 4]{Roc2}.

Since $\mc H (\mc R,q)$ is of finite rank as a module over its commutative subalgebra
$\mc O (T)$, all irreducible $\mc H (\mc R,q)$-modules have finite dimension. 
The set of $\mc O (T)$-weights of a $\mc H (\mc R,q)$-module $V$ will be denoted by Wt$(V)$. 

We regard $\mf t = \mf a \oplus i \mf a$ as the polar decomposition of $\mf t$, with 
associated real part map $\Re : \mf t \to \mf a$.
The vector space $\mf t$ can be interpreted as the Lie algebra
of the complex torus $T = \Hom_\Z (X,\C^\times)$. The latter has a polar decomposition 
$T = T_\rs \times T_\un$ where $T_\rs = \Hom_\Z (X,\R_{>0})$ and $T_\un = \Hom_\Z (X, S^1)$ 
is the unique maximal compact subgroup of $T$. The polar decomposition of an element $t \in T$ is
written as $t = |t| \, (t \, |t|^{-1})$. The exponential map $\exp : \mf t \to T$ becomes
bijective when restricted to $\mf a \to T_\rs$. We denote its inverse by $\log : T_\rs \to \mf a$.

We write
\begin{align*}
& \mf a^+ = \{ \mu \in \mf a : \inp{\alpha}{\mu} \geq 0 \: \forall \alpha \in \Delta \} , \\
& \mf a^{*+} =  \{ x \in \mf a^* :  \inp{x}{\alpha^\vee} \geq 0 
\: \forall \alpha \in \Delta \} , \\
& \mf a^- = \{ \lambda \in \mf a : \inp{x}{\lambda} \leq 0 \: \forall x \in \mf a^{*+} \} = 
\big\{ \sum\nolimits_{\alpha \in \Delta} \lambda_\alpha \alpha^\vee : \lambda_\alpha \leq 0 \big\} .
\end{align*}
The interior $\mf a^{--}$ of $\mf a^-$ equals
$\big\{ {\ts \sum_{\alpha \in \Delta}} \lambda_\alpha \alpha^\vee : \lambda_\alpha < 0 \big\}$
if $\Delta$ spans $\mf a^*$, and is empty otherwise. We write 
\[
T^- = \exp (\mf a^-) \subset T_\rs \quad \text{and} \quad 
T^{--} = \exp (\mf a^{--}) \subset T_\rs.
\] 
We say that a module $V$ for $\mc H (\mc R,q)$ (or for $\mc H (\mc R,q) \rtimes \Gamma$) 
is tempered if $|\mr{Wt}(V)| \subset T^-$, and that it is discrete series if
$|\mr{Wt}(V)| \subset T^{--}$. The latter is only possible if $R$ spans $\mf a$, for
otherwise $\mf a^{--}$ and $T^{--}$ are empty. We alleviate these notions by calling a 
$\mc H \rtimes \Gamma$-module essentially discrete series if its 
restriction to $\mc H_\Delta$ is discrete series. Equivalently, essentially 
discrete series means that $\mr{Wt}(V) \subset T^{--} T_\un T^\Delta$. Such a
representation is tempered if and only if $\mr{Wt}(V) \subset T^{--} T_\un$. We denote the set
of (equivalence classes of) irreducible tempered essentially discrete series representations by
$\Irr_{L^2}(\mc H (\mc R,q) \rtimes \Gamma)$.\\

It follows from the Bernstein presentation \cite[\S 3]{Lus-Gr} that 
\begin{equation}\label{eq:1.11}
\text{the centre of } \mc H (\mc R,q) \rtimes \Gamma \text{ contains }
\mc O (T)^{W \Gamma} = \mc O (T / W \Gamma),
\end{equation}
with equality if $W \Gamma$ acts faithfully on $T$. By Schur's Lemma $\mc O (T)^{W \Gamma}$ 
acts on every irreducible $\mc H \rtimes \Gamma$-representation $\pi$ by a character. 
Such a character can be identified with a $W \Gamma$-orbit $W \Gamma t \subset T$. 
We will just call $W \Gamma t$ the central character of $\pi$. Then $W \Gamma |t| \subset T_\rs$
and $cc (\pi) := W \Gamma \log |t|$ is a single $W \Gamma$-orbit in $\mf a$. We fix a 
$W \Gamma$-invariant inner product on $\mf a$ and we define 
\begin{equation}\label{eq:1.cc}
\norm{cc (\pi)} = \norm{\log |t|}.
\end{equation}
In some cases that we will encounter, the appropriate parabolic subalgebras of
$\mc H (\mc R,q) \rtimes \Gamma$ are not $\mc H (\mc R^Q,q^Q)$, but $\mc H (\mc R^Q,q^Q) \rtimes
\Gamma_Q$ for some subgroup $\Gamma_Q \subset \Gamma$. To make this work well, we need some 
assumptions on the groups $\Gamma_Q$ for $Q \subset \Delta$.
\begin{cond}\label{cond:Gamma}
\begin{enumerate}
\item $\Gamma_Q \subset \Gamma_{Q'}$ if $Q \subset Q'$;
\item the action of $\Gamma_Q$ on $T$ stabilizes $T_Q$ and $T^Q$;
\item $\Gamma_Q$ acts on $T^Q$ by multiplication with elements of $K_Q$.
\end{enumerate}
\end{cond}
Notice that $\Gamma_Q$ is a subgroup of $\Gamma (Q,Q) = \{\gamma \in \Gamma : \gamma (Q) = Q \}$, 
but that we do not require these two groups to be equal.

We also note that Conditions \ref{cond:Gamma} entail that $\Gamma_\emptyset$ acts trivially on 
$\mc O (T) = \mc H^\emptyset$, so
\begin{equation}\label{eq:2.5}
\Irr (\mc H^\emptyset \rtimes \Gamma_\emptyset) \cong T \times \Irr (\Gamma_\emptyset) .
\end{equation}
\begin{rem}\label{rem:Gamma}
Often there is a larger root system $\tilde R \supset R$ in $X$, such that $W_Q \Gamma_Q$
is contained in the parabolic subgroup of $W(\tilde R)$ associated to $\tilde R \cap \Q Q$.
Then parts (2) and (3) of Condition \ref{cond:Gamma} are automatically satisfied 
(and part (1) is usually obvious).
\end{rem}

Under Conditions \ref{cond:Gamma} $\Gamma_Q$ commutes with $K_Q$, The conditions also
entail that the projection
$\mc H^Q \to \mc H_Q$ and the isomorphisms $\phi_t : \mc H^Q \to \mc H^Q \; (t \in T^Q)$ are 
$\Gamma_Q$-equivariant, so they extend to algebra homomorphisms
\begin{equation}\label{eq:1.20}
\mc H^Q \rtimes \Gamma_Q \to \mc H_Q \rtimes \Gamma_Q \quad \text{and} \quad
\phi_t : \mc H^Q \rtimes \Gamma_Q \to \mc H^Q \rtimes \Gamma_Q \; (t \in T^Q) .
\end{equation}
Via the first map of \eqref{eq:1.20} we can inflate any representation of $\mc H_Q \rtimes 
\Gamma_Q$ to $\mc H^Q \rtimes \Gamma_Q$, which we often do tacitly. For any representation
$\pi$ of $\mc H^Q \rtimes \Gamma_Q$ and any $t \in T^Q$ we write
\[
\pi \otimes t = \pi \circ \phi_t \in \Mod (\mc H^Q \rtimes \Gamma_Q) .
\]

\begin{lem}\label{lem:2.2}
\enuma{
\item Every irreducible $\mc H^Q \rtimes \Gamma_Q$-representation is of the form 
$\pi_Q \otimes t^Q$ for some $\pi_Q \in \Irr (\mc H_Q \rtimes \Gamma_Q)$ and $t^Q \in T^Q$.
\item $\pi_Q \otimes t^Q$ is tempered if and only if $\pi$ is tempered and $t^Q \in T^Q_\un$.
\item $\pi_Q \otimes t^Q$ is essentially discrete 
series if and only if $\pi_Q$ is discrete series.
} 
\end{lem}
\begin{proof}
(a) First we consider the situation without $\Gamma_Q$. Let $\pi \in \Irr (\mc H^Q)$ with central 
character $W_Q t \in T / W_Q$. The group $W_Q = W(R_Q)$ acts trivially on $T^Q$, so 
$W_Q t = t^Q W_Q t_Q$ for some $t^Q \in T^Q, t_Q \in T_Q$. Then $\pi \otimes (t^Q)^{-1}$ 
factors through $\mc H^Q \to \mc H_Q$  (say as $\pi_Q$), and $\pi = \pi_Q \otimes t^Q$. 

To include $\Gamma_Q$ we use Clifford theory \cite[Theorem A.6]{RaRa}. It says that every 
irreducible $\mc H^Q \rtimes \Gamma_Q$-representation is of the form 
\[
\pi \rtimes \rho := \ind_{\mc H^Q \rtimes \Gamma_{Q,\pi}}^{\mc H^Q \rtimes \Gamma_Q} 
(\pi \otimes \rho) .
\]
Here $\Gamma_{Q,\pi}$ is the stabilizer of $\pi \in \Irr (\mc H^Q)$ in $\Gamma_Q$ and 
$(\rho, V_\rho)$ is an irreducible representation of a twisted group algebra of $\Gamma_{Q,\pi}$.
If $\mc O (T)$ acts by $t^Q t_1$ on a vector subspace $V_1 \subset V_\pi$, then for 
$\gamma \in \Gamma_Q$ it acts by the character $\gamma^{-1} (t^Q t_1)$ on 
$N_\gamma (V_1 \otimes V_\rho)$. By Condition \ref{cond:Gamma}.(3) $\gamma^{-1}(t^Q t_1) \in
t^Q K_Q \gamma^{-1}(t_1)$. Hence $(\pi \rtimes \rho ) \otimes (t^Q)^{-1}$ factors through
$\mc H^Q \rtimes \Gamma_Q \to \mc H_Q \rtimes \Gamma_Q$ as $\pi_Q \rtimes \rho$, and
\[
\pi \rtimes \rho = (\pi_Q \rtimes \rho ) \otimes t^Q .
\]
(b) As $T = T^Q T_Q$ with $T^Q \cap T_Q \subset T_\un$, there is a factorization 
$T_{rs} = T^Q_{rs} \times T_{Q,rs}$ and (with respect to $\mc R^Q$) $T^{-}_{rs} = 
\{1\} \times T^{-}_{Q,rs}$. Also $|\mr{Wt}(\pi \otimes t^Q)| = |t^Q| \, |\mr{Wt}(\pi)|$. 
These observations imply the result.\\
(c) This is obvious from Wt$(\pi \otimes t^Q) = t^Q \mr{Wt}(\pi)$.
\end{proof}

\subsection{The Schwartz and $C^*$-completions} \

To get nice completions of $\mc H (\mc R,q)$ we assume from now on that $q$ is a positive 
parameter function for $\mc R$. As a topological vector space the Schwartz completion of 
$\mc H (\mc R,q)$ will consist of rapidly decreasing functions on $W^e$, with respect to a suitable 
length function $\mc N$. For example we can take a $W$-invariant norm on $X \otimes_\Z \R$ and put
$\mc N (w t_x) = \norm{x}$ for $w \in W$ and $x \in X$. Then we can define, for $n \in \N$, 
the following norm on $\mc H$:
\[
p_n \big( \sum\nolimits_{w \in W^e} h_w N_w \big) = 
\sup\nolimits_{w \in W^e} |h_w| (\mc N (w) + 1)^n .
\]
The completion of $\mc H$ with respect to these norms is the Schwartz algebra 
$\mc S = \mc S (\mc R,q)$. It is known from \cite[Section 6.2]{Opd-Sp} that it is a 
Fr\'echet algebra. The $\Gamma$-action on $\mc H$ extends continuously to $\mc S$, so the
crossed product algebra $\mc S (\mc R,q) \rtimes \Gamma$ is well-defined. 
By \cite[Lemma 2.20]{Opd-Sp} a finite dimensional $\mc H \rtimes \Gamma$-representation is 
tempered if and only if it extends continuously to an $\mc S \rtimes \Gamma$-representation. 

We define a *-operation and a trace on $\mc H (\mc R,q)$ by
\begin{align*}
& \big( \sum\nolimits_{w \in W^e} c_w N_w \big)^* = 
\sum\nolimits_{w \in W^e} \overline{c_w} N_{w^{-1}} , \\
& \tau \big( \sum\nolimits_{w \in W^e} c_w N_w \big) = c_e .
\end{align*}
Since $q(s_\alpha) > 0$, * preserves the relations \eqref{eq:multrules}
and defines an anti-involution of $\mc H (\mc R,q)$. The set $\{ N_w : w \in W^e\}$
is an orthonormal basis of $\mc H (\mc R,q)$ for the inner product
\[
\inp{h_1}{h_2} = \tau (h_1^* h_2) . 
\]
This gives $\mc H (\mc R,q)$ the structure of a Hilbert algebra.
The Hilbert space completion $L^2 (\mc R)$ of $\mc H (\mc R,q)$ is a module 
over $\mc H (\mc R,q)$, via left multiplication. Moreover every 
$h \in \mc H (\mc R,q)$ acts as a bounded linear operator \cite[Lemma 2.3]{Opd-Sp}. 
The reduced $C^*$-algebra of $\mc H (\mc R,q)$ \cite[\S 2.4]{Opd-Sp}, denoted
$C_r^* (\mc R,q)$, is defined as the closure of $\mc H (\mc R,q)$ in the algebra of 
bounded linear operators on $L^2 (\mc R)$. By \cite[Theorem 6.1]{Opd-Sp}
\[
\mc H (\mc R,q) \subset \mc S (\mc R,q) \subset C_r^* (\mc R,q) .
\]
As in \eqref{eq:1.22}, we can extend this to a $C^*$-algebra 
$C_r^* (\mc R,q) \rtimes \Gamma$, provided that $q$ is $\Gamma$-invariant.

Let us recall some background about $C_r^* (\mc R,q) \rtimes \Gamma$, mainly
from \cite{Opd-Sp,SolAHA}. It follows from \cite[Corollary 5.7]{DeOp1} that it is a
finite type I $C^*$-algebra and that $\Irr (C_r^* (\mc R,q))$ is precisely the tempered 
part of $\Irr (\mc H (\mc R,q))$. 
According to \cite[Theorem 4.23]{Opd-Sp} all irreducible $\mc S (\mc R,q) \rtimes 
\Gamma$-representations extend continuously to $C_r^* (\mc R,q) \rtimes \Gamma$. Hence
we can regard the representation theory of $C_r^* (\mc R,q) \rtimes \Gamma$ as the 
tempered unitary representation theory of $\mc H (\mc R,q) \rtimes \Gamma)$.

The structure of $C_r^* (\mc R,q) \rtimes \Gamma$ is described in terms of parabolically 
induced representations. As induction data we use triples $(Q,\delta,t)$ where 
$Q \subset \Delta$, $\delta \in \Irr_{L^2}(\mc H_Q)$ and $t \in T^Q$.
We regard two triples $(Q,\delta,t)$ and $(Q',\delta',t')$ as equivalent if
$Q = Q', t = t'$ and $\delta \cong \delta'$. Notice that $\mc H_Q$ comes from a semisimple 
root datum, so it can have discrete series representations. We inflate such a representation
to $\mc H^Q$ via the projection \eqref{eq:1.1}. To a triple $(Q,\delta,t)$ we associate the 
$\mc H \rtimes \Gamma$-representation
\begin{equation}\label{eq:1.24}
\pi^\Gamma (Q,\delta,t) = \ind_{\mc H^Q}^{\mc H \rtimes \Gamma} (\delta \circ \psi_t). 
\end{equation}
(When $\Gamma = 1$, we often suppress it from these and similar notations.)
For $t \in T_\un^Q = T^Q \cap T_\un$ these representations extend continuously
to the respective $C^*$-completions of the involved algebras. Let $\Xi_\un$ be
the set of triples $(Q,\delta,t)$ as above, such that moreover $t \in T_\un$.
Considering $Q$ and $\delta$ as discrete variables, we regard $\Xi_\un$ as a
disjoint union of finitely many compact real tori (of different dimensions).

Let $\mc V_\Xi^\Gamma$ be the vector bundle over $\Xi_\un$, whose fibre at
$\xi = (Q,\delta,t)$ is the vector space underlying $\pi^\Gamma (Q,\delta,t)$. That
vector space is independent of $t$, so the vector bundle is trivial. Let
$\End (\mc V_\Xi^\Gamma)$ be the algebra bundle with fibres $\End_\C \big( \pi^\Gamma 
(Q,\delta,t) \big)$. These data give rise to a canonical map
\begin{equation}\label{eq:1.2}
\begin{array}{ccc} 
\mc H (\mc R,q) \rtimes \Gamma & \to & \mc O \big(\Xi ; \End (\mc V_\Xi^\Gamma) \big) \\
 h & \mapsto & \big( \xi \mapsto \pi^\Gamma (\xi) (h) \big) 
\end{array}
\end{equation}
which we refer to as the Fourier transform. By \cite[Lemma 2.22]{Opd-Sp}
every discrete series representation is unitary, so $V_\delta$ carries an $\mc H_Q$-invariant
inner product and $\End_\C (V_\delta)$ has a natural *-operation. For any $t \in T^Q$ this
becomes an $\mc H^Q$-invariant nondegenerate pairing between $\delta \circ \phi_t$ and 
$\delta \circ \phi_{t |t|^{-2}}$. By \cite[Proposition 4.19]{Opd-Sp} this extends canonically
to an inner product on the vector space
\begin{equation}\label{eq:1.7}
\pi^\Gamma (Q,\delta,t) = \Gamma \ltimes \mc H (W,q) \otimes_{\mc H (W_Q,q)} V_\delta .
\end{equation}
That yields an anti-involution on $\End_\C (\pi^\Gamma (Q,\delta,t))$ and a nondegenerate 
$\mc H \rtimes \Gamma$-invariant pairing between $\pi^\Gamma (Q,\delta,t)$ and 
$\pi^\Gamma (Q,\delta,t \, |t|^{-2})$. 

The algebra $\mc O \big(\Xi ; \End (\mc V_\Xi^\Gamma) \big)$ is endowed with the anti-involution
\begin{equation}
(f^*) (Q,\delta,t) = f (Q,\delta, t \, |t|^{-2} )^* . 
\end{equation}
With respect to this anti-involution, \eqref{eq:1.2} is a *-homomorphism.

To administer the upcoming intertwining operators we use
a finite groupoid $\mc G$ which acts on $\End (\mc V_\Xi^\Gamma)$.
It is made from elements of $W\rtimes \Gamma$ and of $K_Q := T_Q \cap T^Q$.
More precisely, its base space is the power set of $\Delta$, and for
$Q,Q' \subseteq \Delta$ the collection of arrows from $Q$ to $Q'$ is
\begin{equation}\label{eq:GPQ}
\mc G_{QQ'} = \{ (g,u) : g \in \Gamma \ltimes W , u \in K_Q , g (Q) = Q' \} .
\end{equation}
Whenever it is defined, the multiplication in $\mc G$ is 
\[
(g',u') \cdot (g,u) = (g' g, g^{-1} (u') u) .
\]
In particular, writing $W\Gamma (Q,Q) = \{ w \in W\Gamma : w (Q) = Q \}$, 
we have the group
\begin{equation}
\mc G_{QQ} = W\Gamma (Q,Q) \rtimes K_Q . 
\end{equation}
Usually we will write elements of $\mc G$ simply as $gu$. 
There is an analogous groupoid $\mc G^Q$ for $\mc H^Q \rtimes \Gamma_Q$, which under
Conditions \ref{cond:Gamma} satisfies $\mc G^Q_{QQ} = \Gamma_Q \times K_Q$. 

For $\gamma \in \Gamma W$ with $\gamma (Q) = Q' \subset \Delta$ there are algebra isomorphisms
\begin{equation}\label{eq:psigamma}
\begin{array}{llcl}
\psi_\gamma : \mc H_Q \to \mc H_{Q'} , &
\theta_{x_Q} N_w & \mapsto & \theta_{\gamma (x_Q)} N_{\gamma w \gamma^{-1}} , \\
\psi_\gamma : \mc H^Q \to \mc H^{Q'} , &
\theta_x N_w & \mapsto & \theta_{\gamma x} N_{\gamma w \gamma^{-1}} .
\end{array}
\end{equation}
The groupoid $\mc G$ acts from the left on $\Xi_\un$ by
\begin{equation}\label{eq:2.3}
(g,u) \cdot (Q,\delta,t) := (g (Q),\delta \circ \psi_u^{-1} \circ \psi_g^{-1},g (ut)) ,
\end{equation}
the action being defined if and only if $g (Q) \subset \Delta$. 

Suppose that $g(Q) = Q' \subset \Delta$ and $\delta' \cong \delta \circ \psi_u^{-1} 
\circ \psi_g^{-1}$. By \cite[Theorem 4.33]{Opd-Sp} and \cite[Theorem 3.1.5]{SolAHA} 
there exists an intertwining operator
\begin{equation}\label{eq:2.9}
\pi^\Gamma (gu,Q,\delta,t) \in \Hom_{\mc H (\mc R,q) \rtimes \Gamma} 
\big( \pi^\Gamma (Q,\delta,t) , \pi^\Gamma (Q',\delta',g (ut)) \big) ,
\end{equation}
which depends algebraically on $t \in T^Q_\un$. This implies that, for all $\xi \in \Xi$ 
and $g \in \mc G$ such that $g \xi$ is defined, $\pi^\Gamma (\xi)$ and $\pi^\Gamma (g \xi)$ 
have the same irreducible constituents, counted with multiplicity \cite[Lemma 3.1.7]{SolAHA}.

The action of $\mc G$ on the continuous sections $C (\Xi_\un ;\End (\mc V_\Xi^\Gamma))$ 
is given by
\begin{equation}\label{eq:2.1}
(g \cdot f) (g \xi) = \pi^\Gamma (g,\xi) f (\xi) \pi^\Gamma (g,\xi )^{-1}
\qquad g \in \mc G_{QQ}, \xi = (Q,\delta,t).
\end{equation}
The next result is the Plancherel isomorphism for affine Hecke algebras, proven in
\cite[Theorem 5.3 and Corollary 5.7]{DeOp1} and \cite[Theorem 3.2.2]{SolAHA}.

\begin{thm}\label{thm:2.1}
The map Fourier transform \eqref{eq:1.2} induces *-homomorphisms
\[
\begin{array}{rlr}
\mc H (\mc R,q) \rtimes \Gamma & \to & \mc O \big( \Xi ; \End (\mc V_\Xi^\Gamma) \big)^{\mc G} ,\\
\mc S (\mc R,q) \rtimes \Gamma & \to & 
C^\infty \big( \Xi_\un ; \End (\mc V_\Xi^\Gamma) \big)^{\mc G} ,\\
C_r^* (\mc R,q) \rtimes \Gamma & \to & C \big( \Xi_\un ; \End (\mc V_\Xi^\Gamma) \big)^{\mc G}.
\end{array}
\]
The first is injective, the second is an isomorphism of Fr\'echet algebras and the
third is an isomorphism of $C^*$-algebras.
\end{thm}

For $q=1$ these simplify to the well-known isomorphisms
\begin{equation}\label{eq:2.2}
\begin{array}{rrrrr}
\mc H (\mc R,1) \rtimes \Gamma & = & \mc O (T) \rtimes W \Gamma & \to & 
\mc O \big( T ; \End_\C (\C [W \Gamma]) \big)^{W \Gamma} ,\\
\mc S (\mc R,1) \rtimes \Gamma & = & C^\infty (T_\un) \rtimes W \Gamma & \to & 
C^\infty \big( T_\un ; \End_\C (\C [W \Gamma]) \big)^{W \Gamma} ,\\
C^*_r (\mc R,1) \rtimes \Gamma & = & C (T_\un) \rtimes W \Gamma & \to &
C \big( T_\un ; \End_\C (\C [W \Gamma]) \big)^{W \Gamma} . 
\end{array}
\end{equation}
Unfortunately, the bookkeeping in Theorem \ref{thm:2.1} is not entirely suitable for our purposes,
because sometimes the parabolic subalgebras need to be extended by diagram automorphisms. In those
cases we should rather use induction data based on $\Irr_{L^2}(\mc H^Q \rtimes \Gamma_Q)$ 
than based on $\Irr_{L^2}(\mc H_Q)$ or $\Irr_{L^2}(\mc H^Q)$. 

We fix a system of subgroups $\Gamma_Q \subset \Gamma \; (Q \subset \Delta)$ satisfying 
Condition \ref{cond:Gamma}.
With Lemma \ref{lem:2.2} in mind we define new induction data. They are triples $(Q,\sigma,t)$
with $Q \subset \Delta$, $t \in T^Q$ and $\sigma \in \Irr_{L^2}(\mc H_Q \rtimes \Gamma_Q)$.
We regard another such triple $(Q',\sigma',t')$ as equivalent if and only if $Q' = Q,\; t' = t$
and $\sigma' \cong \sigma$.
We keep the same groupoid $\mc G$ as before, it also acts on the new triples via \eqref{eq:2.3}.
To such a triple we associate the representation
\begin{equation}\label{eq:1.36}
\pi (Q,\sigma,t) = \ind_{\mc H^Q \rtimes \Gamma_Q}^{\mc H \rtimes \Gamma}(\sigma \otimes t) .  
\end{equation}
The vector space underlying $\pi (Q,\sigma,t)$ does not depend on $t$, we denote it by 
$V_{Q,\sigma}$. There is a natural homomorphism
\begin{equation}\label{eq:1.3}
\begin{array}{ccc}
\mc H (\mc R,q) \rtimes \Gamma & \to & \mc O (T^Q) \otimes \End_\C (V_{Q,\sigma}) \\
h & \mapsto & \big( t \mapsto \pi (Q,\sigma,t)(h) \big) .
\end{array} 
\end{equation}
We refer to the system of these maps, for all $Q$ and $\sigma$, as the Fourier transform for 
$\mc H (\mc R,q) \rtimes \Gamma$.
The recipe for the intertwining operators from \cite[\S 4]{Opd-Sp} and 
\cite[Theorem 3.1.5]{SolAHA} remains valid, so we get
\begin{equation}\label{eq:2.10}
\pi (gu,Q,\sigma,t) \in \Hom_{\mc H \rtimes \Gamma} 
\big( \pi (Q,\sigma,t),\pi (g(Q),\sigma',g(ut)) \big) 
\end{equation}
with the same properties as in \eqref{eq:2.9}. In particular $\pi (Q,\sigma,t)$ and
$\pi (g(Q),\sigma',g(ut))$ have the same irreducible constituents, counted with multiplicity.
With these notions we can vary on the Plancherel isomorphism (Theorem \ref{thm:2.1}). 

To do so, we first consider essentially discrete series representations
of $\mc H^Q \rtimes \Gamma_Q$.
Pick $\delta_1 \in \Irr_{L^2}(\mc H_Q)$ and $t_1 \in T_\un^Q$. We note that 
$\ind_{\mc H^Q}^{\mc H^Q \rtimes \Gamma_Q} (\delta \otimes t_1)$ is unitary and essentially 
discrete series, because $\Gamma_Q$ stabilizes $Q$. 
Write $\mc G^Q_{Q Q} \delta_1 = \{ \delta_i \}_i$. The summand of $C^\infty \big( \Xi_{Q,\un} ; 
\End (\mc V_{\Xi_Q}^{\Gamma_Q}) \big)^{\mc G^Q}$ associated to $(Q,\delta_1)$ is
\begin{equation}\label{eq:1.4}
\Big( \bigoplus\nolimits_i C^\infty \big( T^Q_\un ; 
\End_\C (\pi^{\Gamma_Q} (Q, \delta_i, t_i)) \big) \Big)^{\mc G^Q_{Q Q}} .
\end{equation}
Let $\{ \sigma_j \}_j$ be the members of $\Irr_{L^2}(\mc H_Q \rtimes \Gamma_Q)$ contained
in $\ind_{\mc H_Q}^{\mc H_Q \rtimes \Gamma_Q}(\delta_1 \circ \psi_u)$ for some
$u \in K_Q$. This set is stable under $\mc G^Q_{Q Q} = \Gamma_Q \times K_Q$.
The summand of $\Big( \bigoplus_{\sigma} C^\infty \big( T^Q_\un ; \End_\C (V_{\sigma}) \big) 
\Big)^{\mc G^Q}$ corresponding to the $\sigma_j$ is
\begin{equation}\label{eq:1.5}
\Big( \bigoplus\nolimits_j C^\infty \big( T^Q_\un ; 
\End_\C (V_{\sigma_j}) \big) \Big)^{\mc G^Q_{Q Q}} . 
\end{equation}

\begin{lem}\label{lem:2.7}
The algebras \eqref{eq:1.4} and \eqref{eq:1.5} are naturally isomorphic. 
\end{lem}
\begin{proof}
For $\sigma \in \Irr_{L^2}(\mc H^Q \rtimes \Gamma_Q)$ we write $(\sigma,t) > (\delta_1,t_1)$ if
\[
\Hom_{\mc H^Q \rtimes \Gamma_Q} \big( \sigma \otimes t , \ind_{\mc H^Q}^{\mc H^Q \rtimes \Gamma_Q} 
(\delta \otimes t_1 ) \big) \cong \Hom_{\mc H^Q}(\sigma \otimes t ,\delta \otimes t_1 )
\]
is nonzero. Since $\Gamma_Q$ is finite, the set of such $(\sigma,t)$ is finite. Hence the map
\[
\bigoplus_{(\sigma,t) > (\delta_1 ,t_1)} \sigma \circ \psi_t \;:\; \mc H^Q \rtimes \Gamma_Q \to
\bigoplus_{(\sigma,t) > (\delta_1 ,t_1)} \End_\C (V_\sigma)
\]
is surjective. The specialization of \eqref{eq:1.4} at $\mc G^Q_{Q Q} (Q,\delta_1 ,t_1)$ is also\\ 
$\bigoplus_{(\sigma,t) > (\delta_1 ,t_1)} \End_\C (V_\sigma)$, for that specialization is just 
\[
\ind_{\mc H^Q}^{\mc H^Q \rtimes \Gamma_Q} (\delta \circ \psi_{t_1}) 
(\mc S (\mc R^Q,q^Q) \rtimes \Gamma_Q) .
\]
Similarly, specializing the algebra \eqref{eq:1.5} at all $(\sigma,t) > (\delta_1 ,t_1)$ 
gives a surjection from \eqref{eq:1.5} to $\bigoplus_{(\sigma,t) > (\delta_1 ,t_1)} \End_\C (V_\sigma)$. 

Now we can explicitly compare \eqref{eq:1.5} with \eqref{eq:1.4}. Both are algebras of smooth 
sections of (trivial) algebra bundles, and specialization at the points associated to 
$(\delta_1 ,t_1)$ yields the same algebra in both cases. This holds for any $t_1 \in T^Q_\un$ 
and that accounts for all base points of these algebra bundles, so \eqref{eq:1.4} and \eqref{eq:1.5} 
are isomorphic. Moreover the isomorphism is canonical: it is the composition of the inverse of 
the map in Theorem \ref{thm:2.1} and the map induced by \eqref{eq:1.3} 
(both for $\mc H^Q \rtimes \Gamma_Q$).
\end{proof}

Next we formulate our variation on the Plancherel isomorphism.

\begin{prop}\label{prop:2.3}
The Fourier transform from \eqref{eq:1.3} induces isomorphisms of Fr\'echet *-algebras 
\[
\begin{array}{ccc}
\mc S (\mc R,q) \rtimes \Gamma & \to & \Big( \bigoplus_{Q,\sigma} 
C^\infty \big( T^Q_\un ; \End_\C (V_{Q,\sigma}) \big) \Big)^{\mc G} ,\\
C_r^* (\mc R,q) \rtimes \Gamma & \to & 
\Big( \bigoplus_{Q,\sigma} C \big( T^Q_\un ; \End_\C (V_{Q,\sigma}) \big) \Big)^{\mc G}.
\end{array}
\]
\end{prop}
\begin{proof}
We will analyse the right hand side of Theorem \ref{thm:2.1} (for the Schwartz algebras) and 
compare it with the current setting. 

For every $Q \subset \Delta$, Lemma \ref{lem:2.7} yields a canonical isomorphism
\begin{multline}\label{eq:1.6}
\Big( \bigoplus\nolimits_{\sigma \in \Irr_{L^2}(\mc H_Q \rtimes \Gamma_Q)} C^\infty \big( T^Q_\un ; 
\End_\C (V_\sigma) \big) \Big)^{\mc G^Q_{Q Q}} \cong \\
\Big( \bigoplus\nolimits_{\delta \in \Irr_{L^2}(\mc H_Q)} C^\infty \big( T^Q_\un ; 
\End_\C (\C[\Gamma_Q] \otimes V_\delta) \big) \Big)^{\mc G^Q_{Q Q}} .
\end{multline}
To obtain the right hand side of Theorem \ref{thm:2.1} from \eqref{eq:1.6}, we must apply
$\ind_{\mc H^Q \rtimes \Gamma_Q}^{\mc H \rtimes \Gamma}$ to $\C[\Gamma_Q] \otimes V_\delta \cong
\ind_{\mc H^Q}^{\mc H^Q \rtimes \Gamma_Q} (\delta \otimes t)$ and then take invariants with
respect to the larger groupoid $\mc G \supset \mc G^Q$. The formula \cite[(3.12)]{SolAHA} is the
same for $\mc G^Q$ and for $\mc G$, so the intertwiners associated to elements of $\mc G^Q$ need
not be adjusted in this process. 

With exactly the same procedure we can turn \eqref{eq:1.6} into the right hand side of the
current proposition. The intertwining operators associated to elements of $\mc G$ agree under
the isomorphisms obtained from \eqref{eq:1.6} by applying $\ind_{\mc H^Q \rtimes \Gamma_Q}^{\mc H 
\rtimes \Gamma}$, because in both settings they were constructed with \cite[(3.12) and 
Theorem 3.1.5]{SolAHA}. Consequently 
\[
\Big( \bigoplus\nolimits_{Q,\sigma} C^\infty \big( T^Q_\un ; \End_\C (V_{Q,\sigma}) \big) 
\Big)^{\mc G} \cong C^\infty \big( \Xi_\un ; \End (\mc V_\Xi^\Gamma) \big)^{\mc G} ,
\]
proving the proposition for the Schwartz algebras. For $C_r^* (\mc R,q) \rtimes \Gamma$ one can 
use the same argument, with everywhere $C^\infty$ replaced by continuous functions.
\end{proof}

Choose representatives $Q$ for $\mc P (\Delta)$ modulo $W \Gamma$-association. For every
such $Q$ we choose representatives $\sigma$ for the action of $\mc G_{Q Q} = W \Gamma (Q,Q)
\times K_Q$ on $\Irr_{L^2}(\mc H_Q \rtimes \Gamma_Q)$. By Lemma \ref{lem:2.2} these $\sigma$ 
also form representatives for the action of $\mc G_{Q Q} \ltimes T^Q_\un$ on 
$\Irr_{L^2}(\mc H^Q \rtimes \Gamma_Q)$. We denote the resulting set of representatives of pairs 
by $(Q,\sigma) / \mc G$. Let $\mc G_{Q,\sigma}$ be the setwise stabilizer of $(Q,\sigma,T^Q_\un)$
in the group $\mc G_{QQ}$. Proposition \ref{prop:2.3} can be rephrased as isomorphisms
\begin{equation}\label{eq:2.4}
\begin{array}{ccc}
\mc S (\mc R,q) \rtimes \Gamma & \to & \bigoplus_{(Q,\sigma) / \mc G} 
C^\infty \big( T^Q_\un ; \End_\C (V_{Q,\sigma}) \big)^{\mc G_{Q,\sigma}} ,\\
C_r^* (\mc R,q) \rtimes \Gamma & \to & \bigoplus_{(Q,\sigma) / \mc G} 
C \big( T^Q_\un ; \End_\C (V_{Q,\sigma}) \big)^{\mc G_{Q,\sigma}}.
\end{array}
\end{equation}

Sometimes we have to consider the opposite algebra $(\mc H (\mc R,q) \rtimes \Gamma )^\op$
and its completions. It is, morally, clear that all the previous results can also developed 
for right $\mc H \rtimes \Gamma$-modules, that is, for $(\mc H \rtimes \Gamma )^\op$-modules.
However, none of that has been written down, so we prefer more steady ground.

For every $\mc H \rtimes \Gamma$-representation $(\pi,V_\pi)$, the full linear dual $V_\pi^*$ 
becomes a $(\mc H \rtimes \Gamma )^\op$-representation $\pi^*$ by
\[
\pi^* (h^\op) \lambda = \lambda \circ \pi (h) . 
\]
This sets up a bijection between finite dimensional left and right modules of 
$\mc H \rtimes \Gamma$. In view of the canonical inner products from on the spaces 
\eqref{eq:1.7}, this bijection commutes with induction from parabolic subalgebras.

For infinite dimensional representations there is often some choice for which dual space of 
$V_\pi$ we use here. In particular, when $V_\pi$ is a Hilbert space we can use $V_\pi$ also 
as dual space. With this convention one checks easily that $\pi$ is unitary if and only 
if $\pi^*$ is unitary.

The $\mc O (T)$-weights of $\pi^*$ are the same as for $\pi$, so $\pi^*$ is tempered or 
(essentially) discrete series if and only if $\pi$ is so. Thus the pairs $(Q,\sigma)$ with 
$\sigma \in \Irr_{L^2} (\mc H_Q \rtimes \Gamma_Q)$ are in natural bijection with the pairs 
$(Q,\sigma^*)$ in 
\begin{equation}\label{eq:1.8}
\bigcup\nolimits_{Q \subset \Delta} 
\Irr_{L^2}\big( (\mc H (\mc R_Q,q_Q) \rtimes \Gamma_Q )^\op \big) .
\end{equation}
The bijection is $\mc G$-equivariant for the $\mc G$-action on \eqref{eq:1.8} as in 
\eqref{eq:2.3}. Hence $\mc G_{Q,\sigma} = \mc G_{Q,\sigma^*}$ and we can take 
$(Q,\sigma^*) / \mc G$ to be the image of $(Q,\sigma) / \mc G$.

\begin{lem}\label{lem:2.4}
The Fourier transform for right $\mc H (\mc R,q) \rtimes \Gamma$-modules induces 
isomorphisms of Fr\'echet *-algebras 
\[
\begin{array}{ccc}
( \mc S (\mc R,q) \rtimes \Gamma )^\op & \to & \bigoplus_{(Q,\sigma^*) / \mc G} 
C^\infty \big( T^Q_\un ; \End_\C (V_{Q,\sigma^*}) \big)^{\mc G_{Q,\sigma^*}} ,\\
( C_r^* (\mc R,q) \rtimes \Gamma )^\op & \to & \bigoplus_{(Q,\sigma^*) / \mc G} 
C \big( T^Q_\un ; \End_\C (V_{Q,\sigma^*}) \big)^{\mc G_{Q,\sigma^*}}.
\end{array}
\]
\end{lem}
\begin{proof}
The opposite algebra of 
\[
\End_\C (V_{Q,\sigma}) = 
\End_\C \big( \ind_{\mc H^Q \rtimes \Gamma_Q}^{\mc H \rtimes \Gamma} V_\delta \big)
\]
is naturally isomorphic to $\End_\C \big( V_{Q,\sigma}^* \big)$, which by 
\cite[Proposition 4.19]{Opd-Sp} is canonically isomorphic with 
\[
\End_\C \big( \ind_{\mc H^Q \rtimes \Gamma_Q}^{\mc H \rtimes \Gamma} (V_\sigma^*) \big) =
\End_\C (V_{Q,\sigma^*}) .
\]
For $g \in \mc G_{Q,\sigma}$ we take $\pi (g,Q,\sigma^*,t |t|^{-2})$ to be the transpose inverse 
of $\pi (g,Q,\sigma,t)$. Thus an element of $C \big( T^Q_\un ; \End_\C (V_{Q,\sigma}) \big)$
is $\mc G_{Q,\sigma}$-invariant if and only if its transpose in $C \big( T^Q_\un ; 
\End_\C (V_{Q,\sigma^*}) \big)$ is $\mc G_{Q,\sigma^*}$-invariant for the action
\[
(g \cdot f) (g(Q,\sigma^*,t)) = \pi (g,Q,\sigma^*,t) f (Q,\sigma^*,t) \pi (g,Q,\sigma^*,t)^{-1}. 
\]
Now we take the opposite algebras in Theorem \ref{thm:2.1} and we find the desired isomorphisms.

The implementing algebra homomorphisms are given by transpose, the Fourier transform from 
\eqref{eq:1.3} and again transpose, which works out to the Fourier transform for 
$(\mc H (\mc R,q) \rtimes \Gamma )^\op$-modules. Since the correspondence between left and right
$\mc H (\mc R,q) \rtimes \Gamma$-modules preserves unitarity, the latter Fourier transform is
still a *-homomorphism.
\end{proof}

\subsection{The space of irreducible representations} \label{par:2.2} \

We compare the irreducible representations of $\mc H \rtimes \Gamma$ to its representations 
induced from proper parabolic subalgebras (i.e. the algebras $\mc H^Q \rtimes \Gamma_Q$ with
$Q \subsetneq \Delta$). Let $\mr{Gr}(\mc H \rtimes \Gamma)$ be the 
Grothendieck group of the category of finite length $\mc H \rtimes \Gamma$-representations
and write $\mr{Gr}_\Q (\mc H \rtimes \Gamma) = \Q \otimes_\Z \mr{Gr}(\mc H \rtimes \Gamma)$.
Then (parabolic) induction induces a $\Q$-linear map $\mr{Gr}_\Q (\mc H_Q \rtimes \Gamma_Q) \to
\mr{Gr}_\Q (\mc H \rtimes \Gamma)$.

\begin{thm}\label{thm:1.8}
\enuma{
\item The collection of irreducible $\mc H \rtimes \Gamma$-representations whose image in
$\mr{Gr}_\Q (\mc H \rtimes \Gamma)$ is not a $\Q$-linear combination of representations
induced from proper parabolic subalgebras is a nonempty union of $T^{W \Gamma}$-orbits in
$\Irr (\mc H \rtimes \Gamma)$.
\item Suppose that, for every $w \in W \Gamma \setminus 
\bigcup_{Q \subsetneq \Delta} W(R_Q) \Gamma_Q$, the set $T_\Delta^w$ is finite. 
Then the set in part (a) is a finite union of $T^{W \Gamma}$-orbits.
}
\end{thm}
\begin{proof}
(a) Recall from \eqref{eq:1.11} that $\mc O (T)^{W (R_Q) \Gamma_Q} \subset 
Z (\mc H^Q \rtimes \Gamma_Q)$. Hence all irreducible representations of $\mc H^Q \rtimes \Gamma_Q$ 
come in families parametrized by $T^{W(R_Q) \Gamma_Q}$. Since $T^{W \Gamma} \subset 
T^{W(R_Q) \Gamma_Q}$ for all $Q \subset \Delta$, the set of irreducible representations under
consideration is a union of $T^{W \Gamma}$-orbits. 

By \cite[Lemma 2.3]{SolHomAHA}, $\Irr (\mc H \rtimes \Gamma)$ can be parametrized by 
the extended quotient 
\[
T /\!/ W\Gamma = \Big( \bigcup\nolimits_{w \in W \Gamma} \{w\} \times T^w \Big) \big/ W \Gamma ,
\]
where $W \Gamma$ acts on the union by $w' \cdot (w,t) = (w' w w'^{-1}, w'(t))$. This 
parametrization respects central characters, up to a twists which are constant on
connected components of $T /\!/ W\Gamma$ \cite[Theorem 2.6]{SolHomAHA}. In this parametrization
of $\Irr (\mc H \rtimes \Gamma)$ almost all elements of a piece $\{w\} \times T^w$ with 
$w \in W(R_Q) \Gamma_Q$ come from representations induced from $\mc H^Q \rtimes 
\Gamma_Q$, and such $w$ account for all representations induced from proper parabolic subalgebras.
Hence the set considered in the statement can be parametrized by
\begin{equation}\label{eq:1.12}
\Big( \bigcup\nolimits_{w \in W \Gamma : w \notin W(R_Q)\Gamma_Q \forall Q \subsetneq \Delta} 
\{w\} \times T^w \Big) \big/ W \Gamma
\end{equation}
This set is nonempty because every Coxeter element of $W = W(R)$ contributes at least $(w,1)$
to it.\\
(b) This is obvious from \eqref{eq:1.12}.
\end{proof}

The induction data from $\Xi$ give rise to a partition Irr$(\mc H \rtimes \Gamma)$ into finite packets.

\begin{thm}\label{thm:2.5} \cite[Theorem 3.3.2.b]{SolAHA} \\
For every $\pi \in \Irr (\mc H \rtimes \Gamma)$ there exists a unique $\mc G$-association
class $\mc G (Q,\delta,t) \in \Xi / \mc G$ such that $\pi$ is a constituent of 
$\pi^\Gamma (Q,\delta,t)$ and the invariant $\norm{cc (\delta)}$ from \eqref{eq:1.cc} 
is maximal for this property.
\end{thm}

With the new induction data $(Q,\sigma,t)$ from \eqref{eq:1.36} we can vary on 
Theorem \ref{thm:2.5}. (Now $\sigma \in \Irr_{L^2}(\mc H_Q \rtimes \Gamma_Q)$,
whereas the above $\delta$ was a representation of $\mc H_Q$.)

\begin{thm}\label{thm:2.6}
\enuma{
\item For every $\pi \in \Irr (\mc H \rtimes \Gamma)$ there exists a triple $(Q,\sigma,t)$ as above,
such that $\pi$ is a constituent of $\pi (Q,\sigma,t)$ and $\norm{cc(\sigma)}$ is maximal for this 
property. In this situation we say that $\pi$ is a \emph{Langlands constituent} of $\pi (Q,\sigma,t)$.
\item In the setting of part (a), the restriction of $\sigma \otimes t$ to $\mc H^Q$ is a direct 
sum of irreducible representations in one $\Gamma_Q$-orbit, say 
$\Gamma_Q (\delta \otimes t) \subset \Irr_{L^2}(\mc H^Q)$. 
Then $(Q,\delta,t)$ is uniquely determined by $\pi$, up to the action of $\mc G$.
\item Let $(Q,\sigma,t)$ be any induction datum as in \eqref{eq:1.36}. Every constituent of 
$\pi (Q,\sigma,t)$ is either a Langlands constituent or a constituent of $\pi (Q',\sigma',t')$ for
some induction datum with $\norm{cc (\sigma')} > \norm{cc (\sigma)}$. 
\item $\pi$ is tempered if and only if $t \in T_\un$, where $t \in T^Q$ comes from part (a).
}
\end{thm}
\begin{proof}
(a) Let $(Q,\delta,t)$ be as in Theorem \ref{thm:2.5}. Thus $\pi$ is a constituent of
\[
\pi^\Gamma (Q,\delta,t) = \ind_{\mc H^Q \rtimes \Gamma_Q}^{\mc H \rtimes \Gamma} \big(
\ind_{\mc H^Q}^{\mc H^Q \rtimes \Gamma_Q} (\delta \otimes t) \big) ,
\]
and the norm of the central character of $\delta$ is maximal for this property. 
Let $T^{--Q}$ be the subset $T^{--}$ of $T$, but computed with respect to $Q$.
Every $\mc A$-weight of $\ind_{\mc H^Q}^{\mc H^Q \rtimes \Gamma_Q} (\delta \otimes t)$ 
lies in one of the $\Gamma_Q$-orbits of weights of $\delta \otimes t$, which are entirely
contained in $T^{--Q} T_\un$ because $\Gamma_Q$ stabilizes $Q$. In other words, 
$\ind_{\mc H^Q}^{\mc H^Q \rtimes \Gamma_Q} (\delta \otimes t)$ is a direct sum of finitely
many irreducible essentially discrete series representations of $\mc H^Q \rtimes \Gamma$,
all with central characters in the same $W_Q \rtimes \Gamma_Q$-orbit. By Lemma \ref{lem:2.2}
all these summands are of the form $\sigma_i \otimes t$ with $\sigma_i \in \Irr_{L^2}
(\mc H_Q \rtimes \Gamma_Q)$. Hence 
\begin{equation}\label{eq:1.10}
\pi^\Gamma (Q,\delta,t) = \bigoplus\nolimits_i \pi (Q,\sigma_i,t)
\end{equation}
and $\pi$ is a constituent of (at least) one $\pi (Q,\sigma_i,t)$.

The central characters of the $\sigma_i$ and of $\delta$ have the same norm, and that of
$\delta$ was maximal given $\pi$. Hence the norm of the central character of $\sigma_i$
is also maximal, given $\pi$.\\
(b) Suppose that $(Q,\sigma,t)$ satisfies the requirements of part (a), that is, $\pi$ is
a Langlands constituent of $\pi (Q,\sigma,t)$.
In the proof of Lemma \ref{lem:2.2} we observed that the restriction of $\sigma \otimes t$ to 
$\mc H^Q$ is a direct sum of representations of the form $\gamma ((\delta \otimes k^{-1}) 
\otimes kt)$ with $\gamma \in \Gamma_Q$ and $k \in K_Q$. For all $\gamma \in \Gamma_Q$, $\pi$ 
is a constituent of 
\[
\pi^\Gamma (Q,\delta,t) = \pi^\Gamma (Q,\gamma (\delta \otimes k^{-1}),\gamma(kt)) .
\]
The norm of the central character of $\gamma (\delta \otimes k^{-1})$ 
is maximal for this property, by the assumption on $\sigma$. By Theorem \ref{thm:2.5} all 
$(Q,\gamma (\delta \otimes k^{-1}),\gamma(kt))$ lie in a unique $\mc G$-association class in 
$\Xi$ determined by $\pi$.\\
(c) In view of \eqref{eq:1.10}, it suffices to consider the constituents of $\pi^\Gamma 
(Q,\delta,t)$. Then the statement is implicit in \cite{SolAHA}, we make it more explicit here.
By \cite[Lemma 3.1.7 and Theorem 3.3.2]{SolAHA}, we may furthermore assume that $\xi = 
(Q,\delta,t)$ is in positive position, that is, $|t| \in T^{Q+} = \exp (\mf a^{Q} \cap \mf a^+)$. 
Write $P(\xi) = \{ \alpha \in \Delta : |\alpha (t)| = 1\}$ and consider the representation 
$\mr{ind}_{\mc H^Q}^{\mc H^{P(\xi)} \rtimes \Gamma (P(\xi),P(\xi))} (\delta \otimes t)$. 
By \cite[Proposition 3.1.4]{SolAHA} that representation is completely reducible, and all
its irreducible summands are of the form $\delta' \otimes t'$ where $\delta' \in 
\Irr \big( \mc H_{P(\xi)} \rtimes \Gamma (P(\xi),P(\xi)) \big)$ is tempered and 
$t' \in T^{P(\xi)} \cap t T_{P(\xi)}$ with 
\[
|t'| \in T^{P(\xi)++} = \exp \big( \{ \mu \in \mf a^{P(\xi)} :
\inp{\alpha}{\mu} > 0 \; \forall \alpha \in \Delta \setminus P(\xi) \} \big) .
\]
Thus every constituent $\pi (Q,\sigma,t)$ is a constituent of 
\begin{equation} \label{eq:1.50}
\ind_{\mc H^{P(\xi)} \rtimes \Gamma (P(\xi),P(\xi))}^{\mc H \rtimes \Gamma} (\delta' \otimes t')
\end{equation} 
for such a $\delta' \otimes t'$.
By \cite[Theorem 3.3.2]{SolAHA}, the Langlands constituents of $\pi (Q,\sigma,t)$ (or 
$\pi^\Gamma (Q,\delta,t)$) are precisely those constitutents which are quotients of a 
representation \eqref{eq:1.50}.
The Langlands classification for extended affine Hecke algebras \cite[Corollary 2.2.5]{SolAHA}, 
says that the latter representation has a unique irreducible quotient (called the Langlands
quotient, hence our terminology Langlands constituents). Moreover by 
\cite[Lemma 2.2.6.b]{SolAHA} all other constituents of \eqref{eq:1.50} are Langlands quotients 
for data where the norm of the central character is bigger than $\norm{cc(\sigma)}$.
Then \cite[Proposition 3.1.4 and Theorem 3.3.2]{SolAHA} entail that those other irreducible
representations occur as Langlands constituents of a $\pi^\Gamma (Q',\delta',t')$ with
$\norm{cc(\delta')} > \norm{cc(\sigma)} = \norm{cc(\delta)}$.\\
(d) As observed in the proof of part (c), in the construction for part (a) and for 
\cite[Theorem 3.3.2]{SolAHA} we may assume that $(Q,\sigma,t)$ is positive and that $\pi$ is 
the unique Langlands quotient of \eqref{eq:1.50}, for one of the above $\delta' \otimes t'$.
By the uniqueness in the Langlands classification for extended affine Hecke algebras 
\cite[Corollary 2.2.5]{SolAHA},  $\pi$ is tempered if and only if $P(\xi) = \Delta$ and 
$t' \in T^\Delta_\un$. 

Here $P(\xi) = \Delta$ implies $|t| = |t'| \in T^\Delta$, and then $t' \in T^\Delta_\un$
says that $|t| = 1$. Conversely, $t \in T_\un$ implies $P(\xi) = \Delta$ and
$t = t' \in T^\Delta$.
\end{proof}

With Theorem \ref{thm:2.6} one can express the structure of Irr$(\mc H \rtimes \Gamma)$ in 
terms of its subset of irreducible tempered representations. In essence, the former is the 
complexification of the latter (which is something like a real algebraic variety with 
multiplicities). We do not need the full strength of this. For our purposes it suffices to 
consider half-lines in the parameter space, such that $Q,\sigma$ and the unitary part of $t$ 
are fixed and the absolute value of $t$ can be scaled by a positive factor.

\begin{prop}\label{prop:2.7}
Let $(Q,\sigma,t)$ be an induction datum with $\sigma \in \Irr_{L^2}(\mc H^Q \rtimes \Gamma_Q)$.
\enuma{
\item For $r \in \R_{>-1}$, the number of inequivalent Langlands constituents of
$\pi (Q,\sigma,t \, |t|^{r})$ does not depend on $r$.
\item For all but finitely many $r \in \R_{>-1}$, $\pi (Q,\sigma,t \, |t|^r)$ is completely 
reducible. Then all its irreducible subquotients are Langlands constituents.
}
\end{prop}
\begin{proof}
(a) Let $\xi = (Q,\delta, t) \in \Xi$ be as in Theorem \ref{thm:2.6}.

Notice that the intertwining operators $\pi (gu,Q,\sigma,t)$ depend algebraically 
on $t \in T^Q$. This implies that, for every $gu \in \mc G$, $\pi (Q,\sigma,t)$ and 
$\pi (gu (Q,\sigma,t))$ have the same irreducible subquotients (counted with multiplicity),
see \cite[Lemma 3.1.7]{SolAHA} or \cite[Lemma 3.4]{SolGHA}. Since every $\mc G$-orbit in $\Xi$ 
contains an element in positive position, we may assume that $(Q,\delta,t)$ is positive.
Then $(Q,\delta,t \, |t|^r)$ is positive for $r \geq -1$ and for $r > -1$ its stabilizer in
$\mc G$ does not depend on $r$. 

Now the statement for $\pi^\Gamma (Q,\delta,t)$ is an instance of 
\cite[Proposition 3.4.1]{SolAHA}. Together with \eqref{eq:1.10} and Theorem \ref{thm:2.6}.c 
this implies the statement for $\pi (Q,\sigma,t)$.\\
(b) By \cite[Proposition 3.1.4.a]{SolAHA} the representation
\begin{equation}\label{eq:1.51}
\mr{ind}_{\mc H^Q}^{\mc H^{P(\xi)} \rtimes \Gamma (P(\xi),P(\xi))} (\delta \otimes t \, |t|^r) =
\bigoplus\nolimits_i \mr{ind}_{\mc H^Q \rtimes \Gamma_Q}^{\mc H^{P(\xi)} \rtimes \Gamma 
(P(\xi),P(\xi))} (\sigma_i \otimes t \, |t|^r ) 
\end{equation}
is completely reducible. Therefore it suffices to consider all irreducible direct summands
of \eqref{eq:1.51} separately. This brings us to representations of the form 
\begin{equation}\label{eq:1.54}
\ind_{\mc H^{P(\xi)} \rtimes \Gamma (P(\xi),P(\xi))}^{\mc H \rtimes \Gamma}
(\delta' \otimes t' |t|^r) .
\end{equation}
Then $(P(\xi),\delta', t' |t|^r )$ is a datum for the 
Langlands classification for extended affine Hecke algebras \cite[Corollary 2.2.5]{SolAHA}.
This result says that such a representation has a unique irreducible quotient, so \eqref{eq:1.54} 
is irreducible as soon as it is completely reducible. In that case its Langlands constituent 
obviously is the whole of \eqref{eq:1.54}. That implies the claim about the constituents of
$\pi (Q,\sigma,t \, |t|^ r)$ when that representation is completely reducible.

Next we show that \eqref{eq:1.54} is irreducible for almost all $r \in \R_{> -1}$. From 
\cite[Lemma 2.2.6.a]{SolAHA} we know that the space of $\mc H \rtimes \Gamma$-endomorphisms of
\eqref{eq:1.54} is just $\C $Id.

For all $z \in \C$ the isotropy group $\mc G_{(P(\xi),\delta',t')}$ also fixes
$(P(\xi),\delta',t' |t|^z )$. Since $\mc G_{P(\xi),\delta'}$ acts on $T^Q$ by group 
automorphisms and translations, for almost all $z \in \C$ $\mc G_{(P(\xi),\delta',t' |t|^z)}$ 
equals $\mc G_{(P(\xi),\delta',t')}$. In particular this happens for some 
$z \in -1 + i \R$. Then $t' |t|^z \in T^{P(\xi)}_\un$ and by \cite[Corollary 3.1.3]{SolAHA} 
the representation $\pi^\Gamma (P(\xi),\delta',t' |t|^z)$ is unitary 
(and in particular completely reducible). By \cite[Theorem 3.3.1.b]{SolAHA}
\begin{equation}\label{eq:1.53}
\End_{\mc H \rtimes \Gamma} \big( \ind_{\mc H^{P(\xi)} \rtimes \Gamma (P(\xi),P(\xi))}^{\mc H 
\rtimes \Gamma} (\delta' \otimes t' |t|^z ) \big) 
\end{equation}
is spanned by intertwining operators coming from $\mc G_{P(\xi),\delta',t'}$, just like for
\eqref{eq:1.54}. Therefore \eqref{eq:1.53} consists only of $\C $Id, which implies that
the representation is irreducible for that $z \in -1 + i \R$. In the above algebraic family of
finite dimensional representations parametrized by $z \in \C$, irreducibility is an open 
condition: slightly varying $z$ cannot destroy irreducibility. Hence the locus of $z$'s where 
the representation is reducibile is a Zariski-closed subset of $\C$, that is, it is finite. 
In particular \eqref{eq:1.54} is irreducible for all but finitely many $r \in \R_{>-1}$.
\end{proof}
\vspace{3mm}

\section{Reductive $p$-adic groups}
\label{sec:group}

Let $F$ be a non-archimedean local field and let $G = \mc G (F)$ be a connected reductive
algebraic group over $F$. We endow $G$ with the topology coming from the metric on $F$ and
we fix a Haar measure on $G$. Let $\mc H (G)$ be the Hecke algebra of $G$, the 
convolution algebra of locally constant compactly supported functions $G \to \C$. The 
product on $\mc H (G)$ is convolution (with respect to the Haar measure). 
Let $\mc S (G)$ be the Harish-Chandra--Schwartz algebra of $G$, as defined in \cite{HC} 
and \cite[\S III.6]{Wal}. By definition, a smooth $G$-representation (by default on a complex 
vector space) is tempered if and only if it extends continuously to a module for $\mc S (G)$.
Let $C_r^* (G)$ be the reduced $C^*$-algebra of $G$, the completion of $\mc H (G)$ in
the algebra of bounded linear operators on the Hilbert space $L^2 (G)$.  
By \cite[Theorem 29]{Vig} there are dense inclusions
\[
\mc H (G) \subset \mc S (G) \subset C_r^* (G) .
\]
The *-operation and the trace on these algebras are 
\[
f^* (g) = \overline{f (g^{-1})} \quad \text{and} \quad \tau (f) = f( 1_G ) .
\]
Fix a minimal parabolic $F$-subgroup 
$P_0 = M_0 U_0$ and let $W (G,M_0)$ be the Weyl group of $G$ with respect to the maximal 
$F$-split torus $A_0$ in the centre of $M_0$. Write $\mf a_0 = X^* (A_{M_0}) \otimes_\Z \R$ 
and endow this vector space with a $W(G,M_0)$-invariant inner product.

Every Levi $F$-subgroup of $G$ is conjugate to a standard Levi subgroup, that is, one that
contains $M_0$. Let $M$ be such a standard Levi subgroup of $G$, and let $A_M$ be the
maximal $F$-split torus in $Z(M)$. There is a canonical decomposition
\begin{equation}\label{eq:3.25}
\mf a_0 = \mf a_M \oplus \mf a^M ,
\end{equation}
where $\mf a_M = X^* (A_M) \otimes_\Z \R$ and $\mf a^M = \{ \chi \in X^* (A_0) : 
\chi|_{A_M} = 1 \} \otimes_\Z \R$.  Let $R(G,M)$ be the set of roots of $G$ with respect to 
$A_M$. For a parabolic subgroup $P = MU$ of $G$ we let $R(P,M)$ be the set of roots of
$(G,M)$ that appear in (the Lie algebra of) $P$. 

When $M_1 \subset M$ is another standard Levi subgroup of $G$, we write
$\mf a^{M}_{M_1} = \mf a_{M_1} \cap \mf a^M$. Every parabolic subgroup $P_1 = M_1 U_1$ of $M$ 
determines an obtuse cone in $\mf a_{M_1}^M$:
\[
{}^+ \mf a^M_{P_1} = \big\{ \sum\nolimits_{\alpha \in R (P_1,M_1)} 
c_\alpha \alpha \big|_{A_{M_1}} : c_\alpha > 0 \; \forall \alpha \big\} .
\]
Here we would obtain the same cone if we used only the simple roots for $(P_1,M_1)$. 
The closure of ${}^+ \mf a^M_{P_1}$ in $\mf a_{M_1}^M$ is
\[
\overline{{}^+ \mf a^M_{P_1}} = \big\{ \sum\nolimits_{\alpha \in R (P_1,M_1)} 
c_\alpha \alpha \big|_{A_{M_1}} : c_\alpha \geq 0 \; \forall \alpha \big\} .
\]
It is easy to see that the normalized Jacquet restriction functor $J^G_P : \Rep (G) \to \Rep (M)$
does not preserve temperedness. Fortunately, the normalized parabolic induction functor
$I_P^G : \Rep (M) \to \Rep (G)$ does, and also respects non-temperedness:

\begin{prop}\label{prop:3.6}
Let $\pi \in \Rep (M)$ be of finite length. Then $I_P^G (\pi)$ is tempered if and only if
$\pi$ is tempered.
\end{prop}
\begin{proof}
The if part is well-known, see \cite[Lemme III.2.3]{Wal} or \cite[Lemme VII.2.2]{Ren}.

By conjugating $P,M$ and $\pi$, we can achieve that $P \supset P_0$ and $M \supset M_0$.
Recall from \cite[Proposition III.2.2]{Wal} that $\pi$ is tempered if and only if, for
every parabolic subgroup $P_1 = M_1 U_1$ of $M$  with $M_1 \supset M_0$ and 
every $A_{M_1}$-weight $\chi$ of $J^M_{P_1}(\pi)$:
\begin{equation}\label{eq:3.26}
\log |\chi| \in \overline{{}^+ \mf a^M_{P_1}} . 
\end{equation}
Moreover, it is equivalent to impose this condition for all $P_1$ such 
that $P_1 = M$ or $P_1$ is a standard maximal parabolic subgroup of $M$.

To show that $I_P^G$ preserves non-temperedness, it suffices to consider the case that $P$ is
a standard maximal parabolic subgroup of $G$. Namely, there exists a chain of parabolic subgroups
\[
P \subset P_1 = M_1 U_1 \subset \cdots \subset P_n = M_n U_n \subset G = M_{n+1} 
\]
such that every $P_{i-1} \cap M_i$ is a maximal parabolic subgroup of $M_i$. If we can prove
that each $I_{P_{i-1} \cap M_i}^{M_i}$ preserves non-temperedness, the transitivity of
parabolic induction \cite[Lemme VI.1.4]{Ren} implies that $I_P^G$ does so as well.

Since $I_P^G$ is an exact functor \cite[Th\'eor\`eme VI.1.1]{Ren}, we may furthermore assume
that $\pi$ is irreducible. So, we suppose that $\pi$ is irreducible and not tempered, and
(contrary to what we want to prove) that $I_P^G (\pi)$ is tempered. Let us consider the 
$Z(G)$-character of $\pi$. It is also the $Z(G)$-character of $I_P^G (\pi)$. Since 
$I_P^G (\pi)$ is tempered, its central character is unitary \cite[Corollaire VII.2.6]{Ren}. 

We claim that the $A_M$-character $\zeta$ of $\pi$ must also be unitary. 

Suppose it is not, and consider its absolute value $|\zeta| \in X_\nr (M) \setminus \{1\}$.
Let $\alpha \in R(G,M_0)$ be the unique simple root of $(G,M)$ and let $s_\alpha \in W(G,M_0)$
be the associated reflection. The length of $s_\alpha$ in $W(G,M_0)$ is one, so it is a
minimal length representative for a double coset in $W(M,M_0) \setminus W(G,M_0) \backslash 
W(M,M_0)$. By Bernstein's geometric lemma \cite[Th\'eor\`eme VI.5.1]{Ren} both $\zeta$ and
$s_\alpha \zeta$ occur as $A_M$-weights of $J^G_P (I_P^G (\pi))$. Since $M$ is a maximal
Levi subgroup of $G$ and $|\zeta |_{A_G} = | \zeta |_{Z(G)} = 1$, both $\log |\zeta|$ and
$\log |s_\alpha \zeta|$ lie in the one-dimensional vector space $\mf a_M^G$. The reflection
$s_\alpha$ acts as $-1$ on $\mf a_M^G$, so $|s_\alpha \zeta| = |\zeta|^{-1}$. As $|\zeta| 
\neq 1$, it is not possible that both $\log |\zeta|$ and $\log |s_\alpha \zeta|$ lie in the 
cone $\overline{{}^+ \mf a^G_P}$. From \eqref{eq:3.26} we see that this contradicts the 
temperedness of $I_P^G (\pi)$. Consequently $|\zeta|$ must be 1, and $\zeta$ must be unitary.

Now we invoke the non-temperedness of $\pi$. By \cite[Proposition III.2.2.iii]{Wal} there
exists a standard parabolic subgroup $P' = M' U'$ of $G$ such that:
\begin{itemize}
\item $M' = M$ or $M'$ is a maximal Levi subgroup of $M$;
\item $J^M_{P' \cap M}(\pi)$ has an $A_{M'}$-weight $\chi$ with 
$\log |\chi| \in \mf a_{M'}$ not in $\overline{{}^+ \mf a^M_{P' \cap M}}$.
\end{itemize}
When $M' = M$, then $\chi = \zeta$ and the above claim says that $|\chi| = 1$. That would
be in contradiction with the second bullet. 

Hence $M' \subsetneq M$. As $\chi |_{A_M} = \zeta_{A_M}$ is unitary, $\chi$ is of the form
$c_\beta \beta$, where $\beta \in R(M,M_0)$ is the unique simple root for $(M,M')$. 
Then the second bullet says that $c_\beta < 0$. By Bernstein's geometric lemma 
\cite[Th\'eor\`eme VI.5.1]{Ren} $\chi$ also a $A_{M'}$-weight of $J^G_{P'} (I_P^G (\pi))$.
As $\beta \in \overline{{}^+ \mf a^G_{P'}}$ (a potentially larger cone than
$\overline{{}^+ \mf a^M_{P' \cap M}}$), $c_\beta \beta \notin \overline{{}^+ \mf a^G_{P'}}$. 
Together with \eqref{eq:3.26} this shows that $I_P^G (\pi)$ cannot be tempered.
\end{proof}

Let $L = \mc L (F)$ be a Levi subgroup $G$ and let $(\sigma,V_\sigma) \in \Irr (L)$ be
an irreducible tempered supercuspidal $L$-representation. Let $X_\nr (L)$ be the group
of unramified characters $L \to \C^\times$ and let $X_\unr (L)$ be the subgroup of unitary
unramified characters. Recall that the inertial equivalence class 
of the pair $(L,\sigma)$ consists of all pairs of the form 
\[
(g L g^{-1}, (g \cdot \sigma) \otimes \chi), \text{ where } g \in G 
\text{ and } \chi \in X_\nr (g L g^{-1}).
\]
We write $\mf s = [L,\sigma]_G$ and call this an inertial equivalence class for $G$.
It gives rise to a subset $\Irr (G)^{\mf s} \subset \Irr (G)$, namely all those irreducible
smooth $G$-representations whose supercuspidal support lies in $\mf s$. This in turn
is used to define a subcategory $\Rep (G)^{\mf s}$ of $\Rep (G)$, namely those smooth
$G$-representations all whose irreducible constituents lie in $\Irr (G)^{\mf s}$.
The Bernstein blocks $\Rep (G)^{\mf s}$ have better finiteness properties than $\Rep (G)$:

\begin{thm}\label{thm:2.9} \cite[\S 1.6]{Kaz} \ 
\enuma{
\item Let $M \subset G$ be a Levi subgroup containing $L$.
There exist tempered $\pi_{M,i} \in \Irr (M)^{[L,\sigma]_M} \; (i = 1,\ldots,\kappa_M)$, 
such that $\{ \pi_{M,i} \otimes \chi_M : i = 1, \ldots, \kappa_M, \chi_M \in X_\nr (M) \}$ 
is the collection of irreducible representations in $\Rep (M)^{[L,\sigma]_M}$ that are not 
isomorphic to the normalized parabolic induction of representation of a proper Levi subgroup 
of $M$.
\item Let $M$ run through a set of representatives for the conjugacy classes of Levi
subgroups of $G$ containing $L$. Then the set
\[
\bigcup\nolimits_M \{ I_P^G ( \pi_{M,i} \otimes \chi_M ) : 
i = 1, \ldots, \kappa_M, \chi_M \in X_\nr (M) \} 
\]
spans the Grothendieck group of the category of finite length representations in $\Rep (G)^{\fs}$.
} 
\end{thm}

Let $\mf B (G)$ be the set of all inertial equivalence classes $\mf s$ for $G$. 
By \cite[Corollaire 3.9]{BeDe} it is countably infinite (unless $G = 1$).
The Bernstein decomposition \cite[Theorem 2.10]{BeDe} says that 
\begin{equation}\label{eq:3.1}
\begin{array}{lll}
\Rep (G) & = & \prod_{\mf s \in \mf B (G)} \Rep (G)^{\mf s} , \\
\mc H (G) & = & \bigoplus_{\mf s \in \mf B (G)} \mc H (G)^{\mf s} ,
\end{array}
\end{equation}
where $\mc H (G)^{\mf s}$ is the two-sided ideal of $\mc H (G)$ for which
$\Mod (\mc H (G)^{\mf s} )$ is naturally equivalent with $\Rep (G)^{\mf s}$.

Let $\mc S (G)^{\mf s}$ (resp. $C_r^* (G)^{\mf s}$) be the two-sided ideal of
$\mc S (G)$ (resp. $C_r^* (G)$) generated by $\mc H (G)^{\mf s}$.
Upon completion, \eqref{eq:3.1} yields further Bernstein decompositions
\begin{equation}\label{eq:3.2}
\begin{array}{lll}
\mc S (G) & = & \bigoplus_{\mf s \in \mf B (G)} \mc S (G)^{\mf s} , \\
C_r^* (G) & = & \bigoplus_{\mf s \in \mf B (G)} C_r^* (G)^{\mf s} .
\end{array}
\end{equation}
The latter must be interpreted as a direct sum in the Banach algebra sense:
it is the completion of the algebraic direct sum with respect to the operator
norm of $C_r^* (G)$.

For a compact open subgroup $K$ of $G$ we let $\langle K \rangle$ be the
corresponding idempotent of $\mc H (G)$. Then 
\begin{equation}\label{eq:3.3}
\mc H (G,K) := \langle K \rangle \mc H (G) \langle K \rangle
\end{equation}
is the subalgebra of $K$-biinvariant functions in $\mc H (G)$. We define
$\mc S (G,K)$ and $C_r^* (G,K)$ analogously. For every compact open subgroup $K$ of 
$G$, $\mc S (G,K)$ is a Fr\'echet algebra \cite[Theorem 29]{Vig}. The Schwartz
algebra $\mc S (G)$ is their union (over all possible $K$), so it is an 
inductive limit of Fr\'echet algebras.

We will focus on one Bernstein block $\Rep (G)^{\mf s}$ of $\Rep (G)$.
By \cite[Corollaire 3.9]{BeDe} there exists a compact open subgroup $K_{\mf s}$
of $G$ such that every representation in $\Rep (G)^{\mf s}$ is generated by
its $K_{\mf s}$-fixed vectors. This leads to Morita equivalences
\begin{equation}\label{eq:3.4}
\begin{array}{ccccc}
\mc H (G)^{\mf s} & \sim_M & \mc H (G,K_{\mf s})^{\mf s} & 
:= & \mc H (G)^{\mf s} \cap \mc H (G,K_{\mf s}) \\
\mc S (G)^{\mf s} & \sim_M & \mc S (G,K_{\mf s})^{\mf s} & 
:= & \mc S (G)^{\mf s} \cap \mc S (G,K_{\mf s}) \\
C_r^* (G)^{\mf s} & \sim_M & C_r^* (G,K_{\mf s})^{\mf s} & 
:= & C_r^* (G)^{\mf s} \cap C_r^* (G,K_{\mf s}) .
\end{array}
\end{equation}

\subsection{The Plancherel isomorphism} \

We will describe the structure of $\mc S (G)^{\mf s}$ and $\mc S (G,K_{\mf s})^{\mf s}$ 
in more detail. Let $[L,\sigma]_L = T_{\mf s} \subset \Irr (L)$
be the set of $L$-representations of the form $\sigma \otimes \chi$ with 
$\chi \in X_\nr (L)$. Thus there is a finite covering of complex varieties
\begin{equation}\label{eq:3.5}
X_\nr (L) \to T_{\mf s} : \chi \mapsto \sigma \otimes \chi .
\end{equation}
Let $T_{\mf s,\un}$ be the subset of unitary representations in $T_{\mf s}$,
it is covered by $X_\unr (L)$ via \eqref{eq:3.5}. We write
\[
X_\nr (L,\sigma) = \{ \chi \in X_\nr (L) : \sigma \otimes \chi \cong \sigma \} .
\]
This is a finite subgroup of $X_\unr (L)$. The map \eqref{eq:3.5} induces an
isomorphism of algebraic varieties $X_\nr (L) / X_\nr (L,\sigma) \to T_{\mf s}$.

The group $W(G,L) = N_G (L) / L$ acts on $\Irr (L)$ by
\begin{equation}\label{eq:3.10}
(gL \cdot \pi) (l) = \pi (g l g^{-1}) .
\end{equation}
(The representation $gL \cdot \pi$ is only determined up to isomorphism.) 
This action stabilizes $X_\nr (L)$, the unitary representations in $\Irr (L)$
and the supercuspidal $L$-representations. Let $W_{\mf s}$ be the stabilizer
of $T_{\mf s}$ in $N_G (L) / L$. This group will play the same role as 
$W \Gamma$ did in Section \ref{sec:AHA}. The theory of the Bernstein centre
\cite[Th\'eor`eme 2.13]{BeDe} says that the centre of 
$\mc H (G,K_{\mf s})^{\mf s}$ is naturally isomorphic with 
$\mc O (T_{\mf s})^{W_{\mf s}} = \mc O (T_{\mf s} / W_{\mf s})$.

It will be convenient to lift everything from $T_{\mf s}$ to $X_\nr (L)$.
However, $W_{\mf s}$ does not act naturally on $X_\nr (L)$. To overcome
this and similar issues, we need the following lemma.

\begin{lem}\label{lem:3.1}
Let $p : T' \to T$ be a surjection between complex tori, with finite
kernel $K = \ker p$. Let $\Gamma$ be a finite group acting on $T$ by 
automorphisms of algebraic varieties (so $\Gamma$ need not fix $1 \in T$). 
Then there exists a canonical short exact sequence
\[
1 \to K \to \Gamma' \to \Gamma \to 1
\] 
and a canonical action of $\Gamma'$ on $T'$ which extends the multiplication 
action of $K$ on $T'$ and lifts the action of $\Gamma$ on $T$.
\end{lem}
\begin{proof}
Let $X$ be the character lattice of $T$. Then $\mc O (T) \cong \C [X]$ and
$\Gamma$ acts on $\mc O (T)$ by $(\gamma \cdot f) (t) = f (\gamma^{-1} t)$. Since
\[
\mc O (T)^\times = \{ z \theta_x : z \in \C^\times , x \in X \} \cong
\C^\times \times X ,
\]
$\Gamma$ also acts naturally on $X \cong \mc O (T)^\times / \C^\times$.
Let us denote this action by $l_\gamma : X \to X$. Notice that it defines an
action of $\Gamma$ on $T = \Hom_\Z (X,\C^\times)$ by algebraic group automorphisms.
The given action on $\mc O (T)$ can now be written as
\[
\gamma (z \theta_x) = z z_\gamma^{-1} (l_\gamma (x)) \theta_{l_\gamma (x)} ,
\]
for a unique $z_\gamma \in T$. Consequently the original action of $\Gamma$ on
$T$ can be expressed as 
\begin{equation}\label{eq:3.6}
\gamma (t) = z_\gamma l_\gamma (t) .
\end{equation}
The character lattice $X'$ of $T'$ contains $X$ with finite index
$|K|$, so $l_\gamma$ induces a canonical linear action of $\Gamma$ on $X'$,
also denoted $l_\gamma$. For every $\gamma \in \Gamma$ we choose a $z'_\gamma \in
p^{-1}(z_\gamma)$, and we define 
\[
\phi_\gamma : T' \to T' ,\quad \phi_\gamma (t') = z'_\gamma l_\gamma (t').
\]
Clearly $\phi_\gamma$ is a lift of \eqref{eq:3.6}, so for every $\gamma, \gamma'
\in \Gamma$ there exists a unique $z'_{\gamma,\gamma'} \in K$ with 
\begin{equation}\label{eq:3.7}
\phi_\gamma \circ \phi_{\gamma'} \circ \phi^{-1}_{\gamma \gamma'} (t') =
z'_{\gamma,\gamma'} t'  \quad \forall t' \in T' .
\end{equation}
Let $\Gamma'$ be the subgroup of Aut$(T')$ generated by the $\phi_\gamma \;
(\gamma \in \Gamma)$ and $K$. Then \eqref{eq:3.7} gives a canonical isomorphism
$\Gamma' / K \cong \Gamma$.

The only unnatural steps in the above argument are the choices of the 
$z'_\gamma$. Different choices would lead to different $z'_{\gamma,\gamma'}$ in
\eqref{eq:3.7}, but to the same group $\Gamma'$. Hence $\Gamma'$ is canonically
determined by the data $T,T'$ and $\Gamma$.
\end{proof}

Next we recall the Plancherel isomorphism for $\mc S (G)^{\mf s}$, as discovered 
by Harish-Chandra and worked out by Waldspurger. As induction data for $G$ we take
quadruples $(P,M,\omega,\chi)$, where
\begin{itemize}
\item $P$ is a parabolic subgroup of $G$ with a Levi factor $M$;
\item $(\omega,V_\omega) \in \Irr_{L^2}(M)$, the set of (isomorphism classes of)
irreducible smooth square-integrable modulo centre representations of $M$;
\item $\chi \in X_\nr (M)$.
\end{itemize}
To such a datum we associate the smooth $G$-representation $I_P^G (\omega \otimes \chi)$,
where $I_P^G$ denotes normalized parabolic induction. When $\chi$ is unitary, the 
$M$-invariant inner product on $(\omega \otimes \chi, V_\omega)$ induces a $G$-invariant 
inner product on $I_P^G (V_\omega)$, so $I_P^G (\omega \otimes \chi)$ is pre-unitary 
\cite[Proposition 3.1.4]{Cas}. However, $I_P^G (V_\omega)$ is only complete with respect 
to the associated metric if $\dim (I_P^G (V_\omega))$ is finite.

Let $(\check \omega, \check V_\omega)$ be the smooth contragredient of $\omega$ and put
\[
\mf L (\omega,P) = I_{P \times P}^{G \times G} (\omega \otimes \check \omega ) =
I_P^G (\omega) \otimes I_P^G (\check \omega) .
\]
Since $I_P^G (\check \omega)$ can be identified with the smooth contragredient of
$I_P^G (\omega)$ \cite[Proposition 3.1.2]{Cas}, $\mf L (\omega,P)$ can be regarded as the 
algebra of finite rank linear operators on $I_P^G (V_\omega)$.
Notice that for every $\chi \in X_\nr (M)$ we can identify $\mf L (\omega \otimes \chi,P)$
with $\mf L (\omega,P)$ as algebras. The inner product on $I_P^G (V_\omega)$ induces 
a *-operation on this algebra. That makes $\mc O (X_\nr (M)) \otimes \mf L (\omega,P)$
to a *-algebra with
\[
f^* (\chi) = f (\check \chi )^* . 
\]
There is a natural *-homomorphism
\begin{equation}\label{eq:3.8}
\begin{array}{ccc}
\mc H (G) & \to & \mc O (X_\nr (M)) \otimes \mf L (\omega,P) , \\ 
f & \mapsto & \big( \chi \mapsto I_P^G (\omega \otimes \chi) (f) \big) .
\end{array}
\end{equation}
We put $T_\omega = \{ \omega \otimes \chi : \chi \in X_\nr (M) \}$ 
and we record the covering map
\begin{equation}\label{eq:3.9}
X_\nr (M) \to T_\omega : \chi \mapsto \omega \otimes \chi . 
\end{equation}
The group 
\[
X_\nr (M,\omega) = \{ \chi \in X_\nr (M) : \omega \otimes \chi \cong \omega \}
\]
is finite, because all its elements must be trivial on $Z(M)$. All the fibres of \eqref{eq:3.9}
are isomorphic to $X_\nr (M,\omega)$. 

For every $k \in X_\nr (M,\omega)$ there exists a unitary $M$-intertwiner
$\omega \to \omega \otimes k$, unique up to scalars. The same map $V_\omega \to V_\omega$
also intertwines $\omega \otimes \chi$ with $\omega \otimes \chi k$, for any $\chi \in X_\nr (M)$.
Applying $I_P^G$, we get a family a $G$-intertwiners 
\begin{equation}
\pi (k,\omega, \chi) : I_P^G (\omega \otimes \chi) \to I_P^G (\omega \otimes \chi k), 
\end{equation}
independent of $\chi$ and unitary when $\chi \in X_\unr (M)$. Let 
\[
\check \pi (k,\omega, \chi) : I_P^G (\check \omega \otimes \check \chi)) \to 
I_P^G (\check \omega \otimes \check{\chi k})
\]
be the inverse transpose $\pi (k,\omega, \chi)$. Since $\pi (k,\omega, \chi)$
is unique up to scalars, 
\begin{equation}
I(k,\omega, \chi) := \pi (k,\omega, \chi) \otimes \check \pi (k,\omega, \chi)
\in \Hom_{G \times G} \big( \mf L (\omega \otimes \chi,P) , \mf L (\omega \otimes \chi k,P) \big) 
\end{equation}
is canonical. Moreover it is unitary for $\chi \in X_\unr (M)$ and independent of $\chi$ 
as map between vector spaces.

Let $W(T_\omega)$ be the stabilizer of $T_\omega$ in $W(G,M) = N_G (M) / M$, with respect to
the action on $\Irr (M)$ as in \eqref{eq:3.10}. Then $W(T_\omega)$ acts naturally on $T_\omega$.
From Lemma \ref{lem:3.1} we get a group extension
\begin{equation}\label{eq:3.11}
1 \to X_\nr (M,\omega) \to W'(T_\omega) \to W(T_\omega) \to 1 
\end{equation}
and an action of $W' (T_\omega)$ on $X_\nr (M)$ compatible with the covering \eqref{eq:3.9}.
In \cite{Wal} the representations $\omega \otimes \chi$ and $\omega \otimes \chi k$ are often
not distinguished. The introduction of $W' (T_\omega)$ and of the $I(k,\omega, \chi)$ 
allows us to compare $I_P^G (\omega \otimes \chi)$ and $I_P^G (\omega \otimes \chi k)$ in a 
systematic way. From \cite[\S VI.1]{Wal} one can see that actually our setup is just another
way to keep track of all the ingredients of \cite{Wal}.

The following results are proven in \cite[Paragraphe V]{Wal}. For $w' \in W' (T_\omega)$
there exist unitary $G$-intertwining operators
\begin{equation}\label{eq:3.12}
\pi (w',\omega, \chi) : I_P^G (\omega \otimes \chi) \to I_P^G (\omega \otimes w' (\chi)) 
\quad \chi \in X_\unr (M),
\end{equation}
unique up to scalars. These give canonical unitary intertwiners
\begin{equation}\label{eq:3.13}
I(w',\omega, \chi) = \pi (w',\omega, \chi) \otimes \check \pi (w',\omega, \chi)
\in \Hom_{G \times G} \big( \mf L (\omega \otimes \chi,P) , \mf L (\omega \otimes w'(\chi),P) \big)
\end{equation}
with the following properties \cite[Lemme V.3.1]{Wal}:
\begin{itemize}
\item as functions of $\chi$, $\pi (w',\omega, \chi)$ and $I(w',\omega, \chi)$ are
continuous with respect to the Zariski topology on the real algebraic variety $X_\unr (M)$;
\item $I(w'_2,\omega, w'_1 (\chi)) \circ I(w'_1,\omega, \chi) = 
I(w'_2 w'_1,\omega, \chi)$ for $w'_1, w'_2 \in W' (T_\omega)$.
\end{itemize}
The properties of the intertwining operators \eqref{eq:3.12} imply that, for every $g \in G ,\;
\omega \in \Irr_{L^2}(M) ,\; \chi \in X_\nr (M)$ and every parabolic subgroup $P' \subset G$
with Levi factor $g M g^{-1}$, the representations $I_P^G (\omega \otimes \chi)$ and 
$I_{P'}^G (g \cdot \omega \otimes g \cdot \chi)$ have the same irreducible 
subquotients, counted with multiplicity \cite[Corollary 2.7]{SolPadicHP}. 

We remark that $I(w',\omega, \chi)$ is called ${}^\circ c_{P\mid P}(w' ,\omega \otimes \chi)$ 
in \cite{Wal}. The intertwining operators \eqref{eq:3.13} give rise to an action of 
$W'(T_\omega)$ on the algebra
\[
C^\infty (X_\unr (M)) \otimes \mf L (\omega,P) \quad \text{ by } \quad
(w' \cdot f) (w' \chi) = I(w',\omega, \chi) f (\chi) .
\]
We fix a parabolic subgroup $P_L$ with Levi factor $L$, and we recall that $\mf s = [L,\sigma]_G$. 
To study representations in the Bernstein block $\Rep (G)^{\mf s}$, it suffices to consider 
induction data such that $P \supset P_L, M \supset L$ and the cuspidal support of $\omega$ 
lies in $[L,\sigma]_M$. Then $W (T_\omega)$ can be regarded as a subgroup of $W_{\mf s}$.

Choose representatives for the $G$-association classes of parabolic subgroups $P$ containing
$P_L$. Notice that every such $P$ has a unique Levi factor $M$ containing $L$. We also
choose representatives $\omega$ for the action of $W_\fs \ltimes X_\unr (M)$ on $\Irr_{L^2}(M) \cap
\Irr (M)^{\mf s_M}$, where $\mf s_M = [L,\sigma]_M$. We denote the resulting set of representative 
triples by $(P,M,\omega) / \sim$. Harish-Chandra established the following Plancherel isomorphism,
see \cite[Theorem 8.9]{SZ} for an alternative proof.

\begin{thm}\label{thm:3.2}
\textup{\cite[Th\'eor\`eme VII.2.5]{Wal}} \ \\
The maps \eqref{eq:3.8} induces isomorphisms of topological *-algebras
\[
\begin{array}{lll}
\mc S (G)^{\mf s} & \to & \bigoplus_{(P,M,\omega) / \sim} 
\big( C^\infty (X_\unr (M)) \otimes \mf L (\omega,P) \big)^{W' (T_\omega)} , \\
\mc S (G,K_{\mf s})^{\mf s} & \to & \bigoplus_{(P,M,\omega) / \sim} \Big( C^\infty (X_\unr (M)) 
\otimes \End_\C \big( I_P^G (V_\omega)^{K_{\mf s}} \big) \Big)^{W' (T_\omega)} .
\end{array}
\]
\end{thm}

Plymen \cite{Ply1} showed that Theorem \ref{thm:3.2} has a natural extension to
$C^*$-algebras. Let $\mf H (\omega \otimes \chi, P)$ be the Hilbert space completion of
$I_P^G (V_{\omega \otimes \chi}) = I_P^G (V_\omega)$ and let $\mf K (\omega \otimes \chi, P)$
be the $C^*$-algebra of compact operators on $\mf H (\omega \otimes \chi, P)$.

\begin{thm}\label{thm:3.3}
The  maps \eqref{eq:3.8} induces isomorphisms of $C^*$-algebras
\[
\begin{array}{lll}
C_r^* (G)^{\mf s} & \to & \bigoplus_{(P,M,\omega) / \sim} 
C \big( X_\unr (M) ; \mf K (\omega,P) \big)^{W' (T_\omega)} , \\
C_r^* (G,K_{\mf s})^{\mf s} & \to & \bigoplus_{(P,M,\omega) / \sim} \Big( C (X_\unr (M)) 
\otimes \End_\C \big( I_P^G (V_\omega)^{K_{\mf s}} \big) \Big)^{W' (T_\omega)} .
\end{array}
\]
\end{thm}
\begin{proof}
First we note that we have intertwining operators associated to the group $W' (T_\omega)$,
instead of $W (T_\omega)$ in \cite{Ply1,Wal}. The reason for this is explained after
\eqref{eq:3.11}. In view of Theorem \ref{thm:3.2}, it only remains to prove that completing
with respect to the operator norm of $C_r^* (G)$ boils down to replacing 
$C^\infty (X_\unr (M)) \otimes \mf L (\omega,P)$ by $C \big( X_\unr (M) ; \mf K (\omega,P) \big)$.
This is shown in \cite[Theorem 2.5]{Ply1}.
\end{proof}

\subsection{The space of irreducible representations} \

Like in Paragraph \ref{par:2.2}, we need more information about the space of all irreducible
smooth $G$-representations (tempered or not). 
Suppose that $\pi \in \Irr (G)$ has supercuspidal support $\sigma \otimes \chi$, where
$\sigma \in \Irr_{L^2}(M)$ and $\chi \in X_\nr (M)$. Then $\log |\chi| \in \mf a_M$, and its
image in $\mf a_0$ is uniquely determined, up to $W (G,M_0)$, by $\pi$. In other words, 
\begin{equation}\label{eq:3.24}
cc (\pi) := W(G,M_0) \log |\chi|
\end{equation}
is an invariant of $\pi$. Since the norm on $\mf a_0$ comes from a $W(G,M_0)$-invariant 
inner product, $\norm{cc (\pi)} := \norm{ \log |\chi| }$ is well-defined.

\begin{thm}\label{thm:3.4}
\enuma{
\item For every $\pi \in \Irr (G)$ there exists an induction datum\\
$(P,M,\omega,\chi)$, unique up to conjugation, such that $\pi$ is a constituent of 
$I_P^G (\omega \otimes \chi)$ and $\norm{cc (\omega)}$ is maximal for this property. 
In this case we call $\pi$ a 
\emph{Langlands constituent} of $I_P^G (\omega \otimes \chi)$.
\item $\pi$ is tempered if and only if $\chi \in X_\unr (M)$.
\item For any induction datum $(P,M,\omega,\chi)$, every constituent of $I_P^G (\omega 
\otimes \chi)$ is either a Langlands constituent or a constituent of some $I_{P'}^G (\omega' 
\otimes \chi')$ with $\norm{cc (\omega')} > \norm{cc (\omega)}$.
\item Suppose that $L$ is a standard Levi subgroup and that $\pi \in \Irr (G)^\fs$,
where $\fs = [L,\sigma]_G$. Then we can choose $(P,M,\omega,\chi)$ from part (a) such
that $P \supset P_0 ,\; M \supset L$ and $\omega \in \Irr (M)^{[L,\sigma]_M}$.
}
\end{thm}
\begin{proof}
(a) See \cite[Theorem 2.15.b]{SolPadicHP}.\\
(b) This is a direct consequence of \cite[Proposition 2.14.b and Theorem 2.15.a]{SolPadicHP}.\\
(c) By \cite[Lemma 2.13]{SolPadicHP} $(P,M,\omega,\chi)$ is equivalent to an induction datum
$\xi^+$ in positive position. By \cite[Corollary 2.7]{SolPadicHP} $I_P^G (\omega \otimes \chi)$
has the same irreducible subquotients, counted with multiplicity, as the parabolically
induced representation associated to $\xi^+$. Therefore we may assume that $(P,M,\omega,\chi)$
is in positive position, that is, $P \supset P_0$ and $\log |\chi|$ lies in the closed 
positive cone in $\mf a_0$ (determined by $P_0$). 

Then \cite[Theorem 2.15.a]{SolPadicHP} says
that the Langlands constituents of $I_P^G (\omega \otimes \chi)$ are precisely its irreducible
quotients. Furthermore, by \cite[Proposition 2.15.a]{SolPadicHP} $I_P^G (\omega \otimes \chi)$
is a direct sum of representations of the form $I_Q^G (\tau \otimes |\chi|)$, where 
$(Q,\tau,\log |\chi|)$ is a datum for the Langlands classification of Irr$(G)$. Suppose that 
$\pi'$ is a constituent of $I_Q^G (\tau \otimes |\chi|)$, but not a quotient. 
By \cite[Lemma 2.11.a and Lemma 2.12]{SolPadicHP}, $\pi'$ is the Langlands quotient of 
$I_{Q'}^G (\tau' \otimes \nu')$, for a Langlands datum $(Q', \tau', \log \nu')$ with 
$Q' \supset Q$ and $\norm{cc (\tau')} > \norm{cc (\tau)}$.
By \cite[Proposition 2.15.a]{SolPadicHP} $\pi'$ is a Langlands constituent of $I_{P'}^G 
(\omega' \otimes \chi')$, for some induction datum $(P',M',\omega',\chi')$ with
\[
\norm{cc (\omega')} = \norm{cc (\tau')} > \norm{cc (\tau)} = \norm{cc (\omega)} . 
\]
(d) Let $P_i = M_i U_i$ be a standard parabolic subgroup and let $\overline{P_i}$ be
the unique parabolic subgroup with Levi factor $M_i$ that is opposite to $P_i$. Let 
$J^G_{\overline{P_i}} : \Rep (G) \to \Rep (M_i)$ be the normalized Jacquet restriction functor. 

From \cite[\S VII.4.2]{Ren} we recall how $\pi$ can be realized as a Langlands quotient.
Namely, we take $\overline{P_1}$ such that $J^G_{\overline{P_1}}(\pi)$ contains a
representation of the form $\tau \otimes \nu$, where $(P_1, \tau, \log \nu)$ is a Langlands
datum. By \cite[Proposition III.4.1]{Wal} there exists a parabolic subgroup $P_2$ with 
$P_0 \subset P_2 \subset P_1$, and a $\omega' \in \Irr_{L^2}(M_2)$, such that $\tau \otimes \nu$
is a direct summand of $I_{M_1 \cap P_2}^{M_1} (\omega' \otimes \nu)$. From the proof of
part (c) we see that $\pi$ is a Langlands constituent of $I_{P_2}^G (\omega' \otimes \nu)$. 
By the second adjointness theorem
\begin{equation}\label{eq:3.21}
0 \neq \Hom_G (I_{P_2}^G (\omega' \otimes \nu),\pi) \cong \Hom_{M_2} ( \omega' \otimes \nu,
J^G_{\overline{P_2}}(\pi)) .
\end{equation}
The cuspidal support of $J^G_{\overline{P_2}}(\pi)$ equals that of $\pi$, so $\omega' \otimes 
\nu$ also has cuspidal support in $[L,\sigma]_G$. More precisely, the cuspidal support of
$\omega' \otimes \nu$ is of the form $[L',\sigma']_{M_2}$, where $L'$ is a standard Levi 
subgroup of $G$ conjugate to $L$. Since every Levi subgroup containing $L$ is $G$-conjugate
to a standard Levi subgroup of $G$ containing $L$, we may replace $(P_2,M_2,\omega',\nu)$ 
by a $G$-conjugate $(P,M,\omega,\chi)$ with $M$ standard. Thus, we can arrange 
that the cuspidal support becomes $[L,\sigma'']_M$, for some cuspidal $\sigma'' \in \Irr (L)$. 
Then \eqref{eq:3.21} is also valid for $I_P^G (\omega \otimes \chi)$, since 
$I_{P_2}^G (\omega' \otimes \nu)$ is not affected by $G$-conjugation of $(P_2,M_2,\omega',\nu)$. 
Second adjointness tells us that $J^G_{\overline P}(\pi) \in \Irr (M)^{[L,\sigma]_M}$, so also
$\omega \otimes \chi \in \Irr (M)^{[L,\sigma]_M}$. Finally, we may replace $P$ by a standard
parabolic subgroup with Levi factor $M$, for this does not change the collection of
constituents of $I_P^G (\omega \otimes \chi)$.
\end{proof}

In \cite{SolPadicHP} Theorem \ref{thm:3.4} was used to study the geometry of Irr$(G)$, and
the relation with the subspace of tempered irreducible representations. For our purposes we 
need some aspects of that, and we need to know that for almost all induction data every
constituent is a Langlands constituent.

\begin{prop}\label{prop:3.5}
Let $(P,M,\omega,\chi)$ be an induction datum for $G$. 
\enuma{
\item For $r \in \R_{>-1}$, the number of inequivalent Langlands constituents of \\
$I_P^G (\omega \otimes \chi \, |\chi |^r)$ does not depend on $r$.
\item For all but finitely many $r \in \R_{>-1}$, $I_P^G (\omega \otimes \chi \, |\chi|^r )$
is completely reducible. Then all its irreducible subquotients are Langlands constituents.
}
\end{prop}
\begin{proof}
(a) All the induction data under consideration have the same stabilizer in $W' (T_\omega)$.
As $W' (T_\omega)$ is by construction the stabilizer of $T_\omega$ in the $\mc W$ from
\cite{SolPadicHP}, the statement is a special case of \cite[Lemma 2.16]{SolPadicHP}.\\
(b) As noted in the proof of Theorem \ref{thm:3.4}.c, $I_P^G (\omega \otimes \chi \, |\chi|^r)$ 
is a direct sum of representations of the form $I_Q^G (\tau \otimes |\chi|^{r+1})$, where 
$(Q = L U_Q,\tau,\log |\chi|^{r+1})$ is a Langlands datum. Hence it suffices to show that 
$I_Q^G (\tau \otimes |\chi|^{r+1})$ is irreducible for almost all $r \in \R_{>-1}$.
The conditions of a Langlands datum say that $\tau \in \Irr (L)$ is tempered and that
$\log |\chi|^{r+1} \in \mf a_L$ is strictly positive with respect to the roots for $(Q,L)$. 
This implies that, for every $r \in \R_{>-1}$ and every root $\alpha$ for $(G,L)$,
$\inp{\alpha}{\log |\chi|^{r+1}} \neq 0$. Now \cite[Th\'eor\`eme 3.2]{Sau} says that,
for $r \in \R_{>-1}$ close enough to $-1$, $I_Q^G (\tau \otimes |\chi|^{r+1})$ is irreducible.
On the algebraic family of finite length representations $I_Q^G (\tau \otimes |\chi|^{r+1})$
with $r \in \R$, irreducibility is an Zariski-open condition \cite[Proposition VI.8.4]{Ren}. 
Hence the locus of $r$'s for which this representation is reducible is a finite set.
\end{proof}
\vspace{3mm}

\section{Morita equivalences}
\label{sec:comparison}

In this section we will first formulate a long list of conditions for the objects we want
to compare. Assuming these conditions, we will prove a comparison theorem. In the next
sections we will check that these conditions are fulfilled in cases of interest. 

\subsection{Conditions and first consequences} \

We keep the notations from the previous section. 

\begin{cond}\label{cond:Morita}
For every parabolic subgroup $P$ with $P_L \subset P \subset G$ and Levi factor $M \supset L$,
an algebra $\mc H^M$ and a Morita equivalence
\[
\Phi_M : \Rep (M)^{\mf s_M} \to \Mod (\mc H^M) 
\]
are given. When $P' \supset P$ is another such parabolic subgroup, an algebra injection 
$\lambda_{M M'} : \mc H^M \to \mc H^{M'}$ is given, with the below properties.
\begin{itemize}
\item[(i)] The following diagram commutes:
\[
\xymatrix{
\Rep (M')^{\mf s_{M'}} \ar[rr]^{\Phi_{M'}} & & \Mod (\mc H^{M'}) \\
\Rep (M)^{\mf s_M} \ar[u]^{I_{P \cap M'}^{M'}} \ar[rr]^{\Phi_M} & &
\Mod (\mc H^M) \ar[u]^{\ind_{\lambda_{M M'}(\mc H^M)}^{\mc H^{M'}}}
}
\]
\item[(ii)] Let $\overline P$ be the parabolic subgroup of $G$ which has Levi factor $M$ and is 
opposite to $P$. Let $\mr{pr}_{\mf s_M} : \Rep (M) \to \Rep (M)^{\mf s_M}$ be the projection 
coming from the Bernstein decomposition for $M$. The following diagram commutes:
\[
\xymatrix{
\Rep (M')^{\mf s_{M'}} \ar[rr]^{\Phi_{M'}} 
\ar[d]^{\mr{pr}_{\mf s_M} \circ J^{M'}_{\overline{P} \cap M'}} & & 
\Mod (\mc H^{M'}) \ar[d]^{\Res_{\lambda_{M M'}(\mc H^M)}^{\mc H^{M'}}}  \\
\Rep (M)^{\mf s_M} \ar[rr]^{\Phi_M} & & \Mod (\mc H^M) 
}
\]
\item[(iii)] If $P \subset P' \subset P'' \subset G$, then 
$\lambda_{M M''} = \lambda_{M' M''} \circ \lambda_{M M'}$.
\end{itemize}
\end{cond}

The Conditions \ref{cond:Morita} are quite general, in the sense that they do not involve
the structure of the algebras $\mc H^M$. We will see later that in many cases these 
conditions hold already by abstract functoriality principles.

The next series of conditions is much more specific though. For $P = M U$, let $R(M,L)$ be 
the set of roots of $M$ with respect to the maximal $F$-split torus $A_L$ in the centre of 
$L$. This is a root system when $L$ is a minimal $F$-Levi subgroup
of $G$. In general it is only an orthogonal projection of such a root system (but in many
cases encountered in the literature it is nevertheless a root system).
For $P \supset P_L$ we define the set of positive roots as $R^+ (M,L) = R(M \cap P_L,L)$,
and we call the minimal elements of this set the simple roots of $(M,L)$.

\begin{cond}\label{cond:Hecke} 
Assume Condition \ref{cond:Morita}.
\begin{itemize}
\item[(i)] $\mc H^G$ (or $(\mc H^G)^\op$) is an extended affine Hecke algebra 
$\mc H (\mc R,q) \rtimes \Gamma$.
\item[(ii)] All the $\mc H^M$ (or all the $(\mc H^M)^\op$) are parabolic subalgebras and 
the $\lambda_{M M'}$ are inclusions of parabolic subalgebras.
\item[(iii)] Consider the bijection
\[
\Phi_L : X_\nr (L) / X_\nr (L,\sigma) \cong \Irr (L)^{\mf s_L} \; \to \; \Irr (\mc H^L) \cong T 
\]
and its differential d$\Phi_L : X^* (L) \otimes_\Z \C \to Y \otimes_\Z \C$. 

Then d$\Phi_L^{-1}$ maps the positive coroots $R^{\vee +}$ for $\mc H (\mc R,q)$ to $R^+(G,L)$, 
and d$\Phi_L^{-1} (\Q R^\vee)$ has a $\Q$-basis consisting of simple roots of $(G,L)$.
\item[(iv)] Suppose that $Q \subset \Delta$ and d$\Phi_L \big( \Q R(M,L) \big) \cap R^\vee = 
R_Q^\vee$. Then $\mc H^M = \mc H^Q \rtimes \Gamma_M$ for some $\Gamma_M \subset \Gamma (Q,Q)$. 
If moreover d$\Phi_L \big( \Q R(M,L) \big) = \Q Q^\vee$, then $\Gamma_M$ satisfies 
Condition \ref{cond:Gamma} for $Q$. 
\end{itemize}
\end{cond}
In practice the positivity part of Condition \ref{cond:Hecke}.iii is innocent. Namely, 
usually one starts by fixing a minimal parabolic subgroup, and proves statements with that 
parabolic as the standard one. Suppose that $R(G,L)$ is a root system and that all the 
above conditions hold, except the positivity part of Condition \ref{cond:Hecke}.iii. Then 
d$\Phi_L^{-1} (\Q R^\vee) \cap R (G,L)$ is a parabolic root subsystem of $R(G,L)$, so it is
conjugate under $W(G,L)$ to a standard parabolic root subsystem, say $R(M,L)$. Applying an 
element of $W(M,L)$, we can moreover arrange that the image of $R^{\vee +}$ consists of 
positive roots. Equivalently, with respect to a different parabolic subgroup $P'_L$ of $G$ 
with Levi factor $L$, Condition \ref{cond:Hecke}.iii is fulfilled. 

Then we restart the whole procedure with $P'_L$ instead
of $P_L$, and the same arguments as before will also prove the required positivity statements. 
This applies to all the examples discussed in Sections \ref{sec:types} and \ref{sec:progen}.

We draw some first consequences for the above conditions.

\begin{lem}\label{lem:4.5}
Assume Conditions \ref{cond:Morita} and \ref{cond:Hecke}.
\enuma{
\item There exists a canonical surjective homomorphism of complex tori $\Phi_\nr : X_\nr (L) 
\to T$, with finite kernel $X_\nr (L,\sigma)$. 
\item When d$\Phi_L( \Q R(M,L)) \cap R^\vee = R_Q^\vee$, the image of $X_\nr (M)$ 
under the map from part (a) is contained in $T^Q$. When moreover d$\Phi_L( \Q R(M,L)) =
\Q Q^\vee$, the image of $X_\nr (M)$ equals $T^Q$.
\item For all $\omega \in \Rep (M)^{\fs_M}$ and $\chi \in X_\nr (M)$: 
$\Phi_M (\omega \otimes \chi) = \Phi_M (\omega) \otimes \Phi_\nr (\chi)$.
} 
\end{lem}
\begin{proof}
(a) The map $\chi \mapsto \sigma \otimes \chi$ induces an isomorphism of algebraic varieties
\[
X_\nr (L) / X_\nr (L,\sigma) \to
\Irr (L)^{\mf s_L} = \{ \sigma \otimes \chi : \chi \in X_\nr (L) \} .
\]
By Condition \ref{cond:Hecke}.(ii) $\Phi_L$ gives a bijection
$\Irr (L)^{\mf s_L} \to \Irr (\mc H^L) = T$. This lifts to a surjective group homomorphism 
\begin{equation}\label{eq:4.1}
\Phi_\nr : X_\nr (L) \to T \quad \text{with} \quad \Phi_L (\sigma \otimes \chi) = 
\Phi_L (\sigma) \otimes \Phi_\nr (\chi) .
\end{equation}
(b) For every $Q \subset \Delta$ we defined the subtorus
\[
T^Q = \{ t \in T : t(x) = 1 \; \forall x \in \Q Q \cap X \} 
\]
of $T$. Using Condition \ref{cond:Hecke}.(iv) we can write
\[
X_\nr (M) = \{ \chi \in X_\nr (L) : \chi = 1 \text{ on } \Q R(M,L)^\vee \cap X^* (X_\nr (L)) \} .
\]
The relation between $M$ and $Q$ shows that the preimage of $\Phi_\nr^{-1} (T^Q)$ contains
$X_\nr (M)$. When d$\Phi_L( \Q R(M,L)) = \Q Q^\vee$, $\Phi_L$ also induces a bijection between
$\Q R (M,L)^\vee$ and $\Q Q$, and $X_\nr (M)$ is the full preimage of $T^Q$.\\
(c) The kernel of $\Phi_\nr : X_\nr (M) \to T^Q$ is
\[
X_\nr (M,\sigma) := X_\nr (L,\sigma) \cap X_\nr (M) . 
\]
Then $X_\nr (M,\sigma)$ acts on $X_\nr (M)$ by translations and $\mc G_{Q,\sigma}$ acts
on $T^Q \cong \\ X_\nr (M) / X_\nr (M,\sigma)$. By Lemma \ref{lem:3.1} there exists a
canonical short exact sequence
\begin{equation}\label{eq:4.2}
1 \to X_\nr (M,\sigma) \to \mc G'_{Q,\sigma} \to \mc G_{Q,\sigma} \to 1 ,
\end{equation}
such that the action of $\mc G'_{Q,\sigma}$ on $X_\nr (M)$ lifts that of $\mc G_{Q,\sigma}$
on $T^Q$. For $\omega \in \Rep (M)^{\fs_M}$ Condition \ref{cond:Morita}.(ii) and \eqref{eq:4.1}
imply that
\begin{align}\label{eq:4.15}
& \Res^{\mc H^M}_{\lambda_{L M} (\mc H^L)} (\Phi_M (\omega \otimes \chi) ) =
\Phi_L ( J^M_{\overline{P_L} \cap M}(\omega \otimes \chi) ) \\
\nonumber & = \Phi_L (J^M_{\overline{P_L} \cap M} (\omega) \otimes \chi )  = 
\Phi_L (J^M_{\overline{P_L} \cap M} (\omega)) \otimes \Phi_L (\chi) \\
\nonumber & = \Res^{\mc H^M}_{\lambda_{L M}(\mc H^L)}(\Phi_M (\omega)) \otimes \Phi_\nr (\chi) = 
\Res^{\mc H^M}_{\lambda_{L M}(\mc H^L)}(\Phi_M(\omega) \otimes \Phi_\nr (\chi)) .
\end{align}
When $\omega$ is irreducible, $\Phi_M (\omega \otimes \chi)$ lies in the same connected component 
of $\Irr (\mc H^M)$ as $\Phi_M (\omega)$, so \eqref{eq:4.15} shows that it is an unramified twist of 
$\Phi_M (\omega)$. Hence
\begin{equation}\label{eq:4.14}
\Phi_M (\omega \otimes \chi) = \Phi_M (\omega) \otimes \Phi_\nr (\chi) \quad \text{when }
\omega \text{ is irreducible.} 
\end{equation}
Using the invertibility of $\Phi_M$, both sides of \eqref{eq:4.14} define a group action of 
$X_\nr (M)$ on Mod$(\mc H^M)$, by exact functors with commute with inductive (and projective) limits. 
Since these actions agree on irreducible representations, they agree on all representations.
\end{proof}

\begin{lem}\label{lem:4.3}
Assume Conditions \ref{cond:Morita} and \ref{cond:Hecke} and suppose that \\
d$\Phi_L \big( \Q R(M,L) \big) \cap R^\vee = R_Q^\vee$. 
\enuma{
\item The map $\Phi_L$ induces a group isomorphism $W_{\fs_M} \to W (R_Q) \Gamma_M$.
\item $W_{\fs_M}$ fixes $X_\nr (M)$ pointwise and, when moreover $\textup{d} \Phi_L \big( 
\Q R(M,L) \big) = \Q Q^\vee$, $W(R_Q) \Gamma_M$ fixes $T^Q$ pointwise.
}
\end{lem}
\begin{proof}
(a) It suffices to prove this when $M = G$.
By \cite[Th\'eor\`eme 2.13]{BeDe}, the centre of the category Rep$(G)^\fs$ is
\[
\mc O ( \Irr (L)^{\fs_L} )^{W_\fs} = \mc O (\Irr (L)^{\fs_L} / W_\fs ) .
\]
The pointwise fixator of $X_\nr (L)$ in $N_G (L)$ is $Z_G (A_L) = Z_G (Z (L)^\circ ) = L$. 
Since $W_\fs \subset N_G (L) / L$, it acts faithfully on $X_\nr (L)$ by algebraic group 
automorphisms. Hence $W_\fs$ also acts faithfully on $\Irr (L)^{\fs_L}$. 

By \eqref{eq:1.11} the centre of Mod$(\mc H \rtimes \Gamma)$ is
\begin{equation}\label{eq:4.13}
Z( \mc H \rtimes \Gamma ) = \mc O (T)^{W \Gamma} = \mc O (T / W \Gamma) , 
\end{equation}
provided that $W \Gamma$ acts faithfully on $T$. Clearly $W$ acts faithfully on $T$.
By assumption every $\gamma \in \Gamma$ acts on $R$ by a diagram automorphism, so it cannot 
act on $T$ as any nontrivial element of $W$. Hence, to check that $W \Gamma$ acts faithfully
on $T$, it suffices to do so for $\Gamma$. 

In view of \eqref{eq:2.5}, the isomorphism $\Phi_L : \Irr (L)^{\fs_L} \to T$ implies that
$\Gamma_\es$ is trivial. We recall from \cite[Th\'eor\`eme 3.2]{Sau} that that 
$I_{P_L}^G (\sigma \otimes \chi)$ is irreducible for $\chi$ in a Zariski-open nonempty 
subset of $X_\nr (L)$. If $\gamma \in \Gamma \setminus \{1\}$ would act trivially on $T$,
then so would the cyclic group $\langle \gamma \rangle$ generated by it. In that case
\[
\ind_{\mc H^\es}^{\mc H \rtimes \Gamma} \C_t = \ind_{\mc H^\es \rtimes \langle \gamma 
\rangle}^{\mc H \rtimes \Gamma} (\C_t \otimes_\C \C \langle \gamma \rangle ) 
\]
would be reducible for all $t \in T$ (as $\C \langle \gamma \rangle$ is reducible).
That would contradict Condition \ref{cond:Morita}.i. So $W \Gamma$ acts faithfully
on $T$ and \eqref{eq:4.13} holds.

Now Condition \ref{cond:Morita} says that 
\[
Z( \mc H \rtimes \Gamma) = \mc O (T / W \Gamma) \cong \mc O (\Irr (L)^{\fs_L} / W_\fs ) .
\]
From this and Condition \ref{cond:Morita}.i, we deduce that $\Phi_L : \Irr (L)^{\fs_L} \to T$ 
induces a bijection 
\[
\Irr (L)^{\fs_L} / W_\fs \to T / W \Gamma. 
\]
On both sides the finite 
groups act faithfully by automorphisms of complex algebraic varieties. Consider the open
subvariety of $T$ (resp. of $\Irr (L)^{\fs_L}$) where the stabilizers in $W \Gamma$
(resp. in $W_\fs)$ are trivial. For such a $t \in T$ and $\gamma \in W \Gamma$, the equation
$\Phi_L (w \Phi_L^{-1} t) = \gamma t$ holds for a unique $w \in W_\fs$. This defines the
group isomorphism $W \Gamma \to W_\fs$.\\
(b) The first claim is trivial, because $W_{\fs_M} \subset N_M (L) / L$. The second claim 
follows directly from the first claim, part (a) and Lemma \ref{lem:4.5}.b.
\end{proof}

\begin{lem}\label{lem:4.4}
Assume 
\[
M \mapsto \textup{d}\Phi_L \big( \Q R(M,L) \big) \cap \Delta^\vee 
\]
induces a bijection between:
\begin{itemize}
\item $W_\fs$-association classes of Levi subgroups $M \subset G$ such that $L \subset M$ 
and d$\Phi_L \big( \Q R(M,L) \big)$ equals the $\Q$-span of a subset of $\Delta^\vee$;
\item subsets of $\Delta^\vee$, modulo $\Gamma W$-association.
\end{itemize}
\end{lem}
\begin{proof}
Suppose that $M$ and $M'$ are such Levi subgroups, and that they are conjugate
by an element $w \in W_\fs$. Then the functors $I_P^G : \Rep (M)^{\fs_M} \to \Rep (G)^\fs$
and $I_{P'}^G : \Rep (M')^{\fs_{M'}} \to \Rep (G)^{\fs}$ have the same image, for
\[
I_{P'}^G \circ \text{Ad}(w)^* = \text{Ad}(w)^* \circ I_P^G \cong I_P^G . 
\]
With Condition \ref{cond:Morita}.i, this implies that $\mr{ind}_{\lambda_{MG}(\mc H^M )
}^{\mc H^G}$ and $\mr{ind}_{\lambda_{M'G}(\mc H^{M'})}^{\mc H^G}$ have the same image.
Condition \ref{cond:Hecke}.ii says that $\mc H^M = \mc H^Q \rtimes \Gamma_M$ and 
$\mc H^{M'} = \mc H^{Q'} \rtimes \Gamma_{M'}$ are parabolic subalgebras
of $\mc H^G$, and then Proposition \ref{prop:2.3} shows that they must be $\mc G$-associate.
Condition \ref{cond:Hecke}.iv implies that 
\[
Q^\vee = \textup{d}\Phi_L \big( \Q R(M,L) \big) \cap \Delta^\vee \quad \text{and} \quad 
Q'^\vee = \textup{d}\Phi_L \big( \Q R(M',L) \big) \cap \Delta^\vee
\]
are $W \Gamma$-associate. Hence the map of the lemma is well-defined on the given
equivalence classes.

Using the same notations as above, suppose that $Q^\vee$ and $Q'^\vee$ are $W \Gamma$-associate.
Then $\textup{d}\Phi_L \big( \Q R(M,L) \big) = \Q Q^\vee$ is $W \Gamma$-associate to
$\textup{d}\Phi_L \big( \Q R(M',L) \big) = \Q Q'^\vee$. By Lemma \ref{lem:4.3}, 
$\Q R(M,L)$ is $W_\fs$-associate to $\Q R(M',L)$. Hence $M$ and $M'$ are conjugate by an
element of $W_\fs$, showing that the map of the lemma is injective.

By Condition \ref{cond:Hecke}.iii, the subgroup $P_\fs \subset G$ generated by $P_L$ and
the root subgroups for roots in d$\Phi_L^{-1}(\Q R^\vee)$ is a parabolic subgroup of $G$.
The map of the statement sends the standard Levi factor $M_\fs$ of $P_\fs$ to $\Delta^\vee$. 

Suppose that the map is not surjective. Choose $\tilde Q^\vee \subset \Delta^\vee$ which 
does not lie in the image and is maximal for that property. Since $\Delta^\vee$ belongs 
to the image, $\tilde Q \neq \Delta$ and we can find $\alpha \in \Delta \setminus \tilde Q$ 
such that $Q^\vee := \tilde Q^\vee \cup \{\alpha^\vee \}$ does lie in the image. We write 
$\Gamma_{\tilde Q} = \Gamma_Q \cap \Gamma (\tilde Q,\tilde Q)$.

For every Levi subgroup $M' \subset M$ we choose finitely many representations $\pi_{M',i} \in 
\Irr (M')^{\fs_{M'}}$ as in Theorem \ref{thm:2.9}. Then the representations
\[
I_{P' \cap M}^M (\pi_{M',i} \otimes \chi_{M'}) \quad \text{with } \chi_{M'} \in X_\nr (M'),
\]
for all $M'$ and all possible $i$, span the Grothendieck group of $\Rep (M)^{\fs_M}$. 
Applying Condition \ref{cond:Morita}, we find that the representations
\begin{equation}
\mr{ind}_{\lambda_{M' M}(\mc H^{M'})}^{\mc H^M} \Phi_{M'} (\pi_{M',i} \otimes \chi_{M'}) 
\end{equation}
span the Grothendieck group Gr$(\mc H^M)$ of Mod$(\mc H^M)$. By Lemma \ref{lem:4.5} 
\[
\Phi_{M'} (\pi_{M',i} \otimes \chi_{M'}) = \Phi_{M'} (\pi_{M',i}) \otimes \Phi_\nr (\chi_{M'})
\]
and $\Phi_\nr (\chi_{M'}) \in T^{Q'}$. For $M' \neq M$ and $t \in T^{Q'} ,\; 
\Phi_{M'} (\pi_{M',i}) \otimes t$ is a representation of $\mc H^{Q'} \rtimes \Gamma_{M'}$ with
$Q'$ not $W \Gamma$-associate to $\tilde Q$. Hence the collection of representations
\begin{equation}\label{eq:3.23}
\Phi_{M} (\pi_{M,i}) \otimes t \quad \text{with } t \in T^Q
\end{equation}
spans the quotient
\begin{equation}\label{eq:3.22}
\mr{Gr}(\mc H^Q \rtimes \Gamma_M) \big/ \sum_{Q' \subsetneq Q, Q' \text{ not associate to } 
\tilde Q} \mr{ind}_{\lambda_{M' M}(\mc H^{Q'} \rtimes \Gamma_{M'})}^{\mc H^Q \rtimes \Gamma_M} 
\mr{Gr}( \mc H^{Q'} \rtimes \Gamma_{M'} ) .
\end{equation}
By Theorem \ref{thm:1.8}.a, $\mr{ind}_{\lambda_{\tilde Q M}(\mc H^{\tilde Q} \rtimes 
\Gamma_{\tilde Q})}^{\mc H^Q \rtimes \Gamma_M} \mr{Gr}( \mc H^{\tilde Q} \rtimes \Gamma_{\tilde Q})$
contributes an entire $T^{W(R_{\tilde Q}) \Gamma_{\tilde Q}}$-orbit of representations to 
\eqref{eq:3.22}. 

By Lemma \ref{lem:4.3} $T^Q \subset T^{W (R_Q) \Gamma_Q}$, which shows in particular that 
the translation part $z_\gamma$ of $\gamma$ is trivial for all $\gamma \in \Gamma_Q$.
As $W(R_{\tilde Q}) \Gamma_{\tilde Q} \subsetneq W(R_Q) \Gamma_Q$, we have
\[
T^{W(R_{\tilde Q}) \Gamma_{\tilde Q}} \supset T^Q .
\]
We want to see that the left hand side has higher dimension than the right hand side. 
By construction $W(R_{\tilde Q})$ fixes $T^{\tilde Q}$ pointwise.
The torus $T_1 = ( T_Q \cap T^{\tilde Q} )^\circ$ is one-dimensional, because 
$|Q \setminus \tilde Q| = 1$. For the same reason $\alpha : T_1 \to \C^\times$ is a surjection
with finite kernel. The group $\Gamma_{\tilde Q} = \Gamma_Q \cap \Gamma (\tilde Q,\tilde Q)$
stabilizes both $Q$ and $\tilde Q$, so fixes $\alpha$. Therefore $\Gamma_{\tilde Q}$ fixes
$T_1$ pointwise. It follows that $W(R_{\tilde Q}) \Gamma_{\tilde Q}$ fixes the torus 
$T^{\tilde Q} = T^Q T_1$ pointwise. 

Returning to \eqref{eq:3.23} and \eqref{eq:3.22}, we see now that the contribution from 
$\mc H^{\tilde Q} \rtimes \Gamma_{\tilde Q}$ encompasses at least one $T^{\tilde Q}$-orbit. 
But that is impossible, because the $i$'s in \eqref{eq:3.23} belong to a finite set and 
$\dim_\C (T^{Q'}) > \dim_\C (T^Q)$.
This contradiction entails that the map from the statement is surjective.
\end{proof}

\subsection{Preservation of temperedness} \

We will show that the above conditions imply that the Morita equivalences preserve temperedness 
and (under an extra condition) discrete series. For $\Phi_M^{-1}$ this is relatively easy.

\begin{lem}\label{lem:4.6}
Assume Conditions \ref{cond:Morita} and \ref{cond:Hecke}, and let $P = MU$ be a
parabolic subgroup containing $P_L$ such that d$\Phi_L (\Q R(M,L)) = \Q Q^\vee$.
\enuma{
\item $\Phi_M^{-1}$ preserves temperedness of finite length representations.
\item $\Phi_M^{-1}$ sends finite dimensional essentially discrete series 
$\mc H^M$-representations to essentially square-integrable $M$-representations.
\item Suppose that $\pi' \in \Mod (\mc H^M)$ has finite dimension, is tempered, essentially 
discrete series and factors through $\psi_t : \mc H^Q \rtimes \Gamma_M \to \mc H_Q \rtimes 
\Gamma_M$ for some $t \in T^Q$. Then the $M$-representation $\Phi_M^{-1}(\pi')$ is 
square-integrable modulo centre.
}
\end{lem}
\begin{proof}
(a) Since every irreducible $\mc H^M$-module has finite dimension, $\Phi_M$ restricts
to an equivalence between finite length representations on the one hand, and the finite
dimensional modules on the other hand.

Let $\pi \in \Rep (M)^{\mf s_M}$ be of finite length, and recall the criterion for
temperedness from \cite[Proposition III.2.2]{Wal} and \eqref{eq:3.26}. As the supercuspidal
support of $\pi$ is contained in $[L,\sigma]_G$, it is equivalent to impose these conditions 
only with respect to the parabolic subgroup $P_L = L U_L$ \cite[Proposition 1.2.i]{Hei2}.
Let $\overline{P_L}$ be the parabolic subgroup opposite to $P_L$. Then $\overline{P_L} \cap M$ 
is opposite to $P_L \cap M$. The above condition (for $P_L$) is equivalent to: 
\begin{equation}\label{eq:3.17}
\log |\chi| \in \overline{{}^+ \mf a^M_{\overline{P_L} \cap M}} = 
\big\{ \sum\nolimits_{\alpha \in R(P_L \cap M,L)} c_\alpha \alpha |_{A_L} : c_\alpha \leq 0 \big\} 
\end{equation}
for every $A_L$-weight $\chi$ of $J^M_{M \cap \overline{P_L}} \pi$.

By the assumption on $M$ and Condition \ref{cond:Hecke}.iii, 
d$\Phi_L^{-1} :\Q Q^\vee \to \Q R (M,L)$ is a linear bijection which sends 
\begin{equation}\label{eq:3.18}
\mf a^{-Q} = \big\{ \sum\nolimits_{\alpha \in Q} \lambda_\alpha \alpha^\vee : 
\lambda_\alpha \leq 0 \big\} \quad \text{to} \quad \overline{{}^+ \mf a^M_{\overline{P_L} \cap M}} .
\end{equation}
Suppose that $\Phi_M (\pi)$ is tempered. By definition, this means that
all $\mc H^L$-weights $t$ of $\Res^{\mc H^M}_{\lambda_{L M} (\mc H^L)}(\Phi_M (\pi))$ satisfy 
$\log |t| \in \mf a^{-Q}$. By Condition \ref{cond:Morita}.(ii) and \eqref{eq:3.18}, all 
$A_L$-weights $\chi$ of $J^M_{M \cap \overline{P_L}} \pi$ have 
$\log |\chi| \in \overline{{}^+ \mf a^M_{\overline{P_L} \cap M}}$. 
Thus \eqref{eq:3.17} says that $\pi$ is tempered.\\
(b) By \cite[Proposition III.1.1]{Wal} and arguments analogous to the above, $\pi$ is 
square-integrable modulo centre if and only if 
\begin{equation}\label{eq:3.19}
\log |\chi| \in {}^+ \mf a^M_{\overline{P_L} \cap M} = 
\big\{ \sum\nolimits_{\alpha \in R(P_L \cap M,L)} c_\alpha \alpha |_{A_L} : c_\alpha < 0 \big\} 
\end{equation}
for every $A_L$-weight $\chi$ of $J^M_{M \cap \overline{P_L}} \pi$. The criterium
for essential square-integrability then becomes 
\begin{equation}\label{eq:3.20}
\log |\chi| \in {}^+ \mf a^M_{\overline{P_L} \cap M} + \mf a_M .
\end{equation}
for every $A_L$-weight $\chi$ of $J^M_{M \cap \overline{P_L}} \pi$. Since the rank $|Q|$ of 
$R_Q$ equals the rank of $R(M,L)$ and d$\Phi_L^{-1}$ preserves positivity, it maps $\mf a^{--Q}$ 
(the interior of $\mf a^{-Q}$) to $\mf a^M_{\overline{P_L} \cap M}$. 
By Lemma \ref{lem:4.5}.b $\Phi_L^{-1}(T^Q) = X_\nr (M)$. 

Suppose that $\Phi_M (\pi)$ is essentially discrete series. By definition, this means that
all $\mc H^L$-weights $t$ of $\Res^{\mc H^M}_{\lambda_{L M} (\mc H^L)}(\Phi_M (\pi))$ satisfy 
$|t| \in \exp(\mf a^{--Q}) T^Q$. By the above and Condition \ref{cond:Morita}.ii, all 
$A_L$-weights of $J^M_{M \cap \overline{P_L}} \pi$ lie in \eqref{eq:3.20}. Hence
$\pi$ is essentially square-integrable.\\
(c) Recall that a $M$-representation is essentially square-integrable if its restriction to the
derived subgroup of $M$ is square-integrable. If a $M$-representation with a central character is 
tempered, then $Z(M)$ acts on its by a unitary character. Hence all tempered essentially 
square-integrable representations with a central character are square-integrable modulo centre.  

By parts (a) and (b) $\Phi_M^{-1}(\pi')$ is tempered and essentially square-integrable. Since
$\Phi_M^{-1}(\pi')$ lies in one Bernstein component $\Rep (M)^{\fs_M}$, the maximal compact 
subgroup of $Z(M)$ acts on it by a single character $\chi_0$, which is automatically unitary. 
Then $X_\nr (M)$ (resp. $X_\unr (M)$) parametrizes the extensions of $\chi_0$ to a character 
(resp. unitary character) of $Z(M)$. Lemma \ref{lem:4.5}.c says that $Z(M)$ acts on 
$\Phi_M^{-1}(\pi')$ by the character determined by $\chi_0$ and $\Phi_\nr^{-1}(t)$. Lemma 
\ref{lem:2.2}.b shows that $t \in T^Q_\un$ and $\Phi_\nr^{-1} (T_\un) = X_\unr (L)$, so
$Z(M)$ acts by a unitary character.
\end{proof}

Part (a) of Lemma \ref{lem:4.6} admits a quick generalization to all Levi subgroups that we
encounter. On the other hand, that is not possible for parts (b) and (c). In fact, for Levi
subgroups $M \subset G$ containing $L$ but not of the form as in Lemma \ref{lem:4.6}, 
$\Rep (M)^{\fs_M}$ contains no essentially square-integrable representations. We delay the 
proof of that claim to Proposition \ref{prop:4.1}. We note that in those cases 
$\Irr_{L^2}(\mc H^M)$ can still be nonempty.

\begin{lem}\label{lem:4.9}
Assume Conditions \ref{cond:Morita} and \ref{cond:Hecke}, and let $P' = M'U'$ be a
parabolic subgroup containing $P_L$ such that $\Q R(M',L) \cap \textup{d}\Phi_L^{-1}(R^\vee)$
does not span $\Q R(M',L)$. Then $\Phi_{M'}^{-1}$ preserves temperedness of finite length 
representations.
\end{lem}
\begin{proof}
By Condition \ref{cond:Hecke}.iii d$\Phi_L (\Q R(M',L))
\cap R^\vee$ is a standard parabolic root subsystem of $R^\vee$, that is, of the form
$R_Q^\vee$ for a unique $Q \subset \Delta$. By Lemma \ref{lem:4.4} there exists a Levi
subgroup $M \subset M'$ such that $L \subset M$ and d$\Phi_L(\Q R(M,L)) = \Q Q^\vee$. 

By Condition \ref{cond:Hecke}.iv $\mc H^{M'} = \mc H^Q \rtimes \Gamma_{M'}$ and 
$\mc H^M = \mc H^Q \rtimes \Gamma_M$, and by Condition \ref{cond:Hecke}.ii 
$\Gamma_{M'} \supset \Gamma_M$ and $\lambda_{M M'}$ is just the inclusion. The cone 
$T^{-Q} \subset T_\rs$ is the same for $\mc H^M, \mc H^Q$ and $\mc H^{M'}$. 
For any finite dimensional $\mc H^M$-module $V$:
\begin{equation}\label{eq:3.14}
\mr{Wt} \big( \ind_{\mc H^M}^{\mc H^{M'}} (V) \big) = \{ \gamma (t) : t \in \mr{Wt}(V),
\gamma \in \Gamma_{M'} \} .
\end{equation}
Since $\Gamma_{M'}$ preserves $T^{-Q} ,\; \ind_{\mc H^M}^{\mc H^{M'}} (V)$ 
is tempered if and only if $V$ is tempered.
Similarly, \eqref{eq:3.14} shows that a finite dimensional $\mc H^{M'}$-module $V'$ is 
tempered if and only if $\Res^{\mc H^{M'}}_{\mc H^M}(V')$ is tempered. We note also that 
\begin{equation}\label{eq:3.15}
\ind_{\mc H^M}^{\mc H^{M'}} \Res^{\mc H^{M'}}_{\mc H^M}(V') \cong
\C [\Gamma_{M'}] \otimes_{\C [\Gamma_M]} V' \cong \C [\Gamma_{M'} / \Gamma_M] \otimes_\C V' ,
\end{equation}
a $\mc H^Q \rtimes \Gamma_{M'}$-module for the diagonal action. Then \eqref{eq:3.15}
contains $V'$ as the direct summand $\C [ \Gamma_{M'} / \Gamma_{M'}] \otimes_\C V'$, 
and the restriction of \eqref{eq:3.15} to $\mc H^M$ is a direct sum of copies of
$\Res^{\mc H^{M'}}_{\mc H^M}(V')$. 

We recall from \cite[Lemme VII.2.2]{Ren} that $I^{M'}_{P \cap M'}$ always preserves 
temperedness. Consider a finite dimensional tempered module 
$V' \in \Mod (\mc H^{M'})$. By Lemma \ref{lem:4.6}.a the $M'$-representation
$I^{M'}_{P \cap M'} \circ \Phi_M^{-1} \circ \Res^{\mc H^{M'}}_{\mc H^M}(V')$
is tempered. By Condition \ref{cond:Morita} it is isomorphic to
\begin{equation}\label{eq:4.18}
\Phi_{M'}^{-1} \circ \ind_{\mc H^M}^{\mc H^{M'}} \circ \Res^{\mc H^{M'}}_{\mc H^M}(V') 
\cong I^{M'}_{P \cap M'} \circ \mr{pr}_{\mf s_M} \circ J^{M'}_{\overline{P} \cap M'} \circ
\Phi_{M'}^{-1} (V') ,
\end{equation}
Since $\Phi_{M'}^{-1}$ is an equivalence, \eqref{eq:3.15} shows that \eqref{eq:4.18} 
contains $\Phi_{M'}^{-1} (V')$ as a direct summand. So the latter is tempered as well.
\end{proof}

Let $M \supset L$ be a Levi subgroup of $G$ and write $\textup{d}\Phi_L (\Q R(M,L)) \cap 
R^\vee = R_Q^\vee$.  As family of parabolic subalgebras of $\mc H^M$ we take $\mc H^M$ and 
the $\mc H^{M_1}$ where $M \supsetneq M_1 \supset L$ and d$\Phi_L (\Q R(M_1,L)) = \Q Q_1$ for 
some $Q_1 \subsetneq Q$. By Lemma \ref{lem:4.4} every (proper) subset of $Q$ is obtained in this 
way. Recall the notion of Langlands constituents from Theorems \ref{thm:2.6} and \ref{thm:3.4}.

\begin{lem}\label{lem:4.8}
Let $M'$ be a Levi subgroup of $M$ containing $L$, such that \\
$\textup{d}\Phi_L (\Q R(M',L)) = \Q Q'^\vee$ for a subset $Q' \subset Q$. 
Let $\sigma \in \Irr_{L^2}(\mc H_{Q'} \rtimes \Gamma_{M'})$ and $t \in T^{Q'}$. Then
\[
\Phi_{M'}^{-1} (\sigma \otimes t) = \Phi_{M'}^{-1} (\sigma \otimes t \,|t|^{-1}) \otimes
\Phi_\nr^{-1}( |t| ) \quad \text{with} \quad 
\Phi_{M'}^{-1} (\sigma \otimes t \,|t|^{-1}) \in \Irr_{L^2}(M')^{\fs_{M'}} .
\]
The Morita equivalence $\Phi_M^{-1}$ restricts to a bijection between the Langlands
constituents of $\ind_{\mc H^{M'}}^{\mc H^M} (\sigma \otimes t)$ and those of
\[
\Phi_M^{-1} \big( \ind_{\mc H^{M'}}^{\mc H^M} (\sigma \otimes t) \big) \cong
I^M_{(M \cap P_L) M'} \big( \Phi_{M'}^{-1} (\sigma \otimes t \,|t|^{-1}) \otimes
\Phi_\nr^{-1}( |t| ) \big) .
\]
\end{lem}
\textbf{Remarks.}
By Lemma \ref{lem:4.5}.a $\Phi_\nr$ becomes injective when restricted to
unramified characters with values in $\R_{>0}$. Therefore $\Phi_\nr^{-1}( |t| ) \in
X_\nr (M')$ is well-defined.

When $Q' = Q$, Langlands constituents are not defined on the Hecke algebra side.
In that case the lemma must be interpreted differently. The functor 
$\ind_{\mc H^{M'}}^{\mc H^M} = \ind_{\mc H^Q \rtimes \Gamma_{M'}}^{\mc H^Q \rtimes \Gamma_M}$ 
preserves complete reducibility (by Clifford theory, as $\Gamma_M$ is finite). By Condition
\ref{cond:Morita}.i, so does $I^M_{(P_L \cap M)M'}$. In view of Proposition \ref{prop:3.5}.b, 
the lemma becomes true in the case $Q' = Q$, provided we declare that all irreducible 
subquotients of $\ind_{\mc H^{M'}}^{\mc H^M} (\sigma \otimes t)$ are Langlands constituents.
\begin{proof}
The alternative 
expression for $\Phi_{M'}^{-1} (\sigma \otimes t)$ comes from Lemma \ref{lem:4.5}.c.
By Lemma \ref{lem:2.2} $\sigma \otimes t \, |t|^{-1} \in \Irr_{L^2}(\mc H^{Q'} \rtimes
\Gamma_{M'})$, and by Lemma \ref{lem:4.6}.c 
$\Phi_{M'}^{-1} (\sigma \otimes t \,|t|^{-1}) \in \Irr_{L^2}(M')$. 

\textbf{Case I.} Suppose that $\ind_{\mc H^{M'}}^{\mc H^M} (\sigma \otimes t)$ is completely
reducible. By Proposition \ref{prop:2.7}.b all its irreducible subquotients are Langlands
constituents. Since $\Phi_M$ is an equivalence, $\Phi_M^{-1} \big( \ind_{\mc H^{M'}}^{\mc H^M} 
(\sigma \otimes t) \big)$ is also completely reducible. By Proposition \ref{prop:3.5}.b all 
its irreducible subquotients are Langlands constituents. Hence $\Phi_M^{-1}$ provides a
bijection between these two collections of Langlands constituents.

\textbf{Case II.} Suppose that $\ind_{\mc H^{M'}}^{\mc H^M} (\sigma \otimes t)$ is not
completely reducible. By Proposition \ref{prop:2.7}.b there exists an $r \in \R_{>-1}$
such that $\ind_{\mc H^{M'}}^{\mc H^M} (\sigma \otimes t \, |t|^r )$ is completely
reducible. By Proposition \ref{prop:2.7}.a, Case I and Proposition \ref{prop:3.5}.a, the 
four representations
\begin{multline*}
\ind_{\mc H^{M'}}^{\mc H^M} (\sigma \otimes t) ,\; 
\ind_{\mc H^{M'}}^{\mc H^M} (\sigma \otimes t \, |t|^r ) ,\;
\Phi_M^{-1} \big( \ind_{\mc H^{M'}}^{\mc H^M} (\sigma \otimes t \, |t|^r) \big) 
\text{ and } \Phi_M^{-1} \big( \ind_{\mc H^{M'}}^{\mc H^M} (\sigma \otimes t) \big) 
\end{multline*}
have the same number of inequivalent Langlands constituents.

Let $\pi'$ be a non-Langlands constituent of $\ind_{\mc H^{M'}}^{\mc H^M} (\sigma \otimes t)$.
By Theorem \ref{thm:2.6}.c there exists a Levi subgroup $M_1 \subset M$, with $M_1 = M$ or
d$\Phi_L (\Q R(M_1,L)) = \Q Q_1^\vee \subsetneq Q^\vee$, such that  $\pi'$ is a constituent of 
$\ind_{\mc H^{M_1}}^{\mc H^M} (\sigma_1 \otimes t_1)$ for some 
$t_1 \in T^{Q_1}$ and $\sigma_1 \in \Irr_{L^2} (\mc H_{Q_1} \rtimes \Gamma_{M_1})$ with
$\norm{cc (\sigma_1)} > \norm{cc (\sigma)}$. When $M_1 = M$, Lemma \ref{lem:2.2} shows
that the same condition on $\pi'$ is also fulfilled for the unique Levi subgroup 
$M_2 \subset M$ with d$\Phi_L (\Q R(M_2,L)) = \Q Q^\vee$. In that case we replace $M_1$ by $M_2$.

Now $\Phi_{M_1}^{-1} ( \sigma_1 \otimes t_1 |t_1|^{-1} ) \in \Irr_{L^2}(M_1)$ and 
$\pi'$ is a constituent of 
\begin{equation}\label{eq:4.16}
I^M_{(M \cap P_L) M_1} \big( \Phi_{M'}^{-1} (\sigma_1 \otimes t_1 |t_1|^{-1}) \otimes
\Phi_\nr^{-1}( |t_1| ) \big) . 
\end{equation}
Recall that the invariant $\norm{cc_{M_1}}$ on $\Irr (M_1)^{\fs_{M_1}}$ is defined via
a $W_\fs$-invariant inner product on $\mf a_L = X^* (A_L) \otimes_\Z \R$. Via Lemmas 
\ref{lem:4.5}.a and \ref{lem:4.3}.a this can be transferred (canonically) to a 
$W \Gamma$-invariant inner product on $\mf a$. The supercuspidal support (which
is involved in $cc_{M_1}$) on $\Irr (M_1)^{\fs_{M_1}}$ is (up to conjugation) given
by $J^{M_1}_{\overline{P_L} \cap M_1}$. Then Condition \ref{cond:Morita}.ii shows that
\begin{multline}
\norm{cc_{M_1} \big( \Phi_{M_1}^{-1}(\sigma_1 \otimes t_1 |t_1|^{-1}) \big)} = 
\norm{cc (\sigma_1 \otimes t_1 |t_1|^{-1})} = \norm{cc (\sigma_1)} > \\
\norm{cc (\sigma)} = \norm{cc (\sigma \otimes t \, |t|^{-1})} =
\norm{cc_{M'} \big( \Phi_{M'}^{-1} (\sigma \otimes t \, |t|^{-1}) \big)} . 
\end{multline}
Now Theorem \ref{thm:3.4}.c says that $\Phi_M^{-1} (\pi')$ is not a Langlands constituent of\\
$\Phi_M^{-1} \big( \ind_{\mc H^{M'}}^{\mc H^M} (\sigma \otimes t) \big)$. 
Summarizing, we know that:
\begin{itemize}
\item $\Phi_M^{-1}$ provides a bijection between the collections of inequivalent
irreducible subquotients of $\ind_{\mc H^{M'}}^{\mc H^M} (\sigma \otimes t)$ and of
$\Phi_M^{-1} \big( \ind_{\mc H^{M'}}^{\mc H^M} (\sigma \otimes t) \big)$;
\item these two collections have the same number of Langlands constituents and the same
number of non-Langlands constituents;
\item $\Phi_M^{-1}$ maps non-Langlands constituents of $\ind_{\mc H^{M'}}^{\mc H^M} 
(\sigma \otimes t)$ to non-Langlands constituents of $\Phi_M^{-1} \big( 
\ind_{\mc H^{M'}}^{\mc H^M} (\sigma \otimes t) \big)$.
\end{itemize}
Consequently $\Phi_M^{-1}$ also provides a bijection between the collections of 
inequivalent Langlands constituents on both sides.
\end{proof}

Now we are ready for the proof of main result of this paragraph.

\begin{thm}\label{thm:4.7}
Assume Conditions \ref{cond:Morita} and \ref{cond:Hecke}, and let $P = MU$ be a
parabolic subgroup containing $P_L$.
\enuma{
\item $\Phi_M$ restricts to an equivalence between the category of finite length tempered 
representations in $\Rep (M)^{\mf s_M}$ and the category of finite dimensional tempered 
$\mc H^M$-modules. 
\item Suppose that d$\Phi_L(\Q R(M,L)) = \Q Q^\vee$ for some $Q \subset \Delta$. Then 
$\Phi_M$ sends finite length essentially square-integrable $M$-representations to essentially 
discrete series $\mc H^M$-representations, and $\Phi_M^{-1}$ does the converse.
}
\end{thm}
\begin{proof}
(a) In view of Lemmas \ref{lem:4.6}.a and \ref{lem:4.9}, it suffices to prove that $\Phi_M$ 
preserves temperedness of finite length representations.

Suppose, on the contrary, that there exists a finite length tempered $\pi \in \Rep (M)^{\mf s_M}$
such that $\Phi_M (\pi) \in \mr{Mod}(\mc H^M)$ is not tempered. Since $\Phi_M (\pi)$ has finite
length, it has a composition series with finite dimensional irreducible quotients. It follows 
directly from the definition of temperedness for $\mc H^M = \mc H^Q \rtimes \Gamma_M$ that at least 
one of these irreducible subquotients, say $\rho_1$, is not tempered. Then we may replace 
$\Phi_M (\pi)$ by $\rho_1$ and $\pi$ by $\Phi_M^{-1}(\rho_1)$. Hence it suffices to prove the claim 
for irreducible representations.

We take the same family of parabolic subalgebras of $\mc H^M$ as in Lemma \ref{lem:4.8}.
By Theorem \ref{thm:2.6}.c there exists a Levi subgroup $M_1 \subset M$, with $M_1 = M$ or
d$\Phi_L (\Q R(M_1,L)) = \Q Q_1^\vee \subsetneq Q^\vee$, such that $\Phi_M (\pi)$ is a 
Langlands constituent of $\ind_{\mc H^{M_1}}^{\mc H^M} (\sigma_1 \otimes t_1)$ for some 
$t_1 \in T^{Q_1}$ and $\sigma_1 \in \Irr_{L^2} (\mc H_{Q_1} \rtimes \Gamma_{M_1})$. 
When $M_1 = M$, Lemma \ref{lem:2.2} shows
that the same condition on $\pi'$ is also fulfilled for the unique Levi subgroup 
$M_2 \subset M$ with d$\Phi_L (\Q R(M_2,L)) = \Q Q$. In that case we replace $M_1$ by $M_2$.

By Lemma \ref{lem:4.8} $\pi$ is a Langlands constituent of 
\begin{equation}\label{eq:4.17}
\Phi_M^{-1} \big( \ind_{\mc H^{M_1}}^{\mc H^M} (\sigma \otimes t) \big) \cong
I^M_{(M \cap P_L) M_1} \big( \Phi_{M_1}^{-1} (\sigma \otimes t \,|t|^{-1}) \otimes
\Phi_\nr^{-1}( |t| ) \big) . 
\end{equation}
Suppose that $\Phi_M (\pi)$ is not tempered, so $t \notin T_\un$ by Theorem \ref{thm:2.6}.d. 
Then $\Phi_\nr^{-1}( |t|) \in X_\nr (M') \setminus X_\unr (M')$, and by 
Theorem \ref{thm:3.4}.b $\pi$ is not tempered. 

With Lemma \ref{lem:4.6}.a we see that $\Phi_M^{-1}$ preserves both temperedness and
non-tem\-pered\-ness of irreducible representations. Hence so does $\Phi_M$.\\
(b) For $\Phi_M^{-1}$ this is Lemma \ref{lem:4.6}.b, so we only have to consider the claim 
for $\Phi_M$. Up to \eqref{eq:4.17} we can follow the proof of part (a), only replacing 
tempered by essentially discrete series everywhere. 

Suppose that $\Phi_M (\pi)$ is not essentially discrete series. By the uniqueness in Theorem
\ref{thm:2.6}.b $Q'$ is a proper subset of $Q$. Then $M'$ is a proper Levi subgroup of $M$, 
so by the uniqueness in Theorem \ref{thm:3.4}.a $\pi$ is not essentially square-integrable. 
With Lemma \ref{lem:4.6}.b we conclude that $\Phi_M^{-1}$ sends those irreducible representations 
which are essentially discrete series to essentially square-integrable representations, and 
those which are not essentially discrete series to representations that are not essentially 
square-integrable. Now it is clear that $\Phi_M$ also respects these properties.
\end{proof}

For essentially square-integrable representations we can be more precise than Theorem 
\ref{thm:4.7}. We write d$\Phi_L (\Q R(M,L)) \cap R^\vee = R_Q^\vee$.

\begin{prop}\label{prop:4.1}
Assume Conditions \ref{cond:Morita} and \ref{cond:Hecke}, and let $P = MU$ be a
parabolic subgroup containing $P_L$.
\enuma{
\item Suppose that $\Q R(M,L)) \cap \textup{d}\Phi_L^{-1}(R^\vee)$ does not span $\Q R(M,L)$.
Then $\Rep (M)^{\mf s_M}$ contains no finite length essentially square-integrable 
representations. 
\item Suppose that $\pi \in \Rep (M)^{\fs_M}$ is square-integrable modulo centre 
and has finite length. Then $\Phi_M (\pi)$ is tempered, essentially discrete series 
and factors through $\psi_t : \mc H^Q \rtimes \Gamma_M \to \mc H_Q \rtimes \Gamma_M$ 
for some $t \in T_\un^Q$. 
\item Suppose that d$\Phi_L (\Q R(M,L)) = \Q Q^\vee$. Then $\Phi_M$ gives a bijection 
between $\Irr_{L^2}(M)^{\mf s_M}$ and $\Irr_{L^2}(\mc H^M)$.
}
\end{prop}
\begin{proof}
(a) Suppose, contrary to what we need to show, that $\Rep (M)^{\mf s_M}$ does contain
a representation of the indicated kind. Since it has finite length, it has an irreducible
subrepresentation, say $\pi$. Let $\zeta$ be the central character of $\pi$, and let
$|\zeta| \in X_\nr (M)$ be its absolute value. Then $\pi \otimes |\zeta|^{-1} \in 
\Rep (M)^{\fs_M}$ is an irreducible essentially square-integrable representation with a unitary 
central character. Hence it is square-integrable modulo centre and in particular tempered. By
Theorem \ref{thm:4.7}.a $\Phi_M (\pi \otimes |\zeta|^{-1})$ is also tempered. 

Let $P_1 = M_1 U_1 \subset G$ be the parabolic subgroup such that $M \supset M_1 \supset L$ and 
d$\Phi_L (\Q R(M_1,L))$ is spanned by $\Q R(M,L) \cap \textup{d}\Phi_L^{-1}(R^\vee)$.
From \eqref{eq:3.14} we know that $\Res^{\mc H^M}_{\lambda_{M M_1} (\mc H^{M_1})} \circ \Phi_M 
(\pi \otimes |\zeta|^{-1})$ is tempered and nonzero. By Lemma \ref{lem:4.6}.a 
\[
\Phi_{M_1}^{-1} \circ \Res^{\mc H^M}_{\lambda_{M M_1} (\mc H^{M_1})} \circ \Phi_M 
(\pi \otimes |\zeta|^{-1}) = J^M_{P_1 \cap M} (\pi \otimes |\zeta|^{-1}) 
\]
is also tempered and nonzero. But \cite[Lemme III.3.2]{Wal} says that this contradicts
the square-integrability (modulo centre) of $\pi \otimes |\zeta|^{-1}$.
(b) This follows from Theorem \ref{thm:4.7}, in the same way as 
Lemma \ref{lem:4.6}.c followed from parts (a) and (b) of that lemma.\\
(c) This follows from Theorem \ref{thm:4.7}.b, Lemma \ref{lem:4.6}.c and part (b).\\
\end{proof}

\subsection{Comparison of completions} \

In this paragraph we will show that the equivalences $\Phi_M$ induce Morita equivalences
between the appropriate Schwartz algebras.
In Proposition \ref{prop:2.3} we described the Plancherel isomorphism for the Schwartz
completion of an affine Hecke algebra, in terms of the following data:
\begin{itemize}
\item the set of parabolic subalgebras $\mc H^Q \rtimes \Gamma_Q$ of $\mc H \rtimes \Gamma$,
up to $\Gamma W$-equivalence,
\item the tori $T^Q_\un$,
\item the sets $\Irr_{L^2}(\mc H^Q \rtimes \Gamma_Q)$, up to the actions of $T^Q_\un$ and
$W\Gamma (Q,Q)$,
\item the groupoid $\mc G$,
\item the intertwining operators $I(g,Q,\sigma,t)$ for $g \in \mc G_{Q,\sigma}$.
\end{itemize}
These data depend mainly on the categories $\Mod (\mc H^Q \rtimes \Gamma_Q)$.
In Condition \ref{cond:Hecke} we included the possibility that not the $\mc H^M$, but
the $(\mc H^M )^\op$ are affine Hecke algebras, so that $\Phi_M$ becomes an equivalence
between $\Rep (G)^{\mf s}$ and $\Mod \big( (\mc H^Q \rtimes \Gamma_Q )^\op \big)$.
Then we use Lemma \ref{lem:2.4} to describe the Plancherel isomorphism of $\mc S (\mc R,q) 
\rtimes \Gamma$ in terms of right modules of its subalgebras $\mc H^Q \rtimes \Gamma_Q$,
that is in terms of the categories $\Mod (\mc H^M)$. With this in mind, it
suffices to consider the case where each $\mc H^M$ is an (extended) affine Hecke algebra.

On the other hand, in Theorem \ref{thm:3.2} the Plancherel isomorphism for $\mc S (G)^{\mf s}$
was formulated in terms of:
\begin{itemize}
\item the set of parabolic subgroups $P \supset P_L$, up to conjugation by $W_{\mf s}$,
\item the tori $X_\unr (M)$,
\item the sets $\Irr_{L^2}(M)^{\mf s_M}$, 
up to the actions of $X_\unr (M)$ and $\mr{Stab}_{W_{\mf s}}(M)$,
\item the groups $W' (T_\omega)$,
\item the intertwining operators $I(w',\omega \otimes \chi)$ for $w' \in W' (T_\omega)$. 
\end{itemize}
We will compare these two data sets, and manipulate them until we get a nice bijection from
one side to the other. 

By Proposition \ref{prop:4.1}.a only the $P$ with d$\Phi_L (\Q R(M,L))$ of the form 
$\Q Q^\vee$ occur in the Plancherel isomorphism, since for the other $P$ the set 
$\Irr_{L^2}(M)^{\mf s_M}$ is empty. Given $Q \subset \Delta$, we define $\Gamma_Q$ as 
$\Gamma_M$, where $\Q R(M,L) = \Q Q^\vee$.

By Condition \ref{cond:Hecke}.iv there is a canonical bijection from the parabolic subgroups 
$P = M U_P$, with $P\supset P_L, M \supset L$ and d$\Phi_L (\Q R(M,L))$ of the form 
$\Q Q^\vee$ and modulo conjugation by elements of $W_\fs$, to the parabolic subalgebras 
$\mc H^Q \rtimes \Gamma_Q$ of $\mc H^G$, up to association by $W \Gamma$. From Theorem 
\ref{thm:3.2} one sees that two such Levi subgroups $M \subset G$ are $W_\fs$-conjugate if 
and only if the tempered parts of the two subsets $I_P^G (\Rep (M )^{\mf s_M} )$ coincide. 
By Condition \ref{cond:Morita}.i and Theorem \ref{thm:4.7}.a this means precisely that two 
subsets $\ind_{\lambda_{M G}(\mc H^M)}^{\mc H^G} (\Mod (\mc H^M))$ of $\Mod (\mc H^G)$ 
coincide. By Theorem \ref{thm:2.1} that happens if and only if the two $\mc H^M$ are 
$\Gamma W$-equivalent. Thus we can pick of representatives for such $P$ modulo $W_\fs$-
conjugacy, and then the corresponding $\mc H^M$ form representatives for $\Gamma W$-equivalence 
classes of parabolic subalgebras $\mc H^M = \mc H^Q \rtimes \Gamma_Q$ of $\mc H^G$.

By Proposition \ref{prop:4.1},b $\Phi_M$ gives a bijection between $\Irr_{L^2}(M)^{\mf s_M}$ 
and $\Irr_{L^2}(\mc H^M)$. Upon parabolic induction, every $X_\unr (M)$-orbit in 
$\Irr_{L^2}(M)^{\mf s_M}$ (resp. every $T^Q_\un$-orbit in $\Irr_{L^2}(\mc H^M)$ gives rise to a 
family of tempered representations in $\Rep (G )^{\mf s}$ (resp. in $\Mod(\mc H^G)$). From 
Theorem \ref{thm:3.2} we see that $I_P^G (\omega)$ and $I_P^G (\omega')$ belong to the same 
such family if and only if $\omega' = w (\omega \otimes \chi)$ for some 
$w \in \mr{Stab}_{W_{\mf s}}(M)$ and $\chi \in X_\unr (M)$. Similarly, by \ref{cond:Hecke}.ii 
and Proposition \ref{prop:2.3} 
\[
\ind_{\lambda_{M G}(\mc H^M)}^{\mc H^G}(\sigma) \quad \text{and} \quad 
\ind_{\lambda_{M G}(\mc H^M)}^{\mc H^G}(\sigma')
\]
belong to the same family in $\Mod (\mc H^G)$ if and only if $\sigma' = g (\sigma \circ \phi_t)$
for some $g \in \mc G_{Q Q}$ and $t \in T^Q_\un$. 
Applying $\Phi_G$ and Condition \ref{cond:Morita}.i, we see that the respective equivalence
relations on $\Irr_{L^2}(M)^{\mf s_M}$ and $\Irr_{L^2}(\mc H^M)$ agree via $\Phi_M$.

Let the set of representatives $(Q,\sigma) / \sim$ be as in \eqref{eq:2.4}
Let $(P,M,\omega) / \sim$ be its image under Lemma \ref{lem:4.4} and the $\Phi_M^{-1}$.
Then $(P,M,\omega) / \sim$ is a set of representatives as in Theorem \ref{thm:3.2}.
Lemma \ref{lem:4.5}.c and Condition \ref{cond:Morita}.i guarantee that
\begin{equation}\label{eq:4.4}
\Phi_G (I_P^G (\omega \otimes \chi)) = \ind_{\lambda_{M G}(\mc H^M)}^{\mc H^G} (\sigma \otimes
\Phi_\nr (\chi)) = \pi (Q,\sigma,\Phi_\nr (\chi)) .
\end{equation}
Hence $\Phi_G$ matches the finite length tempered elements of $\Rep (G)$ associated to 
$(P,M,\omega)$ (via Theorem \ref{thm:3.2}) with the finite dimensional tempered 
$\mc H^G$-modules associated to $(Q,\sigma)$ (via Proposition \ref{prop:2.3}).
By Theorem \ref{thm:3.2} $I_P^G (\omega \otimes \chi)$ and $I_P^G (\omega \otimes \chi')$
are isomorphic if and only $\chi' = w' \chi$ for some $w' \in W' (T_\omega)$. Analogously,
Proposition \ref{prop:2.3} entails that $\pi (Q,\sigma,t)$ and $\pi (Q,\sigma,t')$ are 
isomorphic if and only if $t' = g (t)$ for some $g \in \mc G_{Q,\sigma}$. From this and 
\eqref{eq:4.2} we see that $\Phi_\nr$ (from Lemma \ref{lem:4.5}.a) induces a bijection
\begin{equation}\label{eq:4.3}
X_\unr (M) / W' (T_\omega) \to T^Q_\un / \mc G_{Q,\sigma} \cong X_\unr (M) / \mc G'_{Q,\sigma} .
\end{equation}
In the proof of Lemma \ref{lem:4.3} we checked that $W_\fs$ (resp. $W \Gamma$) acts faithfully
on $X_\nr (L)$ (resp. on $T$). Then we see from \eqref{eq:3.11} and \eqref{eq:GPQ} that the
group actions in \eqref{eq:4.3} are faithful. Comparing the outer sides of \eqref{eq:4.3} and 
using the same method as in the proof of Lemma \ref{lem:4.3}, we deduce that 
$W' (T_\omega) = \mc G'_{Q,\sigma}$ as subgroups of Aut$(X_\unr (M))$.

Now we come to the intertwining operators. Recall from \eqref{eq:3.12} and \eqref{eq:3.13} that
$I(w',\omega \otimes \chi)$ for $w' \in W' (T_\omega)$ comes from a unitary operator
\begin{equation}\label{eq:4.5}
\pi (w',\omega, \chi) : I_P^G (\omega \otimes \chi) \to I_P^G (\omega \otimes w' (\chi)) .
\end{equation}
For bookkeeping purposes we replace $T^Q$ by $X_\nr (M)$ and $\mc G_{Q,\sigma}$ by
$\mc G'_{Q,\sigma}$, at the same time defining
\[
\pi (Q,\sigma,\chi) := \pi (Q,\sigma,\Phi_\nr (\chi)) \quad \text{and} \quad
\pi (g',Q,\sigma,\chi) = \pi (g,Q,\sigma,\Phi_\nr (\chi))
\]
when $g' \in \mc G'_{Q,\sigma}$ is a lift of $g \in \mc G_{Q,\sigma}$. In particular, for
$k \in X_\nr (M,\sigma)$ the interwiner $\pi (k,Q,\sigma,\chi)$ is the identity as map
on the underlying vector spaces, it only changes $\chi$ to $k \chi$. Then \eqref{eq:2.1} 
says that the action of $\mc G'_{Q,\sigma}$ in Proposition \ref{prop:2.3} and \eqref{eq:2.4}
comes from unitary intertwiners
\begin{equation}\label{eq:4.6}
\pi (g',Q,\sigma,\chi) \in \Hom_{\mc H^G}(\pi (Q,\sigma,\chi), \pi (Q,\sigma, g'(\chi)) .
\end{equation}

\begin{lem}\label{lem:4.10}
The intertwining operators \eqref{eq:4.5} and \eqref{eq:4.6} can be normalized so that
\[
\Phi_G ( \pi (w',\omega, \chi) ) = \pi (g',Q,\sigma,\Phi_\nr (\chi))
\]
whenever $w'$ corresponds to $g'$ under the identification 
$W' (T_\omega) = \mc G'_{Q,\sigma}$ from \eqref{eq:4.3}.
\end{lem}
\begin{proof}
Both \eqref{eq:4.5} and \eqref{eq:4.6} are unique up to scalars, because they depend
algebraically on $\chi$ and because for generic $\chi \in X_\unr (M)$ the involved 
representations are irreducible. (The latter follows for example from the Plancherel
isomorphisms.) Therefore, if $w' = g'$ in the indicated way, 
$\Phi_G ( \pi (w',\omega, \chi) )$ equals $\pi (g',Q,\sigma,\Phi_\nr (\chi))$
up to a complex number of absolute value 1. To make this scalar 1, we simply replace 
$\pi (w',\omega, \chi)$ by $\Phi_G^{-1} \big( \pi (g',Q,\sigma,\Phi_\nr (\chi)) \big)$. 
\end{proof}

We remark that the normalization from Lemma \ref{lem:4.10} is harmless, because it 
does not change $I(w',\omega \otimes \chi)$. 

\begin{thm}\label{thm:4.2}
Under the Conditions \ref{cond:Morita}, \ref{cond:Hecke} and \ref{cond:Gamma},
$\Phi_G : \Rep (G)^{\mf s} \to \Mod (\mc H^G)$ induces Morita equivalences
\[
\mc S (G)^{\mf s} \sim_M \mc S (\mc R,q) \rtimes \Gamma \quad \text{and} \quad
C_r^* (G)^{\mf s} \sim_M C_r^* (\mc R,q) \rtimes \Gamma .
\]
\end{thm}
\begin{proof}
In view of Proposition \ref{prop:2.3} and Theorem \ref{thm:3.2}
we have to compare the Schwartz algebras
\begin{equation}\label{eq:4.12}
\begin{aligned}
& \bigoplus\nolimits_{(P,M,\omega) / \sim} \big( C^\infty (X_\unr (M)) \otimes 
\End_\C (I_P^G (V_\omega )^{K_{\mf s}} \big)^{W' (T_\omega)} \quad \text{and} \\
& \bigoplus_{(Q,\sigma) / \mc G} C^\infty \big( T^Q_\un ; \End_\C (V_{Q,\sigma}) 
\big)^{\mc G_{Q,\sigma}} = \bigoplus_{(Q,\sigma) / \mc G} 
C^\infty \big( X_\unr (M) ; \End_\C (V_{Q,\sigma}) \big)^{\mc G'_{Q,\sigma}}.
\end{aligned}
\end{equation}
To justify the equality in the second line, we note that a section of the algebra bundle 
over $X_\unr (M)$ is $X_\nr (M,\sigma)$-invariant if and only if it descends to a section
of the analogous algebra bundle over $T^Q_\un$.

By the above constructions the $\Phi_M$ provide a bijection between the indexing sets
for the sums in \eqref{eq:4.12}, so it suffices to compare
\begin{equation}\label{eq:4.8}
\begin{aligned}
& A_1 := C^\infty \big( X_\unr (M) ; \End_\C (I_P^G (V_\omega )^{K_{\mf s}}) 
\big)^{W' (T_\omega)} \quad \text{with} \quad \\
& A_2 := C^\infty \big( X_\unr (M) ; \End_\C (V_{Q,\sigma}) \big)^{\mc G'_{Q,\sigma}}
\end{aligned}
\end{equation}
when $(P,M)$ corresponds to $Q$ via Lemma \ref{lem:4.4} and $\Phi_M (\omega)
= \sigma$. The Morita equivalences $\mc S (G,K_{\mf s} )^{\mf s} \sim_M \mc S (G)^{\mf s}$
and $\Phi_G$ send $I_P^G (\omega \otimes \chi)^{K_{\mf s}}$ to $\pi (Q,\sigma,\chi)$
and by Lemma \ref{lem:4.10} this is compatible with the intertwining operators. Identifying
$W' (T_\omega)$ and $\mc G'_{Q,\sigma}$ via \eqref{eq:4.3}, we consider the following 
bimodules for $A_1$ and $A_2$:
\begin{align*}
& B_1 := C^\infty \big( X_\unr (M) ; \Hom_\C (I_P^G (V_\omega )^{K_{\mf s}}, V_{Q,\sigma}) 
\big)^{W' (T_\omega)} , \\
& B_2 := C^\infty \big( X_\unr (M) ; \Hom_\C (V_{Q,\sigma}, I_P^G (V_\omega )^{K_{\mf s}}) 
\big)^{W' (T_\omega)} . 
\end{align*}
Here the $W' (T_\omega)$-actions are
\[
\begin{array}{lll}
(w' \cdot f_1)(w' \chi) & = & 
\pi (w', \omega \otimes \chi) f_1 (\chi) \pi (w',Q,\sigma,\chi)^{-1} \qquad f_1 \in B_1 , \\
(w' \cdot f_2)(w' \chi) & = & 
\pi (w', Q,\sigma,\chi) f_2 (\chi) \pi (w',\omega \otimes \chi)^{-1} \qquad f_2 \in B_2 . 
\end{array}
\]
Notice that by Lemma \ref{lem:4.10} these are honest group actions, not just up to
some scalars. We claim that 
\begin{equation}\label{eq:4.9}
B_1 \otimes_{A_1} B_2 \cong A_2 \quad \text{and} \quad B_2 \otimes_{A_2} B_1 \cong A_1
\end{equation}
as bimodules over $A_2$, respectively $A_1$. Since all these algebras and modules are of
finite rank over $C^\infty (X_\unr (M))^{W' (T_\omega)}$, it suffices to check this locally,
at any $\chi \in X_\unr (M)$. Then the proof of the first half of \eqref{eq:4.9} reduces 
to checking that
\begin{align}
\nonumber \Hom_\C (I_P^G (V_\omega )^{K_{\mf s}}, V_{Q,\sigma})^{W' (T_\omega)_\chi} 
& \otimes_{\End_\C (I_P^G (V_\omega)^{K_{\mf s}} )^{W' (T_\omega)_\chi}}
\Hom_\C (V_{Q,\sigma}, I_P^G (V_\omega )^{K_{\mf s}})^{W' (T_\omega )_\chi}  \\
\label{eq:4.10} & \cong \End_\C (V_{Q,\sigma} )^{W' (T_\omega )_\chi} ,
\end{align}
and the other way round for $B_2 \otimes_{A_2} B_1 \cong A_1$.

By the uniqueness of $\pi (w',Q,\sigma,\chi)$ up to scalars, $w' \mapsto \pi (w',Q,\sigma,\chi)$
defines a projective representation of $W' (T_\omega)_\chi$. Let $W''$ be a finite central
extension of $W' (T_\omega)_\chi$, such that this lifts to a linear representation of $W''$.
By \eqref{eq:4.9} the map $w' \mapsto \pi (w',\omega \otimes \chi)$ also lifts to a linear
representation of $W''$. Then $W''$ and $W' (T_\omega)$ have the same invariants in the all
involved modules, so we can rewrite \eqref{eq:4.10} as
\begin{align}\nonumber
\Hom_{\C [W'']} (I_P^G (V_\omega )^{K_{\mf s}}, V_{Q,\sigma}) 
& \otimes_{\End_{\C [W'']}(I_P^G (V_\omega)^{K_{\mf s}} )}
\Hom_{\C [W'']} (V_{Q,\sigma}, I_P^G (V_\omega )^{K_{\mf s}})  \\
\label{eq:4.11} & \cong \End_{\C [W'']} (V_{Q,\sigma} ) .
\end{align}
This is a statement about finite dimensional representations of the finite group $W''$. One
can verfiy \eqref{eq:4.11} by reducing it to the case of irreducible $W''$-representations,
where it is obvious. 

This also proves \eqref{eq:4.10} and \eqref{eq:4.9}, and shows that the algebras in 
\eqref{eq:4.12} are Morita equivalent. Combining that with Theorem \ref{thm:3.2} and 
\eqref{eq:2.4}, we find the desired Morita equivalences of Schwartz algebras. 

To prove that $C_r (G)^{\mf s}$ and $C_r (\mc R,q) \rtimes \Gamma$ are Morita equivalent,
we can use exactly the same argument. We only have to replace $C^\infty$ by continuous
functions everywhere, and to use Theorem \ref{thm:3.3} instead of Theorem \ref{thm:3.2}.
\end{proof}
\vspace{3mm}

\section{Hecke algebras from Bushnell--Kutzko types}
\label{sec:types}

Let $L \subset G$ be a Levi subgroup and let $\sigma \in \Irr (L)$ be supercuspidal.
Recall from \cite[\S 4]{BuKu} that a type for $\fs = [L,\sigma]_G$ consists of a compact 
open subgroup $J \subset G$, and a $\lambda \in \Irr (J)$, such that $\Rep (G)^\fs$ 
is precisely the category of smooth $G$-representations which are generated by their
$\lambda$-isotypical subspace. To such a type one associates the algebra
\[
\mc H (G,J,\lambda) = \End_G (\ind_J^G \lambda) , 
\]
which (by definition) acts from the right on $\ind_J^G \lambda$. Then there is a
Morita equivalence
\begin{equation}\label{eq:5.1}
\begin{array}{cccc}
\Phi_G : & \Rep (G)^\fs & \to & \Mod (\mc H (G,J,\lambda)) \\
 & \pi & \mapsto & \Hom_J (\lambda, \pi ) \cong \Hom_G (\ind_J^G \lambda, \pi) .
\end{array} 
\end{equation}
For a Levi subgroup $M \subset G$ containing $L$, Bushnell and Kutzko \cite[\S 8]{BuKu} 
deve\-loped the notion that $(J,\lambda)$ covers a $[L,\sigma]_M$-type $(J_M,\lambda_M)$. 
Roughly speaking, this means that $J_M = J \cap M$, that $\lambda_M = \Res^J_{J_M} \lambda$ 
and that $\mc H (G,J,\lambda)$ contains invertible "strongly positive" elements.
Under these conditions, writing $\fs_M = [L,\sigma]_M$, there is a Morita equivalence
$\Phi_M : \Rep (M)^{\fs_M} \to \Mod (\mc H (M,J_M,\lambda_M))$ as in \eqref{eq:5.1},
which is in several ways compatible with $\Phi_G$.

\begin{lem}\label{lem:T}
Suppose that $(J,\lambda)$ is a cover of a $[L,\sigma]_L$-type $(J_L,\lambda_L)$. Then
Condition \ref{cond:Morita} is fulfilled, with $\mc H^M = \mc H (M,J_M,\lambda_M))$.
\end{lem}
\begin{proof}
Let $P$ and $P'$ be as in the condition. By \cite[Proposition 8.5]{BuKu} 
$(J_{M'},\lambda_{M'})$ is a $\fs_{M'}$-type, $(J_M,\lambda_M)$ is a $\fs_M$-type
and the former covers the latter.

By \cite[Corollary 8.4]{BuKu} there exists a unique algebra monomorphism
\[
t_{\overline P \cap M} : \mc H (M,J_M,\lambda_M) \to \mc H (M',J_{M'},\lambda_{M'}) 
\]
such that 
\[
\Res^{\mc H (M',J_{M'},\lambda_{M'})}_{t_{\overline P \cap M} (\mc H (M,J_{M},\lambda_{M}))}
\circ \Phi_{M'} = \Phi_M \circ \mr{pr}_{\fs_M} \circ R^{M'}_{\overline P \cap M'} .
\]
Here $R^{M'}_{\overline P \cap M'}$ means the unnormalized parabolic restriction functor.
To obtain the version with the normalized Jacquet functor $J^{M'}_{\overline P \cap M'}$,
we must adjust $t_{\overline P \cap M}$ by the square root of a modular character.
This yields our $\lambda_{M M'}$. The uniqueness of $\lambda_{M M'}$ and the transitivity of 
normalized Jacquet restriction entail that 
\[
\lambda_{M' M''} \circ \lambda_{M M'} = \lambda_{M M''} \quad \text{when} \quad
P \subset P' \subset P'' \subset G.
\]
On general grounds $\ind_{\lambda_{M M'} (\mc H (M,J_{M},\lambda_{M}))}^{
\mc H (M',J_{M'},\lambda_{M'})}$ is the left adjoint of 
$\Res^{\mc H (M',J_{M'},\lambda_{M'})}_{\lambda_{M M'} (\mc H (M,J_{M},\lambda_{M}))}$.
By Bernstein's second adjointness theorem $I_{P \cap M'}^{M'} : \Rep (M) \to \Rep (M')$
is the left adjoint of $J^{M'}_{\overline{P} \cap M'}$. Hence
\[
I_{P \cap M'}^{M'} : \Rep (M)^{\fs_M} \to \Rep (M')^{\fs_{M'}}
\]
is the left adjoint of $\mr{pr}_{\fs_M} \circ J^{M'}_{\overline P \cap M'}$. By the
uniqueness of adjoints
\[
\Phi_{M'} \circ I_{P \cap M'}^{M'} = \ind_{\lambda_{M M'} (\mc H (M,J_{M},\lambda_{M}))}^{
\mc H (M',J_{M'},\lambda_{M'})} \circ \Phi_M . \qedhere
\]
\end{proof}

Having checked Condition \ref{cond:Morita} in a general framework, we turn to
more specific instances where Condition \ref{cond:Hecke} holds. In most cases the
intermediate algebras $\mc H^M$ are not mentioned explicitly in the literature.
One can obtain them by applying the same references to the group $M$ instead of $G$.
Using the canonical construction of $\lambda_{M M'}$ as in the proof of Lemma \ref{lem:T},
Condition \ref{cond:Hecke}.ii will be satisfied automatically in those cases.
We will check the remaining conditions, mainly by providing relevant references. 
Recall that to achieve Condition \ref{cond:Hecke}.iii we can use the method described 
on page \pageref{cond:Hecke}.\\

\textbf{Iwahori--spherical representations.} \\
This is the classical case. 
Let $\mf o_F$ be the ring of integers of the non-archimedean local field $F$, let
$\mf p_F$ be its maximal ideal, and let $k_F = \mf o_F / \mf p_F$ be the residue field.
Choose an apartment $\mh A$ of the Bruhat--Tits building of $G$ and let $L$ be the
correponding minimal $F$-Levi subgroup of $G$. Let $I$ be an Iwahori subgroup of $G$
associated to a chamber of $\mh A$.
Let $P_L$ be the parabolic subgroup of $G$ with Levi factor $L$, such
that the reduction of $I$ modulo $\mf p_F$ is $P_L (k_F)$.

Borel \cite{Bor} showed that the trivial representation of $I$ is a $\mf s$-type, where
$\mf s = [L,\mathrm{triv}_L]_G$. Borel assumes that $G$ is semisimple,
but it is easy to generalize his arguments to reductive $G$. 

By \cite[\S 3]{IwMa} there is a *-algebra isomorphism 
\begin{equation}\label{eq:2.11}
C_c (I \backslash G / I) \cong \mc H (G, I , \mr{triv}) 
\cong \mc H \big( X_* (L), R^\vee (G,L), X^* (L), R(G,L), \Delta, q_I \big) ,
\end{equation}
where the basis $\Delta$ is determined by $P_L$ and 
$q_{I,\alpha} = \mathrm{vol}(I s_\alpha I) / \mathrm{vol}(I)$ for a simple reflection 
$s_\alpha$. From \cite[\S 3.1]{Bor} one sees that Conditions \ref{cond:Hecke}.iii and 
iv hold. Here $\Gamma_M = 1$ for all $M$, so Condition \ref{cond:Gamma} is vacuous.

Of course Theorem \ref{thm:4.7} was already known for irreducible Iwahori-spherical
representations. Indeed, by \cite[Section 8]{KaLu} and \cite[Theorem 15.1.(2) and 
Proposition 16.6]{ABPS1} the bijection $\Irr (G)^\fs \to \Irr (\mc H (G,I,\mr{triv}))$ 
preserves temperedness and essential square-integrability. Moreover Theorem \ref{thm:4.2}
has been proven for Schwartz algebras in \cite[Theorem 10.2]{DeOp1}: \eqref{eq:2.11}
extends to an isomorphism of Fr\'echet *-algebras
\[
\mc S (I \backslash G / I) \cong 
\mc S \big( X_* (L), R^\vee (G,L), X^* (L), R(G,L), \Delta, q_I \big) .\\[2mm]
\]

\textbf{Principal series representations of split groups.} \\
Suppose that $G$ is $F$-split and let $T$ be a maximal split torus of $G$. Fix a smooth 
character $\chi_{\mf s} \in \Irr (T)$ and put $\mf s = [T,\chi_{\mf s}]_G$, so that 
\[
X_\nr (T) \to T_{\mf s} : \chi \mapsto \chi \chi_{\mf s} 
\]
is a homeomorphism. Notice that $\chi$ restricted to the unique maximal compact 
subgroup $T_{\mr{cpt}}$ of $T$ is a type for $[T,\chi_{\mf s}]_T$.
By \cite[Lemma 6.2]{Roc1} there exist a root subsystem
$R_{\mf s} \subset R^\vee (G,T)$ and a subgroup $\mf R_\fs \subset W_\fs$ such that
$W_\fs = W (R_\fs) \rtimes \mf R_\fs$.

\begin{thm}\label{thm:Roche} \textup{\cite[Theorem 6.3]{Roc1}} \\
There exists a type $(J,\lambda)$ for $\fs$ and a *-algebra isomorphism
\[
\mc H (G,J,\lambda) \cong \mc H (T_\fs,R_\fs,q) \rtimes \mf R_\fs , 
\]
where $q_\alpha = |k_F|$ for all $\alpha \in R_\fs$. Moreover $(J,\lambda)$ is 
a cover of $(T_{\mr{cpt}},\chi)$.
\end{thm}

Furthermore Conditions \ref{cond:Hecke}.iii and iv hold by construction. 
If $\Q R(M,T) = \Q Q^\vee$, then $X \rtimes W_Q \Gamma_Q \subset X_* (T) \rtimes W(M,T)$,
so by Remark \ref{rem:Gamma} Condition \ref{cond:Gamma} holds as well. 

We note that for these Bernstein components Theorem \ref{thm:4.7}.b was already
proven in \cite[Theorem 10.7]{Roc1}, while Theorem \ref{thm:4.7}.a follows from
\cite[Theorem 10.1]{DeOp1}, using \cite[Section 8]{Roc1}.\\

\textbf{Level zero representations.} \\
These are $G$-representations which are generated by non-zero vectors fixed by the
pro-unipotent radical of a parahoric subgroup of $G$. Iwahori-spherical representations
constitute the most basic example of this kind. A type $(J,\lambda)$ for any 
Bernstein component $\mf s$ consisting of level zero representations was exhibited in 
\cite{Mor}, while it was proven in \cite[Theorem 4.9]{Mor2} that it actually is a type. 
More precisely, by \cite[\S 3.8]{Mor2} $(J,\lambda)$ is a cover of a type for the 
underlying supercuspidal Bernstein component of a Levi subgroup $L$ of $G$.

By \cite[Theorem 7.12]{Mor} (see also \cite{Lus3})
\begin{equation}\label{eq:5.2}
\mc H (G,J,\lambda) \cong \mc H (\mc R ,q) \rtimes \C [\Gamma,\natural_\fs]
\end{equation}
for suitable $\mc R, q$ and $\Gamma$. In all examples of level zero Bernstein blocks
which have been worked out, the 2-cocycle $\natural_\fs$ of $\Gamma$ is trivial. 
But even if it were non-trivial, we could easily deal with it. There always exists a
finite central extension 
\[
1 \to \Gamma_1 \to \Gamma_2 \xrightarrow{\phi_\Gamma} \Gamma \to 1 
\]
such that the pullback of $\natural_\fs$ to $\Gamma_2$ splits. Then $\mc H (G,J,\lambda)$ 
can be regarded as the direct summand of $\mc H (\mc R,q) \rtimes \Gamma_2$ associated to
a minimal central idempotent $p_{\natural_\fs} \in \C [\Gamma_1]$. The algebra $\mc H (\mc R,q) 
\rtimes \Gamma_2$, with the parabolic subalgebras $\mc H^Q \rtimes \phi_\Gamma^{-1}(\Gamma_Q)$, 
is of the kind studied in Section \ref{sec:AHA}. In this situation the Conditions 
\ref{cond:Morita} and \ref{cond:Hecke} must be adjusted slightly, now each $\mc H^M$ should 
be $p_{\natural_\fs} \mc H^Q \rtimes \phi_\Gamma^{-1}(\Gamma_M)$ for some $Q \subset \Delta$.
With these minor modifications, all the arguments in Section \ref{sec:comparison} remain valid.

Conditions \ref{cond:Hecke}.iii and iv follow from the setup in \cite[\S 3.12--3.14]{Mor} 
and \cite[\S 1.10]{Mor2}, combined with the description of $\mc R$ in 
\cite[Proposition 7.3]{Mor}. The groups $\Gamma_Q$ for $\Q Q^\vee = \Q R(M,L)$ satisfy 
Condition \ref{cond:Gamma} because they are contained in $X \rtimes W(M,S)$, where 
$W(M,S)$ is the Weyl group of $M$ with respect to a maximal $F$-split torus $S \subset L$.

As in the above examples, there is previous work on temperedness also. It is claimed in
\cite[Theorem 10.1]{DeOp1} that Theorem \ref{thm:4.7}.a holds here. For this one needs 
to know that \eqref{eq:5.2} preserves the traces (maybe up to a positive factor) 
and the natural *-operations. The former follows from the support of the basis elements $T_w$ of 
$\mc H (G,J,\lambda)$ constructed in \cite{Mor} (only the unit element $T_e$ is supported on 
$J$). For a simple (affine) reflection $s$, both $T_s$ and $T_s^*$ have support $J s J$,
so they differ only by a scalar factor. They also satisfy the same quadratic relation,
so $T_s^* = T_s$. This implies that \eqref{eq:5.2} is an isomorphism of *-algebras.

Knowing that \cite[Theorem 10.1]{DeOp1} applies, and together with \cite[Lemma 16.5]{ABPS1} 
it also gives Theorem \ref{thm:4.7}.b. \\

\textbf{Inner forms of $\GL_n (F)$.} \\
Let $D$ be a division algebra with centre $F$. Every Levi subgroup of $G = \GL_m (D)$
is of the form $L = \prod_i \GL_{m_i}(D)^{e_i}$, where $\sum_i m_i e_i = m$. 
Fix a supercuspidal $\omega \in \Irr (L)$, of the form $\omega = \bigotimes_{i=1}^k 
\omega_i^{\otimes e_i}$, where $\omega_i \in \Irr (\GL_{m_i}(D))$ is supercuspidal and 
not inertially equivalent with $\omega_j$ if $i \neq j$. Then $T_\fs \cong \prod_{i=1}^k 
(\C^\times)^{e_i}$, $R_\fs$ is of type $\prod_{i=1}^k A_{e_i - 1}$ and
the stabilizer of $\mf s = [\omega,L]_G$ in $W(G,L)$ is
$W(R_\fs) \cong \prod\nolimits_{i=1}^k S_{e_i}$.
 
\begin{thm}\label{thm:2.2} \textup{\cite{Sec,SeSt}} \\
There exists a type $(J,\lambda)$ for $\fs$, which is a cover of a $[\omega,L]_L$-type.
There exists a parameter function $q_\fs : R_\fs \to q^\N$ such that there is an 
isomorphism of *-algebras
\[
\mc H (G,J,\lambda) \cong \mc H (X^* (T_\fs),R_\fs, X_* (T_\fs), R_\fs^\vee,q_\fs) ,
\]
where the right hand side is a tensor product of affine Hecke algebras of
type $\GL_e$ with $e \leq m$. Moreover this isomorphism sends the natural trace of
$\mc H (G,J,\lambda)$ to a positive multiple of the trace of the right hand side.
\end{thm}
We remark that the claims about the * and the traces are not made explicit in
\cite{Sec,SeSt}. They can be deduced in the same way as for level zero representations,
see above. With \cite[Theorem 10.1]{DeOp1} that proves Proposition \ref{prop:4.1}
for these groups.

Via the tensor product factorization Condition \ref{cond:Hecke}.iii reduces to the
case of a supercuspidal representation $\sigma^{\otimes e}$ of $\GL_r (D)^e$. There
it is a consequence of the constructions involved in \cite[Th\'eor\`eme 4.6]{Sec},
which entail that the same notion of positivity in real tori is used for $(\GL_r (D)^e,
\GL_1 (D)^{re})$ and for $\mc H (\GL_e ,q)$. Condition \ref{cond:Hecke}.iv is irrelevant
because all the groups $\Gamma_M$ are trivial.

For the Schwartz algebras of these groups Theorem \ref{thm:4.2} can be
found in \cite[Theorem 6.2]{ABPS6}. The proof over there is similar but simpler, because
not all complications from Section \ref{sec:comparison} arise.\\

\textbf{Inner forms of $\SL_n (F)$.} \\
Let $G$ be the kernel of the reduced norm map $\GL_m (D) \to F^\times$. It is an
inner form of $\SL_n (F)$, and every inner form looks like this.
It was shown in \cite{ABPS1} that for every inertial equivalence class $\mf s$, 
$\mc H (G)^{\mf s}$ is Morita equivalent with an algebra which is closely related
to affine Hecke algebras of type $\GL_e$ (yet is of a more general kind). It is not
known whether there exists a $\mf s$-type for every $\mf s$, but in any case the
constructions in \cite{ABPS1} are derived from the work of S\'echerre and Stevens
on inner forms of $\GL_n (F)$, so types are not far away. 
Condition \ref{cond:Morita}.i is \cite[Theorem 1.5.b]{ABPS6}, the maps $\lambda_{M M'}$
are simply inclusions, and Condition \ref{cond:Morita}.ii follows from that by
the uniqueness of adjoints.

Condition \ref{cond:Hecke} does not hold precisely for the algebras $\mc H^M$
obtained in this setting (in fact the Plancherel isomorphism for these $\mc H^M$
has not been worked out), so we cannot apply Proposition \ref{prop:4.1} or 
Theorem \ref{thm:4.2}. Nevertheless the conclusions of these results hold, see 
\cite{ABPS6}. \\

Let us summarize the conclusions from this section.
\begin{cor}\label{cor:5.2}
Let $\fs = [L,\sigma]_G$ be an inertial equivalence class of the kind discussed
in this paragraph (principal series of split group, level zero, inner form of
$\GL_n (F)$ or $\SL_n (F)$). Then $\mc S (G)^\fs$ is Morita equivalent to the
Schwartz completion of an extended affine Hecke algebra and $C_r^* (G)^\fs$ is Morita
equivalent to the $C^*$-completion of the same extended affine Hecke algebra. 
\end{cor}
\begin{proof}
Except for the last case, this follows by applying Theorem \ref{thm:4.2}. We just
checked that all its assumptions are fulfilled. For the inner forms of $\SL_n (F)$,
\cite[Theorem 6.4]{ABPS6} gives the result in the case of Schwartz algebras. Like
the proof of Theorem \ref{thm:4.2}, the method in \cite[\S 6.2]{ABPS6} also works
for the $C^*$-algebras, with minor modifications.
\end{proof}
\vspace{3mm}

\section{Hecke algebras from Bernstein's progenerators}
\label{sec:progen}

We return to the notations from Sections \ref{sec:group} and \ref{sec:comparison}.
Let $\fs = [L,\sigma]_G$ be any inertial equivalence class for $G$. Bernstein 
\cite[\S 3]{BeRu} constructed a projective generator $\Pi_\fs$ for the category $\Rep (G)^\fs$. 
By \cite[VI.10.1]{Ren}, for any Levi subgroup $M \subset G$ containing $L$:
\[
\Pi_{\fs_M} = I_{P_L \cap M}^M (\Pi_{\fs_L}) ,
\]
and this is a progenerator of $\Rep (M)^{\fs_M}$. In other words, the map
\[
\Phi_M : V \mapsto \Hom_M ( I_{P_L \cap M}^M (\Pi_{\fs_L}), V)
\]
is an equivalence between $\Rep (M)^{\fs_M}$ and the category of right modules of\\
$\End_M (I_{P_L \cap M}^M \Pi_{\fs_L} )$. For $P_L \subset P = M U_P \subset G$ we put
\[
\mc H^M = \End_M (I_{P_L \cap M}^M \Pi_{\fs_L} )^\op = \End_M (\Pi_{\fs_M} )^\op . 
\]
Then $\Phi_M$ provides an equivalence of categories $\Rep (M)^{\fs_M} \to \Mod (\mc H^M)$.

\begin{lem}\label{lem:6.1}
In the above setting Condition \ref{cond:Morita} is fulfilled. 
\end{lem}
\begin{proof}
The functoriality of normalized parabolic induction gives natural injections
\[
\lambda_{M M'} : \mc H^M \to \mc H^{M'} \quad \text{for} \quad P \subset P' \subset G. 
\]
By naturality the $\lambda_{M M'}$ satisfy Condition \ref{cond:Morita}.iii. 
By Bernstein's second adjointness theorem, for $V' \in \Rep (M')^{\fs_{M'}}$:
\begin{align*}
\Phi_M (J^{M'}_{\overline P \cap M'} V') \; & = \Hom_M ( I_{P_L \cap M}^M \Pi_{\fs_L},
J^{M'}_{\overline P \cap M'} V' ) \\
& \cong \Hom_{M'} ( I_{P \cap M'}^{M'} I_{P_L \cap M}^M \Pi_{\fs_L}, V') \\
& \cong \Hom_{M'} ( I_{P_L \cap M'}^{M'} \Pi_{\fs_L}, V') \; = \; \Phi_{M'} (V')
\end{align*}
as $\mc H^M$-modules (via $\lambda_{M M'}$). This establishes the first commutative
diagram in Condition \ref{cond:Morita}. As in the proof of Lemma \ref{lem:T}, the
second commutative diagram follows from that by invoking the uniqueness of left adjoints.
\end{proof}

In the remainder of this section we assume that $G$ is:
\begin{itemize}
\item either a symplectic group, not necessarily split, 
\item or a special orthogonal group, not necessarily split, 
\item or an inner form of $\GL_n (F)$. 
\end{itemize}
Besides the discussion of inner forms of $\GL_n (F)$ in the previous section, we point
out that types for Bernstein components of symplectic or special orthogonal groups
have been constructed in \cite{MiSt}. However, as far as we know the Hecke algebras
associated to these types are in only few cases known explicitly.

For the groups listed above, Heiermann has subjected $(\mc H^G )^\op = 
\End_G (I_{P_L}^G \Pi_{\fs_L} )$ to a deep study. In \cite{Hei} he proved that it is an 
extended affine Hecke aĺgebra with positive parameters. The constructions in \cite[\S 5]{Hei} 
are such that every $\End_M (I_{P_L \cap M}^M \Pi_{\fs_L} )$ arises as a parabolic subalgebra. 
For Condition \ref{cond:Hecke}.iii see \cite[\S 3]{Hei2}. It served as a step towards
Theorem \ref{thm:4.7} for these groups \cite[Th\'eor\`eme 5]{Hei2}.

By \cite[Proposition 1.15]{Hei}
the groups $W_Q \Gamma_Q$ are always contained in $W(\tilde R_Q)$ where $\tilde R_Q \subset
\Q R_Q$ is a larger root system. In view of Remark \ref{rem:Gamma}, Condition 
\ref{cond:Hecke}.iv holds. 

In fact, a more precise description of the root data and the groups $\Gamma_M$ is available. 
By \cite[1.13]{Hei} the root datum underlying the affine Hecke algebra
$\End_G (I_{P_L}^G \Pi_{\fs_L} )$ is a tensor product of root data of four types: $\GL_n$,
Sp$_{2n}$, SO$_{2n+1}$ and SO$_{2n}$. The groups $\Gamma_M$ are described in \cite[1.15]{Hei},
but unfortunately some elements were overlooked, for the complete picture we refer to \cite{Gol}. 
The only nontrivial $\Gamma_M$ come from the type $D$ factors, it can happen that for a root
datum of type $(\mr{SO}_{2n})^e$ we have (extended) Weyl groups
\begin{equation}\label{eq:6.1}
W_M \cong W(D_n)^e ,\qquad W_M \Gamma_M \cong W(D_{ne}) \cap W(B_n)^e .
\end{equation}
Then $|\Gamma_M| = 2^{e-1}$. In the above setting, Theorem \ref{thm:4.2} says:

\begin{thm}\label{thm:6.2}
Let $G$ be a symplectic group or a special orthogonal group over $F$ (not necessarily split),
or an inner form of $GL_n (F)$. Let $\fs$ be any inertial equivalence class for $G$. 

Then $\mc S (G)^\fs$ is Morita equivalent with the Schwartz completion of an extended affine
Hecke algebra. The underlying root datum is a tensor product of root data of type 
$GL_n$, $Sp_{2n}$, $SO_{2n+1}$ and $SO_{2n}$, and the group $\Gamma$ is a direct product
of groups $\Gamma_M$ as in \eqref{eq:6.1}. Furthermore $C_r^* (G)^\fs$ is Morita equivalent 
with the $C^*$-completion of that extended affine Hecke algebra.
\end{thm}

Theorem \ref{thm:6.2} was one of the motivations for the author to write a paper about
the K-theory of $C^*$-completions of (extended) affine Hecke algebras \cite{SolK}. 
It enables us to show that the K-groups of the reduced $C^*$-algebras of the above
groups are torsion-free.

\begin{thm}\label{thm:6.3}
Let $G$ be as in Theorem \ref{thm:6.2}. Then $K_* (C_r^* (G))$ is a free abelian group.
It is countably infinite (unless $G=1$).
\end{thm}
\begin{proof}
Recall the Bernstein decomposition from \eqref{eq:3.2}: 
\[
C_r^* (G) \cong \prod\nolimits_{\fs \in \mf B (G)} C_r^* (G)^\fs .
\]
Since topological K-theory is a continuous functor on the category of Banach algebras,
it commutes with direct sums. This reduces the theorem to one factor $C_r^* (G)^\fs$.
By Morita invariance and Theorem \ref{thm:6.2}, it suffices to show that the K-theory of
the $C^*$-completion of an extended affine Hecke algebra as in Theorem \ref{thm:6.2}
is a finitely generated free abelian group. It was checked in \cite[(62)]{SolK} that
the K\"unneth theorem for topological K-theory \cite{Scho} applies to such algebras.
Thus we only need to prove the result when the underlying root datum is of type
$\GL_n$, Sp$_{2n}$ or SO$_{2n+1}$ and $\Gamma$ is trivial, and when the root datum
is of type $(\mr{SO}_{2n})^e$ and $\Gamma$ is as in \eqref{eq:6.1}. 
These K-groups were computed in \cite{SolK}, see respectively Theorem 3.2, Theorem 3.3, 
(128), and Proposition 3.5. They are free abelian and have finite rank.
\end{proof}


\begin{thebibliography}{99}

\bibitem[ABPS1]{ABPS1} A.-M. Aubert, P.F. Baum, R.J. Plymen, M. Solleveld,
``Hecke algebras for inner forms of $p$-adic special linear groups",
J. Inst. Math. Jussieu {\bf 16.2} (2017), 351--419.

\bibitem[ABPS2]{ABPS6} A.-M. Aubert, P.F. Baum, R.J. Plymen, M. Solleveld,
``The noncommutative geometry of inner forms of $p$-adic special linear groups'',
arXiv:1505.04361, 2015.

\bibitem[BaCi]{BaCi} D. Barbasch, D. Ciubotaru,
``Unitary equivalences for reductive $p$-adic groups",
Amer. J. Math. {\bf 135.6} (2013), 1633--1674.

\bibitem[BaMo]{BaMo} D. Barbasch, A. Moy,
``A unitarity criterion for $p$-adic groups"
Inv. Math. {\bf 98} (1989), 19--37.

\bibitem[BeDe]{BeDe} J. Bernstein, P. Deligne,	
``Le ``centre" de Bernstein",
pp. 1--32 in: \emph{Repr\'esentations des groupes r\'eductifs sur un corps local}, 
Travaux en cours, Hermann, Paris, 1984.

\bibitem[BeRu]{BeRu} J. Bernstein, K.E. Rumelhart,
``Representations of p-adic groups", draft, 1993.

\bibitem[Bor]{Bor} A. Borel,
``Admissible representations of a semi-simple group over a local field
with vectors fixed under an Iwahori subgroup",
Inv. Math. {\bf 35} (1976), 233--259.

\bibitem[BHK]{BHK} C.J. Bushnell, G. Henniart, P.C. Kutzko,
``Types and explicit Plancherel formulae for reductive $p$-adic groups",
pp. 55--80 in: \emph{On certain L-functions},
Clay Math. Proc. {\bf 13}, American Mathematical Society, 2011.

\bibitem[BuKu]{BuKu} C.J. Bushnell, P.C. Kutzko,  
``Smooth representations of reductive $p$-adic groups: structure theory via types",
Proc. London Math. Soc. {\bf 77.3} (1998), 582--634.

\bibitem[Cas]{Cas} W. Casselman,
``Introduction to the theory of admissible representations of $p$-adic reductive groups",
preprint, 1995.

\bibitem[Ciu]{Ciu} D. Ciubotaru,
``Types and unitary representations of reductive $p$-adic groups",
arXiv:1707.06506, 2017.

\bibitem[DeOp]{DeOp1} P. Delorme, E.M. Opdam,
``The Schwartz algebra of an affine Hecke algebra",
J. reine angew. Math. \textbf{625} (2008), 59--114. 

\bibitem[Gol]{Gol} D. Goldberg,
``Reducibility of induced representations for Sp(2N) and SO(N)",
Amer. J. Math. {\bf 116.5} (1994), 1101--1151.

\bibitem[HC]{HC} Harish-Chandra,	
``The Plancherel formula for reductive $p$-adic groups",
Collected papers Vol. IV, 353--367,
Springer-Verlag, New York, 1984.

\bibitem[Hei1]{Hei} V. Heiermann, 
``Op\'erateurs d'entrelacement et alg\`ebres de Hecke avec param\`etres 
d'un groupe r\'eductif $p$-adique - le cas des groupes classiques",
Selecta Math. {\bf 17.3} (2011), 713--756.

\bibitem[Hei2]{Hei2} V. Heiermann, 
``Alg\`ebres de Hecke avec param\`etres et repr\'esentations d'un
groupe $p$-adique classique: pr\'eservation du spectre temp\'er\'e",
J. Algebra {\bf 371} (2012), 596--608.

\bibitem[IwMa]{IwMa} N. Iwahori, H. Matsumoto,    
``On some Bruhat decomposition and the structure
of the Hecke rings of the $p$-adic Chevalley groups",
Inst. Hautes \'Etudes Sci. Publ. Math {\bf 25} (1965), 5--48.

\bibitem[Kaz]{Kaz} D. Kazhdan, 
``Cuspidal geometry of $p$-adic groups",
J. Analyse Math. {\bf 47} (1986), 1--36.

\bibitem[KaLu]{KaLu} D. Kazhdan, G. Lusztig,      
``Proof of the Deligne--Langlands conjecture for Hecke algebras'',
Invent. Math. {\bf 87} (1987), 153--215.

\bibitem[Lus1]{Lus-Gr} G. Lusztig,
``Affine Hecke algebras and their graded version'',
J. Amer. Math. Soc. {\bf 2.3} (1989), 599--635.

\bibitem[Lus2]{Lus3} G. Lusztig,
``Classification of unipotent representations of simple $p$-adic groups'',
Int. Math. Res. Notices {\bf 11} (1995), 517--589.

\bibitem[MiSt]{MiSt} M. Miyauchi, S. Stevens,
``Semisimple types for $p$-adic classical groups'',
Math. Ann. {\bf 358} (2014), 257--288.

\bibitem[Mor1]{Mor} L. Morris,
``Tamely ramified intertwining algebras'',
Invent. Math. {\bf 114.1} (1993), 1--54.

\bibitem[Mor2]{Mor2} L. Morris,
``Level zero G-types'',
Compositio Math. {\bf 118.2} (1999), 135--157.

\bibitem[Opd]{Opd-Sp} E.M. Opdam,
``On the spectral decomposition of affine Hecke algebras'',
J. Inst. Math. Jussieu \textbf{3.4} (2004), 531--648.

\bibitem[Ply]{Ply1} R.J. Plymen,
``Reduced $C^*$-algebra for reductive $p$-adic groups",
J. Funct. Anal. {\bf 88.2} (1990), 251--266.

\bibitem[RaRa]{RaRa} A. Ram, J. Ramagge,
``Affine Hecke algebras, cyclotomic Hecke algebras and Clifford theory'',
pp. 428--466 in: \emph{A tribute to C.S. Seshadri (Chennai 2002)},
Trends in Mathematics, Birkh\"auser, 2003.

\bibitem[Ren]{Ren} D. Renard,
\emph{Repr\'esentations des groupes r\'eductifs $p$-adiques},
Cours sp\'ecialis\'es {\bf 17}, Soci\'et\'e Math\'ematique de France, 2010.

\bibitem[Roc1]{Roc1} A. Roche,
``Types and Hecke algebras for principal series representations of split reductive $p$-adic groups",
Ann. Sci. \'Ecole Norm. Sup. {\bf 31.3} (1998), 361--413.

\bibitem[Roc2]{Roc2} A. Roche, 
``Parabolic induction and the Bernstein decomposition",
Comp. Math. {\bf 134} (2002), 113--133.

\bibitem[Sau]{Sau} F. Sauvageot,
``Principe de densit\'e pour les groupes r\'eductifs",
Comp. Math. {\bf 108} (1997), 151--184.

\bibitem[SSZ]{SSZ} P. Schneider, E.-W. Zink,
``K-types for the tempered components of a $p$-adic general linear group. 
With an appendix by Schneider and U. Stuhler",
J. reine angew. Math. {\bf 517} (1999), 161--208.

\bibitem[ScZi]{SZ} P. Schneider, E.-W. Zink,
``The algebraic theory of tempered representations of $p$-adic groups. Part II: Projective generators",
Geom. Func. Anal. {\bf 17.6} (2008), 2018--2065.

\bibitem[Sch]{Scho} C. Schochet,
``Topological methods for $C^*$-algebras. II. Geometry resolutions and the K\"unneth formula",
Pacific J. Math. {\bf 98.2} (1982), 443--458.

\bibitem[S\'ec]{Sec} V. S\'echerre,
``Repr\'esentations lisses de $GL_m (D)$ III: types simples",
Ann. Scient. \'Ec. Norm. Sup. {\bf 38} (2005), 951--977.

\bibitem[S\'eSt]{SeSt} V. S\'echerre, S. Stevens,
``Smooth representations of $GL(m,D)$ VI: semisimple types",
Int. Math. Res. Notices {\bf 13} (2012), 2994--3039.

\bibitem[Sol1]{SolPadicHP} M. Solleveld,
``Periodic cyclic homology of reductive $p$-adic groups",
J. Noncommutative Geometry {\bf 3.4} (2009), 501--558.

\bibitem[Sol2]{SolGHA} M. Solleveld,
``Parabolically induced representations of graded Hecke algebras",
Algebras and Representation Theory {\bf 15.2} (2012), 233--271.

\bibitem[Sol3]{SolAHA} M. Solleveld,
``On the classification of irreducible representations of
affine Hecke algebras with unequal parameters",
Representation Theory {\bf 16} (2012), 1--87.

\bibitem[Sol4]{SolHomAHA} M. Solleveld,
``Hochschild homology of affine Hecke algebras",
J. Algebra {\bf 384} (2013), 1--35.

\bibitem[Sol5]{SolK} M. Solleveld,
``Topological K-theory of affine Hecke algebras'',
arXiv:1610.06722 (2016), to appear in Annals of K-theory.

\bibitem[Vig]{Vig} M.-F. Vign\'eras,
``On formal dimensions for reductive $p$-adic groups",
pp. 225--266 in: \emph{Festschrift in honor of I.I. Piatetski-Shapiro 
on the occasion of his sixtieth birthday, Part I}, 
Israel Math. Conf. Proc. {\bf 2}, Weizmann, Jerusalem, 1990.

\bibitem[Wal]{Wal} J.-L. Waldspurger,       
``La formule de Plancherel pour les groupes $p$-adiques (d'apr\`es Harish-Chandra)",
J. Inst. Math. Jussieu {\bf 2.2} (2003), 235--333.

\end{thebibliography}
\end{document}